\documentclass[11pt]{amsart}
\usepackage{amsmath,amssymb,amsthm}
\usepackage[margin=1in]{geometry}
\usepackage{tikz}

\newtheorem{theorem}{Theorem}[section]
\newtheorem{proposition}[theorem]{Proposition}
\newtheorem{lemma}[theorem]{Lemma}
\newtheorem{corollary}[theorem]{Corollary}
\theoremstyle{definition}
\newtheorem{definition}[theorem]{Definition}
\theoremstyle{remark}
\newtheorem{remark}[theorem]{Remark}
\newtheorem{example}[theorem]{Example}
\theoremstyle{plain}
\newtheorem{introtheorem}{Theorem}

\DeclareMathOperator{\sign}{sign}
\DeclareMathOperator{\rank}{rank}
\DeclareMathOperator{\ind}{ind}

\newcommand{\authorinfo}{%
  \author{Fr\'ed\'eric Bihan}
  \address{Laboratoire de Math\'ematiques, Universit\'e Savoie Mont Blanc, 73376 Le Bourget-du-Lac Cedex, France}
  \email{frederic.bihan@univ-smb.fr}}

\title{Descartes' rule of signs for polynomial systems}
\authorinfo
\date{\today}

\begin{document}

\begin{abstract}
The classical Descartes' rule of signs bounds the number of positive roots of a real polynomial in one variable by the number of sign changes of its coefficients. This bound only depends on the order of the exponents and on the signs of the coefficients, that is, on the oriented matroids of the exponent and coefficient matrices. We generalize this rule to fewnomial systems with any number of monomials: for systems of $n$ real polynomial equations in $n$ variables with $n+k+1$ monomials, we obtain upper bounds for the number of nondegenerate positive solutions which only depend on sign conditions on the maximal minors of the coefficient matrix and of the exponent matrix and its liftings. For $k=1$ they are the Descartes' rules of signs for circuits of Bihan--Dickenstein and Bihan--Dickenstein--Forsgård. Our approach starts from the Gale duality and the Khovanskii--Rolle argument of Bihan and Sottile, which reduce the problem to counting common zeros of iterated Jacobians in a chamber $\Delta$ of a hyperplane arrangement determined by the coefficients. We bound this number by the Bézout number minus the zeros forced in the other chambers by a topological degree, which we compute by a facet-count formula involving signs of minors of a Gale dual matrix of the exponent matrix. This leads to a recursion over the flats of the arrangement, which ends with the Descartes' rules for circuits. Restricting the recursion to coherent chambers, in which no cancellation occurs in the facet-count formula, gives weaker bounds in closed form.
\end{abstract}
\maketitle

\section*{Introduction}

How many real solutions can a system of real polynomial equations have? Bézout's theorem bounds the number of complex solutions by the product of the degrees of the equations, and this bound is often far from the number of real solutions. Khovanskii's theory of fewnomials \cite{Kh} shows that the number of real solutions with positive coordinates can be bounded in terms of the number of monomials only, independently of the degrees. In this paper we look for bounds of the same kind which, in addition, are determined by sign conditions on the coefficients and the exponents. This is what the classical Descartes' rule of signs, stated by René Descartes in \emph{La Géométrie} (1637), does for polynomials in one variable, where these signs are the signs of the coefficients and the order of the exponents. Descartes' rule of signs is often regarded as one of the foundational results of real algebraic geometry. In this paper we propose a multivariate generalization of Descartes' rule of signs to square polynomial systems, namely systems of $n$ polynomial equations in $n$ variables, and with any number of monomials.

\subsection*{The classical Descartes' rule of signs}

Let $f(x)=c_1x^{a_1}+c_2x^{a_2}+\dots+c_Nx^{a_N}$ be a real polynomial in one variable, with nonzero coefficients and exponents $a_1<a_2<\dots<a_N$; the exponents may be arbitrary real numbers, in which case $f$ is defined for $x>0$. The classical Descartes' rule of signs states that the number of positive roots of $f$, counted with multiplicities, is at most the number of sign changes in the sequence of coefficients $(c_1,c_2,\dots,c_N)$, and that the difference is even. For instance, $x^3-x+1$ has at most two positive roots, since the signs $(+,-,+)$ change twice. The bound is sharp: for every sequence of signs, there are polynomials with these signs whose number of positive roots is the number of sign changes. Thus the number of positive roots is bounded by a quantity which only depends on two pieces of information: the order of the exponents, and the signs of the coefficients.

\subsection*{Plücker coordinates and sign conditions}

Consider now a system of $n$ equations in $n$ variables with real coefficients, involving $N=n+k+1$ distinct monomials $x^{a_1},\dots,x^{a_N}$ with real exponents $a_j\in\mathbb{R}^n$:
\[
\sum_{j=1}^Nc_{ij}\,x^{a_j}=0,\qquad i=1,\dots,n,\qquad x\in\mathbb{R}^n_{>0}.
\]
It is encoded by two matrices: the $n\times N$ coefficient matrix $C=(c_{ij})$, and the $(n+1)\times N$ exponent matrix $A$ whose $j$-th column is $(1,a_j)$. Two simple operations do not change the number of positive solutions. Replacing the equations by independent linear combinations of them replaces $C$ by $GC$ with $G$ invertible. An invertible affine change of the exponents, $a_j\mapsto La_j+v$, corresponds to the monomial change of variables $x\mapsto x^L$, which is a bijection of the positive orthant, followed by the division of all the equations by $x^v$; it replaces $A$ by a matrix with the same row space. Hence the positive solutions only depend on the row spaces of $C$ and $A$, that is, on two points of Grassmannians. These points are described by their Plücker coordinates, the maximal minors of $C$ and $A$, up to a common factor. The signs of these minors, up to a global sign, form the \emph{oriented matroids} of $C$ and $A$.

Descartes' rule fits into this framework. For $n=1$, the exponent matrix has rows $(1,\dots,1)$ and $(a_1,\dots,a_N)$, and its maximal minors are the differences $a_j-a_i$: the oriented matroid of $A$ records the order of the exponents. The coefficient matrix is the row $(c_1,\dots,c_N)$, whose maximal minors are the coefficients themselves: the oriented matroid of $C$ records their signs. So Descartes' bound is determined by the oriented matroids of $A$ and $C$.

For systems with any $n$ and $k$, a natural problem is to find upper bounds for the number of positive solutions which are determined by \emph{sign conditions on the maximal minors of $C$ and $A$ and on simple functions of them}. Bounds depending only on the oriented matroids, that is, on the signs of the minors themselves, are the simplest instance, but there is no reason to expect the best bounds to be of this form. Indeed, the number of positive solutions is locally constant outside a discriminant locus, and by the Tarski--Seidenberg theorem it is determined by the signs of finitely many polynomials in the maximal minors of $C$ and $A$; the problem is to find sign conditions which are simple and give good bounds.

\subsection*{Descartes' rule of signs for circuits}

The general bounds of fewnomial theory only depend on $n$ and $k$. Bihan and Sottile \cite{BS} reduced the problem, by Gale duality, to a system of $k$ equations $\prod_{i}p_i(y)^{B_{ij}}=1$ in a polyhedron $\Delta\subset\mathbb{R}^k$ cut out by $n+k$ affine functions $p_i$, and obtained the bound $\frac{e^2+3}4\,2^{\binom k2}n^k$ by a Khovanskii--Rolle argument. Here $B$ and the coefficients of the $p_i$ come from Gale dual matrices of $A$ and $C$, whose maximal minors are, up to signs and a common factor, the complementary maximal minors of $A$ and $C$.

Bihan and Dickenstein \cite{BD} gave a Descartes' rule of signs for systems supported on circuits, that is, for $k=1$ ($n+2$ monomials). By Gale duality such a system reduces to a single equation on an interval cut out by $n+1$ points, and they bound the number of positive solutions by the number of sign variations of a sequence built from signed maximal minors of $A$ and $C$. This bound only depends on the oriented matroids of $A$ and $C$, it is sharp, and it is at most $n+1$. For $n=1$ and three monomials, it is Descartes' rule.

Bihan, Dickenstein and Forsgård \cite{BDF} then obtained an optimal Descartes' rule of signs for circuits, in which the signs of this sequence are replaced by the signs of its partial sums. These partial sums are linear functions of the maximal minors, so this rule is determined by sign conditions of the kind above which go beyond the oriented matroids, and it is optimal. The aim of this paper is to extend these Descartes' rules from circuits to systems with any number $n+k+1$ of monomials. For $k=1$ our bounds are those of \cite{BD} and \cite{BDF}, and in general the circuit rules are the first step of a recursion over the flats of a hyperplane arrangement. Our bounds are determined by sign conditions of the same kind: the first one by the oriented matroids of $C$ and $A$ and of some liftings of $A$, the second one also by signs of partial sums of minors as in \cite{BDF} (Theorem~\ref{thm:main} below).

\subsection*{Strategy}

We follow the approach of Bihan and Sottile \cite{BS}, keeping track of the signs of the data at each step. The Gale system is encoded by Gale dual matrices $B$ of $A$ and $D$ of $C$ (Sections~\ref{sec:polysys} and~\ref{sec:gale}), and the starting point is the Bihan--Sottile bound (Theorem~\ref{thm:KR-general}):
\[
\#\{\text{positive solutions}\}\ \le\ \#\big(\{\Gamma_1=\dots=\Gamma_k=0\}\cap\Delta\big)+\sum_{j=1}^kb(C_j),
\]
where $\Gamma_k,\dots,\Gamma_1$ are iterated Jacobians, which are polynomials, and $b(C_j)$ is the number of noncompact components of a curve $C_j\subset\Delta$. It holds under a transversality assumption, which we avoid by applying the argument to small perturbations of the Jacobians, as in \cite{BS2} (Lemma~\ref{lem:perturbed-jacobians}).

Consider first the number of zeros of the Jacobians in $\Delta$. Bihan and Sottile bound it by the Bézout number $\beta_k=n^k2^{\binom k2}$, which bounds the total number of nondegenerate zeros of $\Gamma=(\Gamma_1,\dots,\Gamma_k)$ in $\mathbb{R}^k$. But $\Delta$ is only one of the chambers of the hyperplane arrangement $\bigcup_i\{p_i=0\}$. Our first observation is that topology forces zeros in the other chambers $\mathcal C$: when they are nondegenerate, their number is at least $|d(\mathcal C)|$, where $d(\mathcal C)$, the \emph{boundary degree} of $\mathcal C$, is the degree of the map $\Gamma/|\Gamma|$ from the boundary of $\mathcal C$ to the sphere $S^{k-1}$. Indeed, by the Kronecker--Poincaré--Hopf theorem, $d(\mathcal C)$ is the number of zeros of $\Gamma$ in $\mathcal C$ counted with the signs of the Jacobian determinant of $\Gamma$ (equation~\eqref{eq:KPH}). Subtracting these forced zeros from $\beta_k$ bounds the number of zeros in $\Delta$. Our second observation is that $d(\mathcal C)$ is easy to compute: it is a signed count of the zeros of the Jacobians on the facets of $\mathcal C$, with signs given by the exponents (the facet-count formula below).

Consider now the numbers $b(C_j)$. The ends of the curves $C_j$ are zeros of $\Gamma_{j+1},\dots,\Gamma_k$ on faces of $\Delta$ of codimension $j$ satisfying a sign condition, so $2b(C_j)$ is also bounded by numbers of zeros of Jacobians on flats of the arrangement.

Finally, restricted to a flat of the arrangement, the Jacobians are again the Jacobians of a system of the same type, with a smaller $k$. Both terms therefore reduce to counting zeros of Jacobians in the chambers of the flats, and the boundary degrees of these chambers are given by numbers of zeros on flats of one dimension less. This leads to a recursion on the dimension of the flats, which ends on lines, where the Descartes' rule of signs for circuits applies. The data entering the recursion are sign conditions: the oriented matroid of $C$, the signs of the maximal minors of $A$ and of some liftings of $A$ (Theorem~\ref{thm:main} below), and, for the sharper bounds, the signs of partial sums of maximal minors of $A$ on the lines, as in \cite{BDF}. This is why the final bounds are Descartes-type bounds. The case $k=1$, at the end of this introduction, illustrates the method in its simplest form.

\subsection*{The main steps}

The reader may skip this subsection at first reading and go directly to the main result.

\medskip\noindent\emph{The facet-count formula} (Section~\ref{sec:boundary-degree}). For a chamber $\mathcal C$ of $\mathbb{R}^k$, only the absolute value of $d(\mathcal C)$ matters, and
\[
|d(\mathcal C)|=\frac12\,\Big|\sum_{i=1}^{n+k+1}\sign(B_{i1})\,N_i\Big|,
\]
where $N_i$ is the number of zeros of $(\Gamma_2,\dots,\Gamma_k)$ on the facet of $\mathcal C$ on $\{p_i=0\}$, and the index $n+k+1$ stands for the facet at infinity (Theorems~\ref{thm:facet-count} and~\ref{thm:facet-count-unbounded}). The key point is Lemma~\ref{lem:sign-identity}: the sign of the contribution of a zero depends only on the facet containing it.

The same formula holds for every chamber of every flat of the arrangement (see below), so $|d(\mathcal C)|$ is bounded from below in terms of bounds for numbers of zeros on flats of one dimension less.

\medskip\noindent\emph{Restriction to flats} (Sections~\ref{sec:mu}, \ref{sec:restriction} and~\ref{sec:reduced}). The numbers $N_i$ above, and the ends of the curves $C_j$, are zeros of Jacobians on flats $L_S=\bigcap_{s\in S}\{p_s=0\}$ of the arrangement, possibly at infinity; we write $\Lambda_S\subset\mathbb{RP}^k$ for the corresponding projective subspace, which covers both cases. We first determine which flats are reached: the closure of the level set $\{f_1=\dots=f_{j-1}=1\}$ meets the arrangement, for generic data, exactly along the flats $\Lambda_S$ for which the vectors $(B_{s1},\dots,B_{s,j-1})$, $s\in S$, positively span $\mathbb{R}^{j-1}$ (Theorem~\ref{thm:mu-boundary} and Corollary~\ref{cor:mu-generic}); this is the origin of the sign conditions defining the sets $\mathcal S_j(\Delta)$ below.

We then show that, on a flat $L_S$ of codimension $j$, the Jacobians $\Gamma_{j+1},\dots,\Gamma_k$ are, up to nonzero constants, the Jacobians of an explicit \emph{reduced system} $(\tilde D,\tilde B)$ (Definition~\ref{def:reduced-system}, Theorems~\ref{thm:restriction-LS} and~\ref{thm:restriction-LS-infinity}). Its affine functions are the restrictions of the $p_i$ to $L_S$, written in coordinates whose coefficients are ratios of \emph{maximal minors of $D$}, and its exponents are the relations of the support which vanish on the points $a_s$, $s\in S$. The identities are exact on finite flats and hold modulo the later Jacobians at infinity. Since $(\tilde D,\tilde B)$ is again a system of the same type, with $n$ unchanged and $k$ replaced by $k-j$, everything applies again on $L_S$. In particular, if $\mathcal C$ is a chamber of the flat $L_S$ of codimension $j=|S|\le k-1$, then
\[
|\tilde d(\mathcal C)|=\frac12\,\Big|\sum_{i\notin S}\sign\big(\det B_{S\cup\{i\},\{1,\dots,j+1\}}\big)\,\tilde N_i\Big|,
\]
where the rows of $B_{S\cup\{i\},\{1,\dots,j+1\}}$ are those of $S$ in increasing order followed by the row $i$, $\tilde N_i$ is the number of zeros of the reduced Jacobians $(\tilde\Gamma_{j+2},\dots,\tilde\Gamma_k)$ on the facet of $\mathcal C$ on $\{p_i=0\}$, and $\tilde d(\mathcal C)$ is the boundary degree of $(\tilde\Gamma_{j+1},\dots,\tilde\Gamma_k)$ (Theorem~\ref{thm:facet-count-LS}). For $S=\emptyset$ this is the formula above. The number of zeros in $\mathcal C$ is at least $|\tilde d(\mathcal C)|$ (Proposition~\ref{prop:lower-bound-LS}).

On a line, the reduced system has a single Jacobian, a polynomial of degree exactly $n$, and its numbers of zeros on the $n+2$ arcs of the line satisfy the Descartes' rule of signs for circuits of \cite{BD}: they have a prescribed parity and are bounded in terms of the cyclic sign variation of the minors $\det B_{S\cup\{l\},\{1,\dots,k\}}$ (Proposition~\ref{prop:line-zeros}). This is where the recursion starts.

\medskip\noindent\emph{Recursive bounds} (Section~\ref{sec:recursive}). For every chamber $\mathcal C$ of every flat $\Lambda_S$ we define integers $\ell_S(\mathcal C)\le Z(S,\mathcal C)\le u_S(\mathcal C)$ bounding the number $Z(S,\mathcal C)$ of zeros of the reduced Jacobians in $\mathcal C$ (Theorem~\ref{thm:recursive-bounds}). Starting the recursion on lines with the optimal Descartes' rule of \cite{BDF} instead of the rule of \cite{BD} gives sharper bounds $\ell_S(\mathcal C)\le\ell^*_S(\mathcal C)\le Z(S,\mathcal C)\le u^*_S(\mathcal C)\le u_S(\mathcal C)$ (Section~\ref{sec:BDF}, Theorem~\ref{thm:recursive-bounds-BDF}).

\subsection*{The main result}

For a subset $S$ of the indices with $|S|=j$, let $\Delta_S$ be the face of $\overline\Delta$ of codimension $j$ lying on the flat $\Lambda_S$, when it exists, and let $\mathcal S_j(\Delta)$ be the set of such $S$ for which the vectors $(B_{s1},\dots,B_{s,j-1})$, $s\in S$, positively span $\mathbb{R}^{j-1}$ (Definition~\ref{def:boundary-faces}). These are the faces that the closure of the curve $C_j$ can reach: since $C_j$ lies in the level set $\{f_1=\dots=f_{j-1}=1\}$, its closure can only meet faces $\Delta_S$ satisfying this sign condition. Under the genericity assumptions below, and after a small perturbation of the Jacobians (see below), each end of a noncompact component of $C_j$ converges to a point of a face $\Delta_S$ with $S\in\mathcal S_j(\Delta)$, this point is one of the $Z(S,\Delta_S)$ zeros in $\Delta_S$ of the reduced Jacobians, and distinct ends converge to distinct points. Since each noncompact component has two ends,
\[
2\,b(C_j)\le\sum_{S\in\mathcal S_j(\Delta)}Z(S,\Delta_S)
\]
(Proposition~\ref{prop:ends}), which explains the factor $\frac12$ below.

These sets and the recursive bounds depend on the choice of the basis $b_1,\dots,b_k$ of relations of the support forming the columns of $B$, through the minors of $B$ built on its first columns. These minors are maximal minors of lifted exponent matrices: for $1\le j\le k-1$, let $A_j$ be the matrix obtained from $A$ by adding $k-j$ rows $w_{j+1},\dots,w_k$ with $\langle w_m,b_l\rangle=\delta_{ml}$. Its columns are liftings of the points of $\mathcal A$ to $\mathbb{R}^{n+k-j}$, and its kernel is spanned by $b_1,\dots,b_j$ (Remark~\ref{rem:flag-minors}). Our main result is the following (Section~\ref{sec:possol}).

\begin{introtheorem}\label{thm:main}
Assume that the system~\eqref{eq:system} satisfies the genericity assumptions (R), (U$_D$), (U$_B$) and (G). Then
\begin{align*}
\#\{\text{nondegenerate positive solutions of~\eqref{eq:system}}\}\ &\le\ u^*_\emptyset(\Delta)+\frac12\sum_{j=1}^k\ \sum_{S\in\mathcal S_j(\Delta)}u^*_S(\Delta_S)\\
&\le\ u_\emptyset(\Delta)+\frac12\sum_{j=1}^k\ \sum_{S\in\mathcal S_j(\Delta)}u_S(\Delta_S),
\end{align*}
where $u^*_S(\mathcal C)\le u_S(\mathcal C)$ are the recursive upper bounds for the number $Z(S,\mathcal C)$ of common zeros of the reduced Jacobians in a chamber $\mathcal C$ of the flat $\Lambda_S$, built from the optimal Descartes' rule of \cite{BDF} and from the Descartes' rule of \cite{BD} on the lines of the arrangement. The last bound only depends on the oriented matroid of the coefficient matrix $C$ and on the oriented matroids of the exponent matrix $A$ and of the lifted exponent matrices $A_1,\dots,A_{k-1}$. The first one is sharper, but also depends on the values of the maximal minors of $A$, through partial sums on the lines.
\end{introtheorem}

This is Theorem~\ref{thm:possol-Z}, Corollaries~\ref{cor:possol-recursive} and~\ref{cor:possol-BDF}, and Remark~\ref{rem:flag-choice}. The existence of a face $\Delta_S$ is characterized in terms of the oriented matroid of $D$, or equivalently of $C$ (Proposition~\ref{prop:faces-Delta}), and the positive spanning conditions are sign conditions on minors of $B$, which are maximal minors of liftings $A_j$ of the exponent matrix (Remark~\ref{rem:flag-minors}). So the last bound depends on signs only, but on the signs of more than the maximal minors of $A$. The liftings correspond to the choice of a flag of relations $b_1,\dots,b_{k-1}$ of the support, and the bounds can be optimized over this choice (Remark~\ref{rem:flag-choice}).

The genericity assumptions are discussed in Remark~\ref{rem:genericity}. They are mild: (R), (U$_D$) and (U$_B$) hold after a small perturbation of $C$ and of the exponents (the support may consist of arbitrary points of $\mathbb{R}^n$) and for a generic choice of the flag $b_1,\dots,b_{k-1}$; and since nondegenerate positive solutions persist under small perturbations, the bounds apply to every system which can be approximated by systems satisfying (G) with the same sign data (Remark~\ref{rem:genericity}).

\subsection*{Examples}

We illustrate the bounds of Theorem~\ref{thm:main} on two examples with $k=2$: a polynomial in one variable with four monomials and three positive roots, for which the bound is sharp for a good choice of the first relation $b_1$ and not for a bad one, and the system of El Hilany \cite{EH} with two equations, five monomials and seven positive solutions, for which our two bounds are $8$ and $10$ (Subsection~\ref{subsec:examples}). For comparison, the best previously known bound for two equations in two variables with five monomials is $15$, due to Bihan and Sottile \cite{BS}. More generally, an exhaustive enumeration of the sign data shows that, for two equations in two variables with five monomials, the second bound of Theorem~\ref{thm:main} is at most $11$ for $98.9\%$ of the combinatorial types, at most $13$ outside an explicit family $\mathcal F$ of systems with convex pentagonal support, and equal to $15$ on $\mathcal F$ (Remark~\ref{rem:n2k2}).

\subsection*{The case $k=1$}

For $k=1$, the Gale system is a single equation on an interval $\Delta$ of a line cut by $n+1$ points, $\Gamma_1$ is a polynomial of degree $n$, and Theorem~\ref{thm:main} reduces to the Descartes' rule of signs for circuits of Bihan and Dickenstein \cite{BD} (Remark~\ref{rem:possol-small-k}). It is instructive to see how. The other chambers are the other arcs $\Delta'$ of the line, and the facet-count formula gives $|d(\Delta')|=1$ exactly when the signs of the first column of $B$ at the two endpoints of $\Delta'$ agree, that is, when $\Delta'$ is \emph{coherent}, and $d(\Delta')=0$ otherwise. So the Descartes' rule for circuits is the bound
\[
\#\{\text{positive solutions}\}\ \le\ 1+n-\#\{\text{coherent arcs }\Delta'\neq\Delta\},
\]
that is, the Bézout number minus the coherent chambers different from $\Delta$, plus the contribution of the two ends of $\Delta$. In other words, \emph{the Descartes' rule of signs of \cite{BD} only uses coherent chambers}.

In higher dimension, a chamber is coherent if the signs entering its facet-count formula are the same on all its facets; then no cancellation occurs in this formula, and its boundary degree is bounded from below by half the sum of the lower bounds on its facets. Restricting the lower bounds to coherent chambers gives weaker bounds, which only use lower bounds in the recursion and can be written in closed form (Section~\ref{sec:coherent}). A \emph{coherent flag} of a chamber $\mathcal C$ of dimension $d$ is a chain of faces $\mathcal C=\mathcal C_d\supset\mathcal C_{d-1}\supset\dots\supset\mathcal C_0$, each a facet of the previous one and each coherent. If $\kappa(\mathcal C)$ denotes the number of coherent flags of $\mathcal C$, then $\mathcal C$ contains at least $\lceil\kappa(\mathcal C)/2^d\rceil$ zeros of the Jacobians, and the number of positive solutions is bounded by the Bézout numbers of the flats minus these numbers (Proposition~\ref{prop:closed-form}). This generalizes ``the Bézout number minus the coherent arcs'' to arbitrary $k$.

\subsection*{Organization of the paper}
Sections~\ref{sec:polysys} and~\ref{sec:gale} fix the notation and recall Gale duality, and Section~\ref{sec:KR} recalls the Bihan--Sottile bound. Section~\ref{sec:jacobians} studies the Jacobians $\Gamma_m$, and Section~\ref{sec:boundary-degree} proves the facet-count formula for the boundary degree. Section~\ref{sec:mu} determines the faces reached by the curves $C_j$, Section~\ref{sec:restriction} introduces the reduced systems on the flats, and Section~\ref{sec:reduced} applies the facet-count formula to them. The recursive bounds are defined in Section~\ref{sec:recursive}, and the bounds for the positive solutions, including Theorem~\ref{thm:main}, are proved in Section~\ref{sec:possol}. Section~\ref{sec:BDF} gives the sharper bounds coming from the optimal Descartes' rule for circuits, together with the examples, and Section~\ref{sec:coherent} gives the closed-form bounds from coherent chambers. The proofs of the more technical results are gathered in Appendix~\ref{app:proofs}.

\subsection*{Acknowledgements}
This work has benefited from many discussions with Frank Sottile, Alicia Dickenstein and Sébastien Tavenas, to whom the author is very grateful.

\section{Polynomial systems and their Gale dual matrices}\label{sec:polysys}

In this section we fix the notation for the polynomial systems that we study. We introduce the exponent matrix $A$ and the coefficient matrix $C$, their Gale dual matrices $B$ and $D$, and the pairs $(D,B)$ of type $(n,k)$, which are the data that the rest of the paper works with (Definition~\ref{def:setting}). We also recall the relations between the maximal minors of a matrix and those of its Gale dual matrices.

\subsection{The system}
Let $\mathcal A\subset\mathbb{R}^n$ be a finite set with $|\mathcal A|=n+k+1$ that is not contained in an affine hyperplane. We number its elements $a_1,\dots,a_{n+k+1}$. Throughout the paper we put $N=n+k+1$ and $[N]=\{1,\dots,N\}$. The \emph{exponent matrix} of $\mathcal A$ is the $(n+1)\times N$ matrix
\[
A=\begin{pmatrix}1&1&\cdots&1\\ a_1&a_2&\cdots&a_{N}\end{pmatrix},
\]
whose $j$-th column is $(1,a_j)^{T}$. Since $\mathcal A$ is not contained in an affine hyperplane, $A$ has rank $n+1$. Let $C=(c_{ij})$ be a real $n\times N$ matrix, the \emph{coefficient matrix}. With $x=(x_1,\dots,x_n)$ and $x^{a_j}=x_1^{a_{j1}}\cdots x_n^{a_{jn}}$, we consider the system
\begin{equation}\label{eq:system}
F_i(x)=\sum_{j=1}^{N}c_{ij}\,x^{a_j}=0,\qquad i=1,\dots,n .
\end{equation}
Since the exponents are real, $x^{a_j}$ is defined for $x\in\mathbb{R}^n_{>0}$. We are interested in the number of \emph{positive solutions} of~\eqref{eq:system}, that is, solutions in $\mathbb{R}^n_{>0}$. A solution $x$ is \emph{nondegenerate} if the Jacobian matrix $\big(\partial F_i/\partial x_l(x)\big)$ is invertible.

We assume throughout that
\begin{itemize}
\item[(R)] the $n\times(n+k)$ matrix $C'$ obtained from $C$ by removing its last column has rank $n$.
\end{itemize}
In particular $C$ has rank $n$, so $\dim\ker C=k+1$. Condition~(R) holds for generic coefficients. It fails exactly when some nontrivial linear combination of the equations~\eqref{eq:system} is a multiple of the single monomial $x^{a_N}$; such a system has no positive solution.

\subsection{Gale dual matrices}

\begin{definition}\label{def:gale-dual}
A \emph{Gale dual matrix} of $A$ is a real $N\times k$ matrix $B=(B_{ij})$ whose columns form a basis of $\ker A\subset\mathbb{R}^N$. A \emph{Gale dual matrix} of $C$ is a real $N\times(k+1)$ matrix $D=(D_{il})$ whose columns form a basis of $\ker C\subset\mathbb{R}^N$.
\end{definition}

The rows of $B$ and $D$ are indexed by $i\in[N]$, that is, by the elements $a_i$ of $\mathcal A$, in the chosen order. The columns of $B$ are the \emph{affine relations} of $\mathcal A$:
\[
\sum_{i=1}^{N}B_{ij}=0,\qquad \sum_{i=1}^{N}B_{ij}\,a_i=0,\qquad j=1,\dots,k .
\]
The first identity is the first row of $AB=0$. So the columns of $B$ sum to zero.

\begin{lemma}\label{lem:normalize-D}
Under~(R), there is a Gale dual matrix $D$ of $C$ whose last row is $(0,\dots,0,1)$. Two such matrices $D$ and $\hat D$ satisfy $\hat D=DH$, where $H=\begin{pmatrix}H_0&h\\0&1\end{pmatrix}$ with $H_0\in GL_k(\mathbb{R})$ and $h\in\mathbb{R}^k$.
\end{lemma}

\begin{proof}
Let $e_N$ be the last vector of the standard basis. The hyperplane $\{z_N=0\}$ contains $\ker C$ if and only if $e_N^{T}$ is in the row space of $C$, which means that $e_N^{T}=u^{T}C$ for some $u\in\mathbb{R}^n$. Such a $u$ is nonzero and satisfies $u^{T}C'=0$, which contradicts~(R). It follows that $K_0=\ker C\cap\{z_N=0\}$ has dimension $k$. Take a basis of $K_0$ as the first $k$ columns of $D$, and a vector of $\ker C$ with last coordinate $1$ as its last column. The second statement holds because $D$ and $\hat D$ have the same column space, their first $k$ columns span the same space $K_0$, and their last columns both have last coordinate $1$.
\end{proof}

From now on $B$ is a Gale dual matrix of $A$ and $D$ is a Gale dual matrix of $C$ with last row $(0,\dots,0,1)$. These are the matrices used in the rest of the paper. We set
\[
p_i(y)=\sum_{l=1}^{k}D_{il}\,y_l+D_{i,k+1},\qquad i=1,\dots,N,\qquad y=(y_1,\dots,y_k)\in\mathbb{R}^k,
\]
so that $p_N=1$ and $(p_1(y),\dots,p_N(y))^{T}=D\binom{y}{1}$. So $y\mapsto(p_1(y),\dots,p_N(y))$ is an affine isomorphism from $\mathbb{R}^k$ onto the affine space $\ker C\cap\{z_N=1\}$. A change of $D$ as in Lemma~\ref{lem:normalize-D} is the affine change of coordinates $y\mapsto H_0y+h$, and a change of $B$ is $B\mapsto BG$ with $G\in GL_k(\mathbb{R})$.

Let $D'$ be the $(n+k)\times k$ matrix obtained from $D$ by removing its last row and its last column, and write $D'_i=(D_{i1},\dots,D_{ik})$ for its $i$-th row. Let $B'$ be the $(n+k)\times k$ matrix obtained from $B$ by removing its last row.

\begin{definition}\label{def:setting}
A \emph{pair of type $(n,k)$} is a pair $(D,B)$ of real matrices of sizes $(n+k+1)\times(k+1)$ and $(n+k+1)\times k$ such that the last row of $D$ is $(0,\dots,0,1)$ and the columns of $B$ sum to zero.
\end{definition}

Gale dual matrices $D$ and $B$ as above form a pair of type $(n,k)$. The functions $p_i$, the matrices $D'$ and $B'$, and the functions $f_j$, $\psi_j$ and $\Gamma_j$ defined in Sections~\ref{sec:gale} and~\ref{sec:KR} make sense for an arbitrary pair of type $(n,k)$. From Section~\ref{sec:jacobians} on we only use this structure, and we do not assume that $D$ and $B$ are Gale dual matrices of a polynomial system; this generality is needed for the reduced systems on the flats of the arrangement (Section~\ref{sec:reduced}). Since $p_N=1$, the last row of $B$ does not enter the definitions of $f_j$ and $\psi_j$.

\subsection{Minors}
The maximal minors of a Gale dual matrix are, up to a common factor and signs, the complementary maximal minors of the original matrix. This is the classical duality between the Plücker coordinates of a linear subspace of $\mathbb{R}^N$ and those of its orthogonal complement. For $I\subset[N]$ let $B_I$ and $D_I$ be the submatrices with rows indexed by $I$. Let $A^{J}$ and $C^{J}$ be the submatrices with columns indexed by $J$, and put $I^c=[N]\setminus I$ and $|I|_\Sigma=\sum_{i\in I}i$.

\begin{proposition}[{\cite[Lemma~1.11(ii)]{Karp}}]\label{prop:gale-minors}
There are nonzero constants $\kappa_B$ and $\kappa_D$ such that
\[
\det B_I=\kappa_B\,(-1)^{|I|_\Sigma}\det A^{I^c}\quad(|I|=k),
\qquad
\det D_J=\kappa_D\,(-1)^{|J|_\Sigma}\det C^{J^c}\quad(|J|=k+1).
\]
\end{proposition}

To see this, recall that the maximal minors of $A$ (respectively $C$) are the Plücker coordinates of its row space $V$, those of $B$ (respectively $D$) are the Plücker coordinates of $\ker A=V^\perp$ (respectively $\ker C$), and \cite[Lemma~1.11(ii)]{Karp} states that $V$ and $\operatorname{alt}(V^\perp)$ have the same Plücker coordinates, $\Delta_I(V)=\Delta_{[N]\setminus I}(\operatorname{alt}(V^\perp))$, where $\operatorname{alt}$ changes the sign of every second coordinate. This duality goes back to Hilbert; see the discussion in \cite{Karp}.

In particular:
\begin{itemize}
\item any $k$ rows of $B$ are linearly independent if and only if no $n+1$ elements of $\mathcal A$ lie on an affine hyperplane;
\item any $k+1$ rows of $D$ are linearly independent, which is the assumption (U$_D$) of Section~\ref{sec:recursive}, if and only if all maximal minors of $C$ are nonzero;
\item the arrangement of the hyperplanes $\{p_i=0\}$ is determined, as an oriented matroid, by the signs of the maximal minors of $C$, and the signs of the maximal minors of $B$ are determined by those of $A$.
\end{itemize}

\section{Gale duality for polynomial systems}\label{sec:gale}

We recall the Gale duality for polynomial systems of Bihan and Sottile \cite{BS}, in the notation of Section~\ref{sec:polysys}. It transforms the positive solutions of~\eqref{eq:system} into the solutions of a system of $k$ equations in $k$ variables in the positive chamber $\Delta$, which is a chamber of the hyperplane arrangement defined by the rows of $D$. This is the setting of the rest of the paper.

\begin{definition}\label{def:gale-system}
The \emph{positive chamber} is the polyhedron
\[
\Delta=\{y\in\mathbb{R}^k:\ p_i(y)>0,\ i=1,\dots,n+k\}.
\]
The \emph{Gale system} of~\eqref{eq:system} is the system
\begin{equation}\label{eq:gale-system}
f_j(y)=\prod_{i=1}^{n+k}p_i(y)^{B_{ij}}=1,\qquad j=1,\dots,k,\qquad y\in\Delta .
\end{equation}
Equivalently, $\psi_j(y)=\sum_{i=1}^{n+k}B_{ij}\log p_i(y)=0$ for $j=1,\dots,k$ and $y\in\Delta$. A solution $y$ is \emph{nondegenerate} if the Jacobian matrix of $(\psi_1,\dots,\psi_k)$ is invertible at $y$.
\end{definition}

Since $p_N=1$, the factor $i=N$ can be added to the products and to the sums. Note that $\Delta$ is one of the chambers of the arrangement $\bigcup_{i\le n+k}\{p_i=0\}$, and that it may be empty or unbounded.

\begin{theorem}[Gale duality \cite{BS}]\label{thm:gale-duality}
Assume~(R). The map
\[
\Phi:\ \mathbb{R}^n_{>0}\longrightarrow\mathbb{R}^N_{>0},\qquad x\longmapsto\big(x^{a_1-a_N},\dots,x^{a_N-a_N}\big),
\]
maps the positive solutions of~\eqref{eq:system} into $\ker C\cap\{z_N=1\}$. Followed by the inverse of $y\mapsto D\binom y1$, it induces a bijection
\[
\{x\in\mathbb{R}^n_{>0}:\ F_1(x)=\dots=F_n(x)=0\}\ \xrightarrow{\ \sim\ }\ \{y\in\Delta:\ f_1(y)=\dots=f_k(y)=1\}.
\]
It is characterized by $p_j(y)=x^{a_j-a_N}$ for $j=1,\dots,N$. Moreover, $x$ is a nondegenerate solution of~\eqref{eq:system} if and only if $y$ is a nondegenerate solution of~\eqref{eq:gale-system}.
\end{theorem}

\section{The Bihan--Sottile bound for positive solutions}\label{sec:KR}

In this section we recall the bound of Bihan and Sottile, which is the starting point of our approach. It bounds the number of positive solutions by the number of common zeros in $\Delta$ of iterated Jacobians $\Gamma_1,\dots,\Gamma_k$, plus the numbers of noncompact components of some curves $C_j$ (Theorem~\ref{thm:KR-general}).

We keep the notation of Sections~\ref{sec:polysys} and~\ref{sec:gale} and assume~(R) and $\Delta\neq\emptyset$. By Theorem~\ref{thm:gale-duality}, the positive solutions of~\eqref{eq:system} correspond to the solutions in $\Delta$ of $\psi_1=\dots=\psi_k=0$, where
\[
\psi_j=\sum_{i=1}^{n+k}B_{ij}\log|p_i|,\qquad j=1,\dots,k .
\]
The functions $\psi_j$ are defined and smooth on the complement $\mathbb{R}^k\setminus\bigcup_{i\le n+k}\{p_i=0\}$ of the arrangement, and not only on $\Delta$; on $\Delta$ we have $\psi_j=\log f_j$. Following Bihan and Sottile \cite{BS}, we bound the number of positive solutions by the number of common zeros of successive Jacobians. Put $P=\prod_{i=1}^{n+k}p_i$, a polynomial of degree $n+k$, and define
\[
\Gamma_k=P\det\frac{\partial(\psi_1,\dots,\psi_k)}{\partial(y_1,\dots,y_k)},\qquad
\Gamma_j=P\det\frac{\partial(\psi_1,\dots,\psi_j,\Gamma_{j+1},\dots,\Gamma_k)}{\partial(y_1,\dots,y_k)},\quad j=k-1,\dots,1 .
\]
Since $\partial\psi_j/\partial y_l=\sum_iB_{ij}D_{il}/p_i$, these are polynomials (see Section~\ref{sec:jacobians}); from Section~\ref{sec:jacobians} on we study them for an arbitrary pair $(D,B)$ of type $(n,k)$ (Definition~\ref{def:setting}). The Khovanskii--Rolle argument of \cite{BS} needs the successive systems defined by the $\psi_j$ and the $\Gamma_j$ to be transversal, which is not always the case. Following \cite[Theorem~7]{BS2}, we apply it to small perturbations of the Jacobians. Let $\Theta_1,\dots,\Theta_k$ be smooth functions on $\Delta$; the main example is $\Theta_m=\Gamma_m$. Put
\begin{align*}
C_j(\Theta)&=\{y\in\Delta:\ \psi_1=\dots=\psi_{j-1}=0,\ \Theta_{j+1}=\dots=\Theta_k=0\},\qquad 1\le j\le k,\\
Z_j(\Theta)&=\{y\in\Delta:\ \psi_1=\dots=\psi_j=0,\ \Theta_{j+1}=\dots=\Theta_k=0\},\qquad 0\le j\le k,
\end{align*}
and, for $1\le m\le k$,
\[
J_m(\Theta)=P\det\frac{\partial(\psi_1,\dots,\psi_m,\Theta_{m+1},\dots,\Theta_k)}{\partial(y_1,\dots,y_k)} .
\]
Thus $Z_k(\Theta)$ is the solution set of the Gale system, $Z_0(\Theta)=\{\Theta_1=\dots=\Theta_k=0\}\cap\Delta$, and $Z_{j-1}(\Theta)=C_j(\Theta)\cap\{\Theta_j=0\}$. For $\Theta=\Gamma$ we have $J_m(\Gamma)=\Gamma_m$, and we write $C_j=C_j(\Gamma)$. Consider the conditions:
\begin{itemize}
\item[(i)] for $1\le j\le k$, the differentials of $\psi_1,\dots,\psi_{j-1},\Theta_{j+1},\dots,\Theta_k$ are linearly independent at every point of $C_j(\Theta)$, so that $C_j(\Theta)$ is a smooth curve (it is closed in $\Delta$);
\item[(ii)] for $1\le j\le k$, the points of $Z_j(\Theta)$ are nondegenerate solutions of $\psi_1=\dots=\psi_j=\Theta_{j+1}=\dots=\Theta_k=0$;
\item[(iii)] for $1\le j\le k$, $\Theta_j$ and $J_j(\Theta)$ have the same sign at every point of $Z_j(\Theta)$.
\end{itemize}
For $\Theta=\Gamma$, condition~(iii) is empty, and~(i) and~(ii) form the transversality assumption of the Khovanskii--Rolle argument of \cite{BS}. We write $b(C_j(\Theta))$ for the number of noncompact connected components of $C_j(\Theta)$.

\begin{theorem}[Bihan--Sottile bound \cite{BS,BS2}]\label{thm:KR-general}
Assume (R), and let $\Theta_1,\dots,\Theta_k$ satisfy (i), (ii) and (iii). Then the positive solutions of~\eqref{eq:system} are nondegenerate, and
\[
\#\{x\in\mathbb{R}^n_{>0}:\ F(x)=0\}\ \le\ \#Z_0(\Theta)+\sum_{j=1}^kb(C_j(\Theta)).
\]
\end{theorem}

\begin{proof}
By Theorem~\ref{thm:gale-duality}, the positive solutions correspond to the points of $Z_k(\Theta)$, and they are nondegenerate by~(ii) for $j=k$. Let $1\le j\le k$. By~(i), the vector $\tau(y)$ defined by $\langle x,\tau(y)\rangle=\det\big(\nabla\psi_1;\dots;\nabla\psi_{j-1};x;\nabla \Theta_{j+1};\dots;\nabla \Theta_k\big)(y)$ for all $x\in\mathbb{R}^k$ is a nowhere vanishing tangent vector field on $C_j(\Theta)$, and $J_j(\Theta)=P\,\langle\nabla\psi_j,\tau\rangle$. Since $P>0$ on $\Delta$, the sign of $J_j(\Theta)$ on $C_j(\Theta)$ is the sign of the derivative of $\psi_j$ along $\tau$. The set $Z_j(\Theta)$ is the zero set of $\psi_j$ on $C_j(\Theta)$, and these zeros are simple by~(ii). Let $a$ and $b$ be two consecutive zeros on a connected component $K$ of $C_j(\Theta)$. The derivatives of $\psi_j$ along $\tau$ at $a$ and $b$ have opposite signs, hence so have $J_j(\Theta)(a)$ and $J_j(\Theta)(b)$, and so have $\Theta_j(a)$ and $\Theta_j(b)$ by~(iii). So $\Theta_j$ vanishes on the open arc of $K$ between $a$ and $b$, at a point of $Z_{j-1}(\Theta)$. If $K$ is compact, it contains as many such arcs as zeros of $\psi_j$; if $K$ is noncompact, it contains one less. We get $\#Z_j(\Theta)\le\#Z_{j-1}(\Theta)+b(C_j(\Theta))$, and the theorem follows by induction.
\end{proof}

For $1\le m\le k$ let $D_m=2^{k-m}n$, the degree bound of Proposition~\ref{prop:gammak}(2). We identify a polynomial $\Theta_m$ of degree at most $D_m$ with its homogenization $\Theta^h_m(y,y_{k+1})=y_{k+1}^{D_m}\Theta_m(y/y_{k+1})$, a form of degree $D_m$, and we measure the distance between such polynomials by the maximum of the absolute values of the differences of their coefficients.

\begin{lemma}[Perturbed Jacobians, after {\cite[Theorem~7]{BS2}}]\label{lem:perturbed-jacobians}
Assume (R), (U$_D$), (U$_B$) and (G) (Section~\ref{sec:recursive}), and that, for $1\le r\le k$, the value $0$ is a regular value of the map $(\psi_1,\dots,\psi_r)\colon\Delta\to\mathbb{R}^r$. Then for every $\eta>0$ there are polynomials $\Theta_1,\dots,\Theta_k$, with $\deg \Theta_m\le D_m$ and $\Theta_m$ at distance at most $\eta$ from $\Gamma_m$, which satisfy (i), (ii) and (iii).
\end{lemma}

The idea is to choose the polynomials $\Theta_k,\Theta_{k-1},\dots,\Theta_1$ in this order, each one being the Jacobian of the previously chosen ones plus a small constant, and to use Sard's theorem to ensure transversality. The proof is given in Appendix~\ref{app:proofs}.

The regularity assumption of Lemma~\ref{lem:perturbed-jacobians} is harmless. Let $s\in\mathbb{R}^{n+k}_{>0}$, and replace the coefficient matrix $C$ by $C\operatorname{diag}(s_1^{-1},\dots,s_{n+k}^{-1},1)$, with Gale dual matrix $\operatorname{diag}(s_1,\dots,s_{n+k},1)D$. This multiplies the affine functions $p_i$, $i\le n+k$, by positive constants. It does not change the signs of the minors of $C$ and of $D$, nor the arrangement and its chambers. It multiplies $P$ and the Jacobians $\Gamma_m$ by positive constants, and by Lemma~\ref{lem:chart-independence} it changes neither the common zeros of the Jacobians of the reduced systems of Section~\ref{sec:recursive} nor their nondegeneracy. It replaces $\psi_j$ by $\psi_j+\sum_{i\le n+k}B_{ij}\log s_i$. The columns of $B$ sum to zero, so its last row is determined by the others and $B'$ has rank $k$. It follows that $\log s\mapsto\big(\sum_iB_{ij}\log s_i\big)_j$ is a surjective linear map, and the preimage of a null set is a null set. By Sard's theorem, for almost every $s$ the regularity assumption holds for the new system. By the implicit function theorem, every finite set of nondegenerate positive solutions of the original system gives, for $s$ close enough to $(1,\dots,1)$, a set of as many nondegenerate positive solutions of the new system.

\begin{remark}\label{rem:BS-jacobians}
The Jacobians used in \cite{BS} are not exactly the polynomials $\Gamma_j$. They are the rational functions
\[
G_k=\det\frac{\partial(\psi_1,\dots,\psi_k)}{\partial(y_1,\dots,y_k)},\qquad
G_j=\det\frac{\partial(\psi_1,\dots,\psi_j,G_{j+1},\dots,G_k)}{\partial(y_1,\dots,y_k)},\quad j=k-1,\dots,1,
\]
obtained without multiplying by $P$ at each step, and it is shown in \cite{BS} that $P^{2^{k-j}}G_j$ is a polynomial of degree at most $2^{k-j}n$; the degree bound of Proposition~\ref{prop:gammak}(2) below is essentially theirs. One checks that $P^{2^{k-j}}G_j$ and $\Gamma_j$ differ by polynomial multiples of the later Jacobians $\Gamma_{j'}$, $j'>j$, so that the two families have the same common zeros, with the same nondegeneracy. Since all our results only involve common zero sets of the Jacobians and their nondegeneracy, they hold for either family. We use the $\Gamma_j$ because they are defined by a recursion on polynomials, and because they satisfy exact restriction formulas on the flats (Theorems~\ref{thm:restriction-LS} and~\ref{thm:restriction-LS-infinity}), which hold for the $P^{2^{k-j}}G_j$ only modulo the later Jacobians. For instance $P^2G_{k-1}=\Gamma_{k-1}-Q\,\Gamma_k$ for a polynomial $Q$, so that $P^2G_{k-1}$ does not vanish in general at the vertices of the arrangement, whereas $\Gamma_{k-1}$ does (Corollary~\ref{cor:flat-vanishing}).
\end{remark}

The Jacobians $\Gamma_j$, obtained by multiplying by $P$ at each step of the recursion, were pointed out to the author by Frank Sottile \cite{Sottile-pc}, together with the property that $\Gamma_{k-1}$ vanishes at the vertices of the hyperplane arrangement (see Corollary~\ref{cor:flat-vanishing}).

The rest of the paper provides bounds for the two terms of Theorem~\ref{thm:KR-general}, applied to perturbations $\Theta$ of the Jacobians given by Lemma~\ref{lem:perturbed-jacobians}. By Proposition~\ref{prop:ends}, if $\Theta$ is close enough to $\Gamma$:
\begin{itemize}
\item $\#Z_0(\Theta)$ is the number $Z(\emptyset,\Delta)$ of zeros of $\Gamma_1=\dots=\Gamma_k=0$ in $\Delta$, bounded above by $u_\emptyset(\Delta)$ and $u^*_\emptyset(\Delta)$ in Sections~\ref{sec:recursive} and~\ref{sec:BDF};
\item $2b(C_j(\Theta))$ is at most the number of zeros of $\Gamma_{j+1},\dots,\Gamma_k$ on the faces of $\overline\Delta$ of codimension $j$ that satisfy a sign condition on the minors of $B$.
\end{itemize}
The resulting bounds for the number of positive solutions of~\eqref{eq:system} are Theorem~\ref{thm:possol-Z} and Corollaries~\ref{cor:possol-recursive} and~\ref{cor:possol-BDF}.

\section{The Jacobians $\Gamma_j$}\label{sec:jacobians}

In this section we study the iterated Jacobians $\Gamma_m$ of Section~\ref{sec:KR}. We give a Cauchy--Binet formula which expresses them in terms of maximal minors of $B$ and $D$, and we deduce bounds for their degrees (Proposition~\ref{prop:gammak}). We then show that they vanish on the flats of the arrangement, with prescribed orders (Corollary~\ref{cor:flat-vanishing}). These properties are used throughout the recursion.

\subsection{A Cauchy--Binet formula and the degrees of the Jacobians}\label{sec:gammak}

\begin{proposition}\label{prop:gammak}
Let $(D,B)$ be a pair of type $(n,k)$ (Definition~\ref{def:setting}).
\begin{enumerate}
\item For $1\le m\le k$,
\[
\Gamma_m=\sum_{I}\det B_{I,\{1,\dots,m\}}\prod_{i\in\{1,\dots,n+k\}\setminus I}p_i\cdot\det\big(D'_I;\nabla\Gamma_{m+1};\dots;\nabla\Gamma_k\big),
\]
where $I$ runs over the subsets of $\{1,\dots,n+k\}$ with $|I|=m$, and the rows indexed by $I$ are taken in the same order in $B_{I,\{1,\dots,m\}}$ and in $D'_I$. In particular, for $m=k$,
\[
\Gamma_k=\sum_{|I|=k}\det(B'_I)\,\det(D'_I)\prod_{i\notin I}p_i ,
\]
where $B'_I$ and $D'_I$ are the $k\times k$ submatrices of $B'$ and $D'$ with rows indexed by $I$.
\item For $j=0,\dots,k-1$,
\[
\deg\Gamma_{k-j}\le 2^{j}\,n .
\]
In particular $\deg\Gamma_k\le n$, $\deg\Gamma_{k-1}\le 2n$, $\deg\Gamma_{k-2}\le 4n$, and $\deg\Gamma_1\le 2^{k-1}n$.
\end{enumerate}
\end{proposition}

\begin{proof}
1. By definition, $\Gamma_m=P\det(\nabla\psi_1;\dots;\nabla\psi_m;\nabla\Gamma_{m+1};\dots;\nabla\Gamma_k)$, with $\nabla\psi_l=\sum_{i=1}^{n+k}B_{il}D'_i/p_i$. Expanding the first $m$ rows by multilinearity, and discarding the terms with a repeated index, which vanish, we get
\begin{equation}\label{eq:gamma-expansion}
\Gamma_m=\sum_{(i_1,\dots,i_m)}B_{i_11}\cdots B_{i_mm}\prod_{i\notin\{i_1,\dots,i_m\}}p_i\cdot
\det\big(D'_{i_1};\dots;D'_{i_m};\nabla\Gamma_{m+1};\dots;\nabla\Gamma_k\big),
\end{equation}
where $(i_1,\dots,i_m)$ runs over the sequences of pairwise distinct indices in $\{1,\dots,n+k\}$. Fix $I=\{r_1,\dots,r_m\}$, listed in the chosen order. The sequences with underlying set $I$ are the sequences $(r_{\pi(1)},\dots,r_{\pi(m)})$, where $\pi$ is a permutation of $\{1,\dots,m\}$. Reordering the first $m$ rows, the corresponding determinant is $\sign(\pi)\det(D'_I;\nabla\Gamma_{m+1};\dots;\nabla\Gamma_k)$, and summing $\sign(\pi)\prod_{l=1}^mB_{r_{\pi(l)}l}$ over $\pi$ gives $\det B_{I,\{1,\dots,m\}}$. For $m=k$ there are no rows $\nabla\Gamma_{m'}$, and $B_{I,\{1,\dots,k\}}=B'_I$ since $I\subset\{1,\dots,n+k\}$.

2. Let $d_m$ be an upper bound on $\deg\Gamma_m$. In each term of~1, the product $\prod_{i\notin I}p_i$ has degree at most $n+k-m$, the rows $D'_I$ are constant, and the entries of the row $\nabla\Gamma_{m'}$ have degree at most $d_{m'}-1$. Hence
\[
d_m\le n+k-m+\sum_{m'=m+1}^{k}(d_{m'}-1)=n+\sum_{m'=m+1}^{k}d_{m'} .
\]
Put $S_j=\sum_{m'=k-j+1}^{k}d_{m'}$, so that $S_0=0$ and $d_{k-j}\le n+S_j$. Then $S_{j+1}=S_j+d_{k-j}\le2S_j+n$, and induction gives $S_j\le(2^j-1)n$, hence $d_{k-j}\le2^jn$.
\end{proof}

\subsection{Vanishing on flats}

For $S\subset\{1,\dots,n+k\}$ put $L_S=\bigcap_{i\in S}\{p_i=0\}\subset\mathbb{R}^k$, and let $I_S$ be the ideal of $\mathbb{R}[y_1,\dots,y_k]$ generated by the $p_i$, $i\in S$.

\begin{corollary}\label{cor:flat-vanishing}
Let $(D,B)$ be a pair of type $(n,k)$, $S\subset\{1,\dots,n+k\}$ and $1\le m<|S|$. Then $\Gamma_m\in I_S^{\,|S|-m}$. In particular $\Gamma_m$ vanishes on $L_S$, to order at least $|S|-m$ at every point of $L_S$.
\end{corollary}

\begin{proof}
In Proposition~\ref{prop:gammak}(1), the product $\prod_{i\notin I}p_i$ contains at least $|S\setminus I|\ge|S|-m$ factors $p_i$ with $i\in S$.
\end{proof}

\section{The boundary degree}\label{sec:boundary-degree}

In this section we introduce the boundary degree $d(\mathcal C)$ of a chamber $\mathcal C$ of the arrangement, which is the degree of the map $\Gamma/|\Gamma|$ on $\partial\mathcal C$, and we compute it. By the Kronecker--Poincar\'e--Hopf theorem, $|d(\mathcal C)|$ is a lower bound for the number of zeros of $\Gamma$ in $\mathcal C$. We express $d(\mathcal C)$ by a Cauchy-index formula (Theorem~\ref{thm:cauchy-general}), and then by a facet-count formula, which only involves the numbers of zeros of $\Gamma_2,\dots,\Gamma_k$ on the facets of $\mathcal C$ and the signs of the first column of $B$ (Theorem~\ref{thm:facet-count}). The unbounded chambers are treated in Subsection~\ref{subsec:unbounded}.

By Theorem~\ref{thm:KR-general}, applied to small perturbations of the Jacobians (Lemma~\ref{lem:perturbed-jacobians} and Proposition~\ref{prop:ends}), the number of nondegenerate positive solutions of~\eqref{eq:system} is at most the number of common zeros of $\Gamma_1,\dots,\Gamma_k$ in $\Delta$, plus the numbers of noncompact components of curves $C_j(\Theta)$. To bound the number of common zeros in $\Delta$, we use a topological invariant of the chambers of the arrangement $\bigcup_{i\le n+k}\{p_i=0\}$.

Let $\mathcal C$ be a connected component of $\mathbb{R}^k\setminus\bigcup_{i}\{p_i=0\}$. Assume that $\mathcal C$ is bounded, so that $\overline{\mathcal C}$ is a polytope and $\partial \mathcal C\cong S^{k-1}$. Orient $\partial \mathcal C$ as the boundary of $\mathcal C$ (outward normal first). Assume that $\Gamma=(\Gamma_1,\dots,\Gamma_k)$ does not vanish on $\partial \mathcal C$, and set
\[
d(\mathcal C)=\deg\left(\frac{\Gamma}{|\Gamma|}\colon \partial \mathcal C\longrightarrow S^{k-1}\right).
\]
By the Kronecker--Poincar\'e--Hopf theorem, if the zeros of $\Gamma$ in $\mathcal C$ are isolated, then
\begin{equation}\label{eq:KPH}
d(\mathcal C)=\sum_{y\in \mathcal C,\ \Gamma(y)=0}\ind_y\Gamma ,
\end{equation}
and at a nondegenerate zero the index is $\sign\det D\Gamma(y)$. In particular, if these zeros are nondegenerate, $|d(\mathcal C)|$ is a lower bound for the number of solutions of $\Gamma_1=\dots=\Gamma_k=0$ in $\mathcal C$. Unbounded chambers are treated in the same way, with a contribution at infinity (Subsection~\ref{subsec:unbounded}).

We now look for an upper bound for
\[
\#\big(\{\Gamma_1=\dots=\Gamma_k=0\}\cap\Delta\big)
\]
which is better than the Bézout bound. By Proposition~\ref{prop:gammak}(2), the product of the degrees of $\Gamma_1,\dots,\Gamma_k$ is at most the \emph{Bézout number}
\[
\beta_k=\prod_{j=0}^{k-1}2^jn=n^k\,2^{\binom k2},
\]
so that, if the common zeros of $\Gamma_1,\dots,\Gamma_k$ are nondegenerate, there are at most $\beta_k$ of them in $\mathbb{R}^k$ (see Lemma~\ref{lem:bezout-flat} for the precise statement). This bound does not distinguish $\Delta$ from the other chambers of the arrangement $\bigcup_{i\le n+k}\{p_i=0\}$, and the idea is to subtract the zeros lying in the other chambers:
\[
\#\big(\{\Gamma_1=\dots=\Gamma_k=0\}\cap\Delta\big)\ \le\ \beta_k-\sum_{\mathcal C\neq\Delta}\#\big(\{\Gamma_1=\dots=\Gamma_k=0\}\cap \mathcal C\big),
\]
where $\mathcal C$ runs over the chambers different from $\Delta$. By~\eqref{eq:KPH}, if the zeros are nondegenerate, the number of zeros in a chamber $\mathcal C$ is at least $|d(\mathcal C)|$. Hence
\[
\#\big(\{\Gamma_1=\dots=\Gamma_k=0\}\cap\Delta\big)\ \le\ \beta_k-\sum_{\mathcal C\neq\Delta}|d(\mathcal C)| .
\]
In this section we compute $d(\mathcal C)$ in terms of the zeros of $\Gamma_2,\dots,\Gamma_k$ on the facets of $\mathcal C$ and of the signs of the first column of $B$ (Theorems~\ref{thm:facet-count} and~\ref{thm:facet-count-unbounded}).

\begin{remark}[Normalization]\label{rem:normalization}
Replacing $p_i$ by $-p_i$ does not change the functions $\psi_j$, which involve $\log|p_i|$. It multiplies $P$ and $\Gamma_k$ by $-1$, and it leaves $\Gamma_1,\dots,\Gamma_{k-1}$ unchanged, since by downward induction each of them contains the factor $P$ and the row $\nabla\Gamma_k$, while its other rows are unchanged. So it multiplies $d(\mathcal C)$ by $-1$. This is the origin of the factor $\epsilon=\sign P|_{\mathcal C}$ in Theorem~\ref{thm:facet-count}.
\end{remark}

\subsection{The case $k=1$}
Here $\overline{\mathcal C}=[a,b]$ with $p_{i_0}(a)=0$ and $p_{i_1}(b)=0$, and $\Gamma_1=P\,\psi_1'$. Since $k=1$, $B$ has a single column. Assume $p_i>0$ on $\mathcal C$ for all $i$ (see Remark~\ref{rem:normalization}) and $B_{i_01},B_{i_11}\neq0$. Then $D'_{i_0}>0$ and $D'_{i_1}<0$, so
\[
\sign\Gamma_1(a)=\sign B_{i_01},\qquad \sign\Gamma_1(b)=-\sign B_{i_11}.
\]
With the boundary orientation ($b$ counted positively, $a$ negatively),
\[
d(\mathcal C)=\tfrac12\big(\sign\Gamma_1(b)-\sign\Gamma_1(a)\big)
=-\tfrac12\big(\sign B_{i_01}+\sign B_{i_11}\big)\in\{-1,0,1\}.
\]
So $d(\mathcal C)$ is not a universal quantity, since it depends on $B$. For instance, if $B_{i_01},B_{i_11}>0$, then $\psi_1\to-\infty$ at both ends of $\mathcal C$, so $\psi_1$ has an interior maximum, of index $-1$.

\subsection{A Cauchy-index formula for $d(\mathcal C)$}

\begin{theorem}[Cauchy-index representation, general $k$]\label{thm:cauchy-general}
Assume that every zero of $(\Gamma_2,\dots,\Gamma_k)$ on $\partial \mathcal C$ lies in the relative interior of a facet, and that at each such zero the restriction of $(\Gamma_2,\dots,\Gamma_k)$ to that facet has nonsingular differential. Then
\[
d(\mathcal C)=\frac12\sum_{\substack{y\in\partial \mathcal C\\ \Gamma_2(y)=\dots=\Gamma_k(y)=0}}
\sign\big(\Gamma_1(y)\big)\cdot
\sign\det\big(D_{v_j}\Gamma_i(y)\big)_{\substack{2\le i\le k\\ 1\le j\le k-1}},
\]
where $(v_1,\dots,v_{k-1})$ is a positively oriented basis of the facet containing $y$, for the boundary orientation of $\partial \mathcal C$.
\end{theorem}

\begin{proof}
By the hypotheses, $e_1$ and $-e_1$ are regular values of $\Gamma/|\Gamma|$ on $\partial \mathcal C$. The preimages of $\pm e_1$ are the points $y\in\partial \mathcal C$ with $\Gamma_2(y)=\dots=\Gamma_k(y)=0$ and $\pm\Gamma_1(y)>0$.

Near $e_1$, use the chart $(u_2,\dots,u_k)$ of $S^{k-1}$. It is positively oriented, since $(e_1,e_2,\dots,e_k)$ is a positive basis of $\mathbb{R}^k$ with $e_1$ the outward normal. In this chart the map reads $u_i=\Gamma_i/\Gamma_1$. At a preimage its differential is $D\Gamma_i/\Gamma_1$ with $\Gamma_1>0$. Hence
\[
d(\mathcal C)=\sum_{\substack{\Gamma_2=\dots=\Gamma_k=0\\ \Gamma_1>0}}\sign\det\big(D_{v_j}\Gamma_i\big).
\]
Near $-e_1$ the outward normal is $-e_1$, so the chart $(u_2,\dots,u_k)$ is negatively oriented. The same computation gives
\[
d(\mathcal C)=-\sum_{\substack{\Gamma_2=\dots=\Gamma_k=0\\ \Gamma_1<0}}\sign\det\big(D_{v_j}\Gamma_i\big).
\]
Averaging the two expressions yields the formula.
\end{proof}

\begin{remark}\label{rem:signed-count}
Comparing the two expressions in the proof shows that the signed count of all zeros of $(\Gamma_2,\dots,\Gamma_k)$ on $\partial \mathcal C$ vanishes. This is consistent with the fact that this zero set is the oriented boundary of the compact curve $\partial \mathcal C\cap\{\Gamma_3=\dots=\Gamma_k=0\}$. It is the analogue of $\sum\sign D_v\Gamma_2=0$ for $k=2$.
\end{remark}

\subsection{Simplification using the definition of $\Gamma_1$}

For $i=1,\dots,n+k$, let $\epsilon_i=\sign p_i|_{\mathcal C}$, and let $\epsilon=\sign P|_{\mathcal C}=\prod_i\epsilon_i$. For a point $y$ in the relative interior of a facet, let $i(y)$ be the index such that $y\in\{p_{i(y)}=0\}$.

The following lemma is the key point of this section. A priori, the product on its left-hand side depends on the point $y$, through the value of $\Gamma_1$ and the derivatives of $\Gamma_2,\dots,\Gamma_k$ at $y$. Remarkably, its sign does not depend on $y$ itself, but only on the facet containing $y$, namely on the index $i(y)$, through the sign of the entry $B_{i(y)1}$ of the first column of $B$ and the global sign $\epsilon$. This is what turns the Cauchy-index formula of Theorem~\ref{thm:cauchy-general} into the facet-count formula of Theorem~\ref{thm:facet-count}.

\begin{lemma}\label{lem:sign-identity}
Under the hypotheses of Theorem~\ref{thm:cauchy-general}, for every $y\in\partial \mathcal C$ with $\Gamma_2(y)=\dots=\Gamma_k(y)=0$,
\[
\sign\big(\Gamma_1(y)\big)\cdot
\sign\det\big(D_{v_j}\Gamma_i(y)\big)_{\substack{2\le i\le k\\ 1\le j\le k-1}}
=-\,\epsilon\,\sign B_{i(y)1}.
\]
\end{lemma}

\begin{proof}
Write $i_0=i(y)$.

\emph{Step 1: $\Gamma_1$ on the facet.} By definition, $\Gamma_1=P\cdot\det\big(\nabla\psi_1;\nabla\Gamma_2;\dots;\nabla\Gamma_k\big)$, where the arguments are the rows. Only the first row has poles, and
\[
P\,\nabla\psi_1=\sum_{i=1}^{n+k}B_{i1}\,D'_i\prod_{l\ne i}p_l .
\]
Since $y$ lies in the relative interior of $F_{i_0}$, the only $p_l$ vanishing at $y$ is $p_{i_0}$, so only the term $i=i_0$ survives:
\[
\Gamma_1(y)=B_{i_01}\prod_{l\ne i_0}p_l(y)\cdot\det\big(D'_{i_0};\nabla\Gamma_2(y);\dots;\nabla\Gamma_k(y)\big).
\]

\emph{Step 2: the determinant along the facet.} The outward unit normal to $F_{i_0}$ is $\nu=-\epsilon_{i_0}D'_{i_0}/|D'_{i_0}|$, and $(\nu,v_1,\dots,v_{k-1})$ is a positive basis of $\mathbb{R}^k$. Let $W$ be the matrix with these columns, so $\det W>0$. Multiplying on the right by $W$, the row $D'_{i_0}$ becomes $(-\epsilon_{i_0}|D'_{i_0}|,0,\dots,0)$ and the row $\nabla\Gamma_i$ becomes $(D_\nu\Gamma_i,D_{v_1}\Gamma_i,\dots,D_{v_{k-1}}\Gamma_i)$. Expanding along the first row,
\[
\sign\det\big(D'_{i_0};\nabla\Gamma_2;\dots;\nabla\Gamma_k\big)
=-\epsilon_{i_0}\,\sign\det\big(D_{v_j}\Gamma_i\big)_{\substack{2\le i\le k\\ 1\le j\le k-1}}.
\]

\emph{Step 3: collecting signs.} Since $\sign\prod_{l\ne i_0}p_l(y)=\prod_{l\ne i_0}\epsilon_l$, Steps 1 and 2 give
\[
\sign\Gamma_1(y)=-\,\epsilon\,\sign B_{i_01}\cdot\sign\det\big(D_{v_j}\Gamma_i(y)\big).
\]
By transversality the determinant is nonzero. Multiplying both sides by its sign yields the identity.
\end{proof}

\begin{theorem}[Facet-count formula]\label{thm:facet-count}
Assume the hypotheses of Theorem~\ref{thm:cauchy-general}. For $i=1,\dots,n+k$, let $N_i$ be the number of zeros of $(\Gamma_2,\dots,\Gamma_k)$ in the relative interior of the facet $F_i=\{p_i=0\}\cap\partial \mathcal C$. This is an unsigned count, with $N_i=0$ if $\{p_i=0\}$ does not support a facet of $\mathcal C$. Then
\[
d(\mathcal C)=-\frac{\epsilon}{2}\sum_{i=1}^{n+k}\sign(B_{i1})\,N_i .
\]
In particular,
\[
|d(\mathcal C)|=\frac12\,\Big|\sum_{i=1}^{n+k}\sign(B_{i1})\,N_i\Big| ,
\]
which does not depend on $\epsilon$.
\end{theorem}

\begin{proof}
By Theorem~\ref{thm:cauchy-general} and Lemma~\ref{lem:sign-identity},
\[
d(\mathcal C)=\frac12\sum_{\substack{y\in\partial \mathcal C\\ \Gamma_2(y)=\dots=\Gamma_k(y)=0}}\big(-\epsilon\,\sign B_{i(y)1}\big)
=-\frac{\epsilon}{2}\sum_{i=1}^{n+k}\sign(B_{i1})\,N_i ,
\]
where the second equality groups the zeros according to the facet containing them.
\end{proof}

By Theorem~\ref{thm:restriction-LS}, $N_i$ is the number of common zeros, in the facet $F_i$, of the Jacobians of a reduced system of the same type with $k-1$ variables; this is how these numbers are bounded recursively in Section~\ref{sec:recursive}.

\subsection{An example with $n=2$ and $k=3$}

\begin{example}\label{ex:n2k3}
Consider the $n+k=5$ affine functions in $y=(y_1,y_2,y_3)$
\[
p_1=y_1,\quad p_2=y_2,\quad p_3=y_3,\quad p_4=1-y_1-y_2-y_3,\quad
p_5=1-2y_1-\tfrac13 y_2+\tfrac15 y_3,
\]
and the matrix
\[
B=\begin{pmatrix}
-3 & 1 & 3\\
2 & -3 & -2\\
2 & 1 & -1\\
1 & 3 & -3\\
1 & -3 & -1
\end{pmatrix}.
\]

\emph{The cell.} Let $\mathcal C=\{p_1>0,\dots,p_5>0\}$. It is the standard simplex with the vertex $e_1$ cut off by the plane $\{p_5=0\}$, so all five planes support facets of $\mathcal C$. Its vertices are
\[
(0,0,0),\quad (0,1,0),\quad (0,0,1),\quad (\tfrac12,0,0),\quad
(\tfrac25,\tfrac35,0),\quad (\tfrac{6}{11},0,\tfrac{5}{11}).
\]
The facets $F_1$ and $F_5$ are triangles, and $F_2$, $F_3$, $F_4$ are quadrilaterals. All $p_i$ are positive on $\mathcal C$, so $\epsilon=+1$.

\emph{The functions $\Gamma_j$.} We have
\[
\Gamma_3=-30y_1^2+10y_1y_2-\tfrac{88}{5}y_1y_3-\tfrac53y_2^2-\tfrac{28}{5}y_2y_3
+\tfrac75y_3^2-2y_1+8y_2+\tfrac{42}{5}y_3+7,
\]
while $\Gamma_2$ and $\Gamma_1$ have $34$ and $133$ monomials respectively.

\emph{Zeros on the boundary.} The system $\Gamma_2=\Gamma_3=0$ has exactly four solutions on $\partial \mathcal C$. Each lies in the relative interior of a facet and is transversal. The values below are rounded to the digits shown:
\[
\begin{array}{c|c|c|c|c}
\text{facet} & y & \Gamma_1(y) & \det\big(D_{v_j}\Gamma_i(y)\big) & -\epsilon\,\sign B_{i1}\\
\hline
F_2 & (0.4565,\ 0,\ 0.2367) & -10.8 & +1209 & -1\\
F_3 & (0.4703,\ 0.0455,\ 0) & +3.25 & -3536 & -1\\
F_4 & (0.4873,\ 0.0890,\ 0.4238) & +7.77 & -3032 & -1\\
F_5 & (0.4982,\ 0.1248,\ 0.1899) & -2.82 & +626 & -1
\end{array}
\]
In every row $\sign\Gamma_1(y)\cdot\sign\det\big(D_{v_j}\Gamma_i(y)\big)=-\epsilon\,\sign B_{i(y)1}$, in accordance with Lemma~\ref{lem:sign-identity}.

\emph{The degree.} We have $(N_1,\dots,N_5)=(0,1,1,1,1)$, and the first column of $B$ is $(-3,2,2,1,1)^T$. Theorem~\ref{thm:facet-count} gives
\[
d(\mathcal C)=-\frac12\Big(\sign(-3)\cdot 0+\sign(2)\cdot1+\sign(2)\cdot1+\sign(1)\cdot1+\sign(1)\cdot1\Big)=-2 .
\]
The Cauchy-index formula of Theorem~\ref{thm:cauchy-general}, using the Jacobian signs of the table, gives the same value. In particular, the system $\Gamma_1=\Gamma_2=\Gamma_3=0$ has at least two solutions in $\mathcal C$.
\end{example}

\subsection{The unbounded case: a contribution at infinity}\label{subsec:unbounded}

In this subsection $\mathcal C$ is a connected component of $\mathbb{R}^k\setminus\bigcup_i\{p_i=0\}$, which may be unbounded. As before, $\epsilon_i=\sign p_i|_{\mathcal C}$ and $\epsilon=\prod_i\epsilon_i$.

\medskip
\noindent\emph{The last row of $B$.} The polynomial $p_{n+k+1}=1$ plays the role of the hyperplane at infinity. Since the columns of $B$ sum to zero,
\[
\sum_{i=1}^{n+k}B_{i1}=-B_{n+k+1,1}.
\]

\medskip
\noindent\emph{Top-degree parts.} For a polynomial $g$ and an integer $D\ge\deg g$, let $g^{\mathrm{top}}$ denote its homogeneous component of degree $D$, where $D$ is specified each time. We use:
\begin{itemize}
\item $p_i^{\mathrm{top}}(y)=\langle D'_i,y\rangle$, the linear part of $p_i$, and $P^{\mathrm{top}}=\prod_{i=1}^{n+k}p_i^{\mathrm{top}}$;
\item $\psi_1^{\mathrm{top}}=\sum_{i=1}^{n+k}B_{i1}\log|p_i^{\mathrm{top}}|$;
\item $\Gamma_m^{\mathrm{top}}$, the homogeneous component of degree $D_m:=2^{k-m}n$ of $\Gamma_m$, for $m=1,\dots,k$.
\end{itemize}

\medskip
\noindent\emph{Data at infinity.}
\begin{itemize}
\item The recession cone of $\mathcal C$ is
\[
\mathcal C_\infty=\{y\in\mathbb{R}^k:\ \epsilon_ip_i^{\mathrm{top}}(y)\ge0,\ i=1,\dots,n+k\}.
\]
Set $\Sigma=\mathcal C_\infty\cap S^{k-1}$.
\item Let $\Sigma^\circ$ be the interior of $\Sigma$ in $S^{k-1}$. If $\mathcal C_\infty$ is full-dimensional, then $\Sigma^\circ=\{\theta\in S^{k-1}:\epsilon_ip_i^{\mathrm{top}}(\theta)>0\ \forall i\}$. Otherwise $\Sigma^\circ=\emptyset$.
\item Set $\mathcal C_R=\mathcal C\cap B(0,R)$.
\end{itemize}

\begin{lemma}[Top-degree parts]\label{lem:top}
For $m=1,\dots,k$, $\Gamma_m^{\mathrm{top}}$ is the function $\Gamma_m$ built from the central arrangement $p_1^{\mathrm{top}},\dots,p_{n+k}^{\mathrm{top}}$ with the same matrix $B$. It is homogeneous of degree $D_m$, and
\[
\Gamma_m=\Gamma^{\mathrm{top}}_m+(\text{terms of degree}<D_m).
\]
\end{lemma}

\begin{proof}
Expanding the logarithmic rows by multilinearity, as in~\eqref{eq:gamma-expansion}, gives
\[
\Gamma_{k-j}=\sum_{(i_1,\dots,i_{k-j})}\Big(\prod_{s=1}^{k-j}B_{i_s s}\Big)\prod_{l\notin I}p_l\cdot
\det\big(D'_{i_1};\dots;D'_{i_{k-j}};\nabla\Gamma_{k-j+1};\dots;\nabla\Gamma_k\big).
\]
Here the sum runs over ordered tuples of pairwise distinct indices in $\{1,\dots,n+k\}$, and $I=\{i_1,\dots,i_{k-j}\}$.

We argue by induction on $j$. Replace $p_l$ by $p_l^{\mathrm{top}}$ and $\Gamma_m$ by the functions built from the central arrangement. By the induction hypothesis these are the $\Gamma_m^{\mathrm{top}}$, and $\nabla\Gamma_m^{\mathrm{top}}$ is the homogeneous component of degree $D_m-1$ of $\nabla\Gamma_m$. So each term of the expansion equals the corresponding term of the central expansion, which is homogeneous of degree
\[
(n+j)+\sum_{m=k-j+1}^{k}(D_m-1)=n+j+(2^j-1)n-j=2^jn=D_{k-j},
\]
plus terms of lower degree. The sum of the central terms is the function $\Gamma_{k-j}$ of the central arrangement.
\end{proof}

\begin{lemma}[Sign identity at infinity]\label{lem:sign-infinity}
Let $\theta\in\Sigma^\circ$ satisfy $\Gamma^{\mathrm{top}}_2(\theta)=\dots=\Gamma^{\mathrm{top}}_k(\theta)=0$. Let $(v_1,\dots,v_{k-1})$ be a basis of $T_\theta S^{k-1}$ such that $(\theta,v_1,\dots,v_{k-1})$ is a positive basis of $\mathbb{R}^k$. If
\[
\det\big(D_{v_j}\Gamma^{\mathrm{top}}_i(\theta)\big)\neq0,
\]
then
\[
\sign\big(\Gamma^{\mathrm{top}}_1(\theta)\big)\cdot
\sign\det\big(D_{v_j}\Gamma^{\mathrm{top}}_i(\theta)\big)_{\substack{2\le i\le k\\ 1\le j\le k-1}}
=-\,\epsilon\,\sign B_{n+k+1,1} .
\]
\end{lemma}

\begin{proof}
By Lemma~\ref{lem:top},
\[
\Gamma^{\mathrm{top}}_1=P^{\mathrm{top}}\det\big(\nabla\psi^{\mathrm{top}}_1;\nabla\Gamma^{\mathrm{top}}_2;\dots;\nabla\Gamma^{\mathrm{top}}_k\big).
\]
We use three facts.
\begin{enumerate}
\item On $\Sigma^\circ$ we have $\sign p_l^{\mathrm{top}}=\epsilon_l$ for all $l$, so $\sign P^{\mathrm{top}}(\theta)=\epsilon$.
\item For $t>0$,
\[
\psi^{\mathrm{top}}_1(ty)=\psi^{\mathrm{top}}_1(y)+\Big(\sum_{i=1}^{n+k}B_{i1}\Big)\log t=\psi^{\mathrm{top}}_1(y)-B_{n+k+1,1}\log t,
\]
so $\langle\theta,\nabla\psi^{\mathrm{top}}_1(\theta)\rangle=-B_{n+k+1,1}$.
\item By Euler's identity, $\langle\theta,\nabla\Gamma^{\mathrm{top}}_i(\theta)\rangle=D_i\Gamma^{\mathrm{top}}_i(\theta)=0$ for $i\ge2$.
\end{enumerate}
Let $W$ be the matrix with columns $\theta,v_1,\dots,v_{k-1}$, so that $\det W>0$. Multiply the matrix
\[
\big(\nabla\psi^{\mathrm{top}}_1;\nabla\Gamma^{\mathrm{top}}_2;\dots;\nabla\Gamma^{\mathrm{top}}_k\big)
\]
on the right by $W$. By facts 2 and 3, its first column becomes $(-B_{n+k+1,1},0,\dots,0)^T$. Expanding along this column,
\[
\det\big(\nabla\psi^{\mathrm{top}}_1;\nabla\Gamma^{\mathrm{top}}_2;\dots;\nabla\Gamma^{\mathrm{top}}_k\big)(\theta)
=-\frac{B_{n+k+1,1}}{\det W}\,\det\big(D_{v_j}\Gamma^{\mathrm{top}}_i(\theta)\big).
\]
Multiplying by $P^{\mathrm{top}}(\theta)$ and using fact 1 gives the claim.
\end{proof}

\begin{theorem}[Facet-count formula, unbounded case]\label{thm:facet-count-unbounded}
Assume the following.
\begin{enumerate}
\item[(H)] $\Gamma$ does not vanish on $\partial \mathcal C$. The zeros of $(\Gamma_2,\dots,\Gamma_k)$ on $\partial \mathcal C$ lie in the relative interiors of facets and are transversal there.
\item[(H$_\infty$)] (i) $\Gamma^{\mathrm{top}}_1,\dots,\Gamma^{\mathrm{top}}_k$ have no common zero on $\Sigma$. (ii) The common zeros of $\Gamma^{\mathrm{top}}_2,\dots,\Gamma^{\mathrm{top}}_k$ on $\Sigma$ lie in $\Sigma^\circ$ and are nondegenerate as zeros on $S^{k-1}$.
\end{enumerate}
For $i=1,\dots,n+k$, let $N_i$ be the number of zeros of $(\Gamma_2,\dots,\Gamma_k)$ in the relative interior of the facet $F_i=\{p_i=0\}\cap\partial \mathcal C$, with $N_i=0$ if $\{p_i=0\}$ does not support a facet of $\mathcal C$. Let $N_{n+k+1}$ be the number of common zeros of $\Gamma^{\mathrm{top}}_2,\dots,\Gamma^{\mathrm{top}}_k$ on $\Sigma$.

Then there is $R_0>0$ such that $\Gamma$ does not vanish on $\mathcal C\setminus B(0,R_0)$. As a consequence, the degree
\[
d(\mathcal C)=\deg\left(\frac{\Gamma}{|\Gamma|}\colon\partial \mathcal C_R\longrightarrow S^{k-1}\right),
\]
with $\partial \mathcal C_R$ oriented as the boundary of $\mathcal C_R$, does not depend on $R\ge R_0$. When the zeros of $\Gamma$ in $\mathcal C$ are isolated, it equals the sum of their indices, and
\[
d(\mathcal C)=-\frac{\epsilon}{2}\sum_{i=1}^{n+k+1}\sign(B_{i1})\,N_i ,
\qquad\text{so that}\qquad
|d(\mathcal C)|=\frac12\,\Big|\sum_{i=1}^{n+k+1}\sign(B_{i1})\,N_i\Big| .
\]
\end{theorem}

The idea is to apply Theorem~\ref{thm:facet-count} to the bounded convex body $\mathcal C_R$ for $R$ large. Its boundary consists of the facets of $\mathcal C$ truncated by the ball and of a piece of the sphere of radius $R$, on which the Jacobians behave like their top-degree parts (Lemma~\ref{lem:top}). This piece contributes as one more facet, the hyperplane at infinity. The proof is given in Appendix~\ref{app:proofs}.

\section{Closures of the level sets $\mu_j$}\label{sec:mu}

In this section we study the closures of the level sets $\mu_j=\{f_1=\dots=f_{j-1}=1\}$ at the hyperplanes of the arrangement and at infinity. Theorem~\ref{thm:mu-boundary} gives a sign condition on the rows of $B$ which is necessary, and sufficient in the generic case, for a point of these hyperplanes to lie in the closure of $\mu_j$ (Corollary~\ref{cor:mu-generic}). This determines the faces of $\Delta$ which are reached by the ends of the curves $C_j$ (Section~\ref{sec:possol}). The proof of Theorem~\ref{thm:mu-boundary} is given in Appendix~\ref{app:proofs}.

We take $f_m=\prod_{i=1}^{n+k}|p_i|^{B_{im}}$, so that $\psi_m=\log f_m$. For $2\le j\le k$ let
\[
\mu_j=\big\{y\in\mathbb{R}^k\setminus H:\ f_1(y)=\dots=f_{j-1}(y)=1\big\},\qquad H=\bigcup_{i=1}^{n+k}\{p_i=0\}.
\]
For $i=1,\dots,n+k+1$ we write
\[
b_i^{(j-1)}=(B_{i1},\dots,B_{i,j-1})\in\mathbb{R}^{j-1}
\]
for the $i$-th row of $B$ truncated to its first $j-1$ columns. Since the columns of $B$ sum to zero, $b_{n+k+1}^{(j-1)}=-\sum_{i=1}^{n+k}b_i^{(j-1)}$.

\subsection{Projective setting}

Let $(y,y_{k+1})=(y_1,\dots,y_k,y_{k+1})\in\mathbb{R}^{k+1}$ and
\[
P_i(y,y_{k+1})=\sum_{l=1}^{k+1}D_{il}\,y_l,\qquad i=1,\dots,n+k+1 .
\]
The coefficient vector of $P_i$ is the $i$-th row $D_i$ of $D$. In particular $P_{n+k+1}=y_{k+1}$, and $P_i(y,1)=p_i(y)$ for all $i$.

Let $\mathbb{S}^k=(\mathbb{R}^{k+1}\setminus\{0\})/\mathbb{R}_{>0}$ be the oriented projective space. We identify $\mathbb{R}^k$ with $\{y_{k+1}>0\}$ via $y\mapsto[y:1]$. The closure of $\mathbb{R}^k$ in $\mathbb{S}^k$ is the closed hemisphere $\{y_{k+1}\ge0\}=\mathbb{R}^k\sqcup S^{k-1}_\infty$. A direction $\theta\in S^{k-1}$ corresponds to the point $[\theta:0]$, and $y^\nu\to[\theta:0]$ means $|y^\nu|\to\infty$ with $y^\nu/|y^\nu|\to\theta$.

Since $\sum_{i=1}^{n+k+1}B_{im}=0$, the function
\[
f_m(y,y_{k+1})=\prod_{i=1}^{n+k+1}|P_i(y,y_{k+1})|^{B_{im}}
\]
is homogeneous of degree $0$. It is therefore well defined on $\mathbb{S}^k\setminus\widehat H$, where $\widehat H=\bigcup_{i=1}^{n+k+1}\{P_i=0\}$, and it agrees with the previous $f_m$ on $\mathbb{R}^k$. Closures of subsets of $\mathbb{R}^k$ are taken in $\mathbb{S}^k$. Since $\mu_j$ is closed in $\mathbb{R}^k\setminus H$,
\[
\overline{\mu_j}\setminus\mu_j\subset\widehat H\cap\{y_{k+1}\ge0\}.
\]
For $\xi\in\mathbb{S}^k$ let $S(\xi)=\{i\le n+k+1:\ P_i(\xi)=0\}$. For $S\subset\{1,\dots,n+k+1\}$ let $L_S=\bigcap_{i\in S}\{P_i=0\}\subset\mathbb{S}^k$.

\begin{remark}
For a finite point $\xi=[y:1]$, $S(\xi)=\{i\le n+k:\ p_i(y)=0\}$. For a point at infinity $\xi=[\theta:0]$,
\[
S(\xi)=\{n+k+1\}\cup\{i\le n+k:\ p_i^{\mathrm{top}}(\theta)=\langle D'_i,\theta\rangle=0\}.
\]
So the hyperplane at infinity plays exactly the role of the other hyperplanes.
\end{remark}

\subsection{The characterization}

We use Stiemke's lemma.

\begin{lemma}[Stiemke]\label{lem:stiemke}
Let $M$ be a real $r\times s$ matrix. Exactly one of the following holds:
\begin{enumerate}
\item there is $\tau\in\mathbb{R}^s$ with $\tau_i>0$ for all $i$ and $M\tau=0$;
\item there is $\mu\in\mathbb{R}^r$ such that $M^T\mu\ge0$ and $M^T\mu\neq0$.
\end{enumerate}
\end{lemma}

For a finite family of vectors $\{b_i^{(j-1)}\}_{i\in S}$ in $\mathbb{R}^{j-1}$, condition~1 applied to the matrix with columns $b_i^{(j-1)}$ says that the convex cone generated by the $b_i^{(j-1)}$ is a linear subspace. We say that $\{b_i^{(j-1)}\}_{i\in S}$ \emph{positively spans} $\mathbb{R}^{j-1}$ if this cone is all of $\mathbb{R}^{j-1}$, that is, if condition~1 holds and $\rank\{b_i^{(j-1)}\}_{i\in S}=j-1$.

\begin{theorem}\label{thm:mu-boundary}
Let $\xi\in\widehat H\cap\{y_{k+1}\ge0\}$ and $S=S(\xi)$.
\begin{enumerate}
\item If $\xi\in\overline{\mu_j}$, then there is $\tau\in\mathbb{R}^S$ with $\tau_i>0$ for all $i\in S$ and $\sum_{i\in S}\tau_i\,b_i^{(j-1)}=0$. Equivalently, the cone generated by $\{b_i^{(j-1)}\}_{i\in S}$ is a linear subspace of $\mathbb{R}^{j-1}$.
\item Conversely, assume that the rows $D_i$, $i\in S$, are linearly independent in $\mathbb{R}^{k+1}$, and that $\{b_i^{(j-1)}\}_{i\in S}$ positively spans $\mathbb{R}^{j-1}$. Then $\xi\in\overline{\mu_j}$. More precisely, for every sufficiently small convex neighborhood $V$ of $\xi$, $\mu_j$ meets every connected component of $V\setminus\widehat H$ contained in $\{y_{k+1}>0\}$.
\end{enumerate}
\end{theorem}

The idea of the proof is the following. Near $\xi$, the equations of $\mu_j$ read $\sum_{i\in S}B_{im}\log|q_i|=c_m$, $m<j$, where the $q_i$, $i\in S$, are local coordinates vanishing at $\xi$ and the $c_m$ are bounded. Points of $\mu_j$ close to $\xi$ thus correspond to vectors $t=(-\log|q_i|)_{i\in S}$ with large positive coordinates, such that $\sum_{i\in S}t_i\,b_i^{(j-1)}$ is bounded. By Stiemke's lemma, this forces the vectors $b_i^{(j-1)}$, $i\in S$, to have a positive linear dependence. Conversely, when they positively span $\mathbb{R}^{j-1}$, such points are constructed by a fixed point argument.

\subsection{The generic case}

\begin{corollary}[Generic description]\label{cor:mu-generic}
Assume that any $k+1$ rows of $D$ are linearly independent, and that any $j-1$ of the vectors $b_1^{(j-1)},\dots,b_{n+k+1}^{(j-1)}$ are linearly independent. Then
\[
\overline{\mu_j}\setminus\mu_j=\bigcup_{S}\big(L_S\cap\{y_{k+1}\ge0\}\big),
\]
where $S$ runs over the subsets of $\{1,\dots,n+k+1\}$ with $|S|\le k$ such that $\{b_i^{(j-1)}\}_{i\in S}$ positively spans $\mathbb{R}^{j-1}$. It suffices to take the minimal such $S$, which satisfy $j\le|S|\le 2(j-1)$. In particular $\overline{\mu_j}\setminus\mu_j$ has dimension at most $k-j$, whereas $\mu_j$ has dimension $k-j+1$ at its regular points.
\end{corollary}

\begin{proof}
Let $\xi\in\widehat H\cap\{y_{k+1}\ge0\}$ and $S=S(\xi)$. Since $\xi$ is a common zero of the $P_i$, $i\in S$, and any $k+1$ rows of $D$ are independent, we have $|S|\le k$ and the rows $D_i$, $i\in S$, are independent. Under the second assumption, condition~1 of Theorem~\ref{thm:mu-boundary} forces $|S|\ge j$, because at most $j-1$ vectors are linearly independent and so admit no positive dependence. Then any $j-1$ of the $b_i^{(j-1)}$, $i\in S$, span $\mathbb{R}^{j-1}$, so condition~1 is equivalent to positive spanning. Theorem~\ref{thm:mu-boundary} thus gives: $\xi\in\overline{\mu_j}$ if and only if $\{b_i^{(j-1)}\}_{i\in S(\xi)}$ positively spans $\mathbb{R}^{j-1}$.

Positive spanning is monotone: if $\{b_i^{(j-1)}\}_{i\in S}$ positively spans, so does $\{b_i^{(j-1)}\}_{i\in S'}$ for $S'\supset S$. Every $\xi\in L_S$ has $S(\xi)\supset S$. This gives the displayed equality. A minimal positively spanning family in $\mathbb{R}^{j-1}$ has at least $j$ and at most $2(j-1)$ elements by Steinitz's theorem. Finally $L_S$ has dimension $k-|S|\le k-j$.
\end{proof}

\section{Restriction to the flats: reduced systems}\label{sec:restriction}

In this section we show that the restrictions of the Jacobians to a flat $L_S$ of the arrangement, finite or at infinity, are again the Jacobians of a system of the same type, with $k$ replaced by $k-|S|$. This is the key to the recursion of Section~\ref{sec:recursive}. We first introduce coordinates on $L_S$ whose coefficients are ratios of maximal minors of $D$ (Lemma~\ref{lem:minor-coordinates}), and we define the reduced system along $S$ (Definition~\ref{def:reduced-system} and Proposition~\ref{prop:reduced-format}). The restriction theorems are Theorem~\ref{thm:restriction-LS} for finite flats and Theorem~\ref{thm:restriction-LS-infinity} for flats at infinity. Their proofs are computations with the Cauchy--Binet formula of Proposition~\ref{prop:gammak}, and they are given in Appendix~\ref{app:proofs}.

\subsection{Coordinates on a flat $L_S$ given by minors of $D$}\label{subsec:minor-coordinates}

We now describe the restrictions of the Jacobians to a flat $L_S$. We use coordinates on $L_S$ whose coefficients are, up to a common factor, maximal minors of $D$.

For a set $U\subset\{1,\dots,n+k+1\}$ with $|U|=k+1$, write $D_U$ for the square matrix with rows $D_u$, $u\in U$, in a specified order. For $u\in U$ and any index $i$, write $D_U^{(u\leftarrow i)}$ for the matrix obtained from $D_U$ by replacing its row $D_u$ by $D_i$, in the same position.

\begin{lemma}\label{lem:minor-coordinates}
Let $S\subset\{1,\dots,n+k+1\}$ with $|S|=j$, and let $T\subset\{1,\dots,n+k+1\}\setminus S$ with $|T|=k+1-j$, such that $D_{S\cup T}$ is invertible. Here the rows of $S$ come first and then those of $T$. For every index $i$ put
\[
c_{it}=\frac{\det D_{S\cup T}^{(t\leftarrow i)}}{\det D_{S\cup T}},\qquad t\in T .
\]
Then, as linear forms on $L_S=\bigcap_{s\in S}\{P_s=0\}$,
\[
P_i\big|_{L_S}=\sum_{t\in T}c_{it}\,P_t\big|_{L_S}\qquad\text{for every }i\notin S .
\]
\end{lemma}

\begin{proof}
The rows $D_r$, $r\in S\cup T$, form a basis of $\mathbb{R}^{k+1}$. By Cramer's rule, $D_i=\sum_{r\in S\cup T}c_{ir}D_r$ with $c_{ir}=\det D_{S\cup T}^{(r\leftarrow i)}/\det D_{S\cup T}$. Since $P_s=0$ on $L_S$ for $s\in S$, the claim follows.
\end{proof}

From now on we fix $S$ and $T$ as in Lemma~\ref{lem:minor-coordinates}, with $1\le j\le k-1$, and an index $t_0\in T$. We put $T'=T\setminus\{t_0\}$ and $\tilde R=\{1,\dots,n+k+1\}\setminus S$, so that $|T'|=k-j$ and $|\tilde R|=n+(k-j)+1$. On the open subset $\{P_{t_0}>0\}$ of $L_S$, the functions
\[
w_t=\frac{P_t}{P_{t_0}},\qquad t\in T',
\]
are $k-j$ affine coordinates. We write $\varphi\colon\mathbb{R}^{k-j}\to L_S\cap\{P_{t_0}>0\}$ for the inverse map. By Lemma~\ref{lem:minor-coordinates},
\begin{equation}\label{eq:minor-coordinates}
\frac{P_i}{P_{t_0}}\circ\varphi\,(w)=c_{it_0}+\sum_{t\in T'}c_{it}\,w_t\qquad(i\in\tilde R).
\end{equation}
This function equals $w_t$ for $i=t\in T'$ and equals $1$ for $i=t_0$.
\begin{itemize}
\item If $n+k+1\notin S$, we may take $n+k+1\in T$ and $t_0=n+k+1$. Then $\varphi$ is an affine parametrization of the affine subspace $L_S\cap\mathbb{R}^k=\{p_s=0,\ s\in S\}$, and $w_t=p_t\circ\varphi$.
\item If $n+k+1\in S$, then $L_S$ lies in the hyperplane at infinity, and $L_S\cap\{P_{t_0}>0\}$ is identified with the affine subspace $\{\theta\in V:\ \ell_{t_0}(\theta)=1\}$ of the linear subspace $V\subset\mathbb{R}^k$ defined in Subsection~\ref{subsec:restriction-infinity} below.
\end{itemize}

\subsection{The reduced system}\label{subsec:reduced-system}

Split the rows of $B$ indexed by $S$, in increasing order, into two blocks:
\[
M_S=(B_{s l})_{s\in S,\ 1\le l\le j},\qquad Q=(B_{s l})_{s\in S,\ j<l\le k}.
\]
We assume $\det M_S\neq0$.

\begin{remark}
If $\{b_s^{(j-1)}\}_{s\in S}$ positively spans $\mathbb{R}^{j-1}$, with positive dependence $\sum_{s\in S}\tau_s\, b_s^{(j-1)}=0$, then the first $j-1$ columns of $M_S$ have rank $j-1$ and their left kernel is spanned by $\tau$. In that case $\det M_S\neq0$ if and only if $\sum_{s\in S}\tau_s B_{sj}\neq0$.
\end{remark}

\begin{definition}\label{def:reduced-system}
Order $\tilde R$ so that $t_0$ comes last, and set $\tilde n=n$ and $\tilde k=k-j$. The \emph{reduced system along $S$}, with chart $t_0$, is the pair $(\tilde D,\tilde B)$ defined as follows.
\begin{enumerate}
\item The matrix $\tilde D$ has rows indexed by $\tilde R$ and $\tilde k+1$ columns, indexed by $T'$ (in increasing order) and then $t_0$. Its entries are the coefficients $\tilde D_{it}=c_{it}$ of~\eqref{eq:minor-coordinates}.
\item The matrix $\tilde B$ has rows indexed by $\tilde R$ and $\tilde k$ columns, indexed by $m=j+1,\dots,k$. Its entries are
\[
\tilde B_{im}=B_{im}-\sum_{l,l'=1}^jB_{il}\,(M_S^{-1})_{ll'}\,Q_{l'm},\qquad i\in\tilde R .
\]
\end{enumerate}
\end{definition}

\begin{proposition}\label{prop:reduced-format}
The pair $(\tilde D,\tilde B)$ is a pair of type $(\tilde n,\tilde k)$ (Definition~\ref{def:setting}):
\begin{enumerate}
\item the last row of $\tilde D$, indexed by $t_0$, is $(0,\dots,0,1)$, and for $t\in T'$ the row $\tilde D_t$ is the unit vector of the column $t$;
\item the columns of $\tilde B$ sum to zero;
\item for $i\in\tilde R$,
\[
\tilde B_{i,j+1}=\frac{\det B_{S\cup\{i\},\{1,\dots,j+1\}}}{\det M_S},
\]
where the rows of $B_{S\cup\{i\},\{1,\dots,j+1\}}$ are those of $S$ in increasing order followed by the row $i$.
\end{enumerate}
Its affine functions are the functions $\tilde p_i(w)=c_{it_0}+\sum_{t\in T'}c_{it}\,w_t$ of~\eqref{eq:minor-coordinates}, $i\in\tilde R\setminus\{t_0\}$, so that $\tilde p_t=w_t$ for $t\in T'$.
\end{proposition}

\begin{proof}
Assertion~1 holds because $c_{it}=\delta_{it}$ for $i,t\in T$. For assertion~2, let
\begin{equation}\label{eq:G}
G=\begin{pmatrix} I_j & -M_S^{-1}Q\\ 0 & I_{k-j}\end{pmatrix}.
\end{equation}
The rows of $BG$ indexed by $S$ are those of $\big[\,M_S\,\big|\,0\,\big]$, and its rows indexed by $\tilde R$ in the columns $m>j$ are those of $\tilde B$. The column sums of $BG$ are those of $B$ multiplied by $G$, hence zero, and assertion~2 follows. For assertion~3, $\tilde B_{i,j+1}$ is the Schur complement of $N$ in the $(j+1)\times(j+1)$ matrix $B_{S\cup\{i\},\{1,\dots,j+1\}}$, hence equals the ratio of the two determinants.
\end{proof}

We write $\tilde\psi_m=\sum_{i\in\tilde R\setminus\{t_0\}}\tilde B_{im}\log|\tilde p_i|$ ($m=j+1,\dots,k$), $\tilde P=\prod_{i\in\tilde R\setminus\{t_0\}}\tilde p_i$, and $\tilde\Gamma_k,\dots,\tilde\Gamma_{j+1}$ for the Jacobians of the reduced system, defined from $\tilde P$ and $\tilde\psi_{j+1},\dots,\tilde\psi_k$ as in Section~\ref{sec:KR}. In the numbering of Section~\ref{sec:KR} these are the functions $\Gamma_{\tilde k},\dots,\Gamma_1$ of $(\tilde D,\tilde B)$, and $\deg\tilde\Gamma_m\le2^{k-m}n$ by Proposition~\ref{prop:gammak}(2). We also write $\tilde\Gamma=(\tilde\Gamma_{j+1},\dots,\tilde\Gamma_k)$.

\subsection{Restriction to a finite flat}\label{subsec:restriction-finite}

In this subsection $n+k+1\notin S$ and $t_0=n+k+1$, so that $L_S$ also denotes the affine subspace $\{p_s=0,\ s\in S\}$ of $\mathbb{R}^k$, of dimension $k-j$, and $\varphi$ is the affine parametrization of $L_S$ given by the coordinates $w_t=p_t$, $t\in T'$.

\begin{theorem}\label{thm:restriction-LS}
Put
\[
\lambda=\det M_S\cdot\det\big(D_S;\,D_{T'};\,D_{n+k+1}\big)\neq0,
\]
where the rows of each block are taken in increasing order. Then the following identities between polynomials on $\mathbb{R}^{k-j}$ hold:
\begin{enumerate}
\item $\Gamma_m\circ\varphi=\lambda^{2^{k-m}}\,\tilde\Gamma_m$ for $m=j+1,\dots,k$;
\item $\Gamma_j\circ\varphi=\lambda^{2^{k-j}}\,\tilde P\cdot\det\dfrac{\partial(\tilde\Gamma_{j+1},\dots,\tilde\Gamma_k)}{\partial(w_1,\dots,w_{k-j})}$;
\item $\Gamma_m\circ\varphi=0$ for $m=1,\dots,j-1$.
\end{enumerate}
In particular,
\[
\{\Gamma_{j+1}=\dots=\Gamma_k=0\}\cap L_S=\varphi\big(\{\tilde\Gamma_{j+1}=\dots=\tilde\Gamma_k=0\}\big),
\]
so this intersection is defined by $k-j$ equations in $k-j$ variables.
\end{theorem}

The proof, given in Appendix~\ref{app:proofs}, applies the Cauchy--Binet formula of Proposition~\ref{prop:gammak}(1) to $\Gamma_m$. On $L_S$, only the terms indexed by sets $I$ containing $S$ survive. Their minors of $B$ factor as $\det M_S$ times minors of $\tilde B$, and their minors of $D$ factor as a constant times minors of the matrix of the coordinates of Lemma~\ref{lem:minor-coordinates}. This gives the Cauchy--Binet formula for the reduced system.

\subsection{Restriction to a flat at infinity}\label{subsec:restriction-infinity}

In this subsection $n+k+1\in S$, so that $L_S$ is contained in the hyperplane at infinity $\{y_{k+1}=0\}$. Write $S=S'\cup\{n+k+1\}$ and $R=\{1,\dots,n+k\}\setminus S'$, so that $t_0\in T\subset R$. For $i\le n+k$ let $\ell_i(y)=p_i^{\mathrm{top}}(y)=\langle D'_i,y\rangle$, and put
\[
V=\bigcap_{s\in S'}\ker\ell_s\subset\mathbb{R}^k,\qquad \dim V=k-j+1 .
\]
Then $L_S=\{[\theta:0]:\ \theta\in V\setminus\{0\}\}$, the chart $L_S\cap\{P_{t_0}>0\}$ is the affine subspace $A_{t_0}=\{\theta\in V:\ \ell_{t_0}(\theta)=1\}$, and $\varphi\colon\mathbb{R}^{k-j}\to A_{t_0}$ is the parametrization given by the coordinates $w_t=\ell_t$, $t\in T'$. By~\eqref{eq:minor-coordinates}, $\tilde p_i=\ell_i\circ\varphi$ for $i\in R\setminus\{t_0\}$. Let $\Gamma_m^h(y,y_{k+1})=y_{k+1}^{D_m}\,\Gamma_m(y/y_{k+1})$ be the homogenization of $\Gamma_m$ of degree $D_m=2^{k-m}n$. Then $\Gamma_m^h(\theta,0)=\Gamma_m^{\mathrm{top}}(\theta)$, so
\[
\{\Gamma^h_{j+1}=\dots=\Gamma^h_k=0\}\cap L_S=\big\{[\theta:0]:\ \theta\in V\setminus\{0\},\ \Gamma^{\mathrm{top}}_{j+1}(\theta)=\dots=\Gamma^{\mathrm{top}}_k(\theta)=0\big\}.
\]
Let $I_m\subset\mathbb{R}[w_1,\dots,w_{k-j}]$ be the ideal generated by $\tilde\Gamma_m,\dots,\tilde\Gamma_k$, with $I_{k+1}=0$.

\begin{theorem}\label{thm:restriction-LS-infinity}
Put
\[
\lambda=(-1)^{k-j+1}\,\det M_S\cdot\det\big(D_{S'};\,D_{T'};\,D_{t_0};\,D_{n+k+1}\big)\neq0,
\]
where the rows of each block are taken in increasing order. Then:
\begin{enumerate}
\item $\Gamma^{\mathrm{top}}_k\circ\varphi=\lambda\,\tilde\Gamma_k$;
\item $\Gamma^{\mathrm{top}}_m\circ\varphi\equiv\lambda^{2^{k-m}}\,\tilde\Gamma_m\pmod{I_{m+1}}$ for $m=j+1,\dots,k$;
\item $\Gamma^{\mathrm{top}}_j\circ\varphi\equiv\lambda^{2^{k-j}}\,\tilde P\cdot\det\dfrac{\partial(\tilde\Gamma_{j+1},\dots,\tilde\Gamma_k)}{\partial(w_1,\dots,w_{k-j})}\pmod{I_{j+1}}$;
\item $\Gamma^{\mathrm{top}}_m\circ\varphi=0$ for $m\le j-2$, and $\Gamma^{\mathrm{top}}_{j-1}\circ\varphi$ belongs to the ideal generated by $\Gamma^{\mathrm{top}}_j\circ\varphi,\dots,\Gamma^{\mathrm{top}}_k\circ\varphi$.
\end{enumerate}
It follows that, for $m=j+1,\dots,k$, the ideal generated by $\Gamma^{\mathrm{top}}_m\circ\varphi,\dots,\Gamma^{\mathrm{top}}_k\circ\varphi$ equals $I_m$. In particular
\[
\{\theta\in A_{t_0}:\ \Gamma^{\mathrm{top}}_{j+1}(\theta)=\dots=\Gamma^{\mathrm{top}}_k(\theta)=0\}=\varphi\big(\{\tilde\Gamma_{j+1}=\dots=\tilde\Gamma_k=0\}\big),
\]
which is again defined by $k-j$ equations in $k-j$ variables.
\end{theorem}

The proof, given in Appendix~\ref{app:proofs}, follows the same lines as for finite flats, using the top-degree parts of the Jacobians (Lemma~\ref{lem:top}). The difference is that the coordinates on $L_S$ are now homogeneous, and Euler's identity produces additional terms, which are multiples of the later Jacobians. This is why the identities only hold modulo the ideals $I_{m+1}$.

\section{The reduced system on a flat $L_S$ and its boundary degree}\label{sec:reduced}

In this section we combine the two previous tools: the boundary degree of Section~\ref{sec:boundary-degree} and the reduced systems of Section~\ref{sec:restriction}. Applying the facet-count formula to the reduced system along a flat $L_S$, we obtain lower bounds for the numbers of zeros of the Jacobians in the chambers of $L_S$, and, on the lines of the arrangement, the bounds which start the recursion.

We fix $S$, $T$ and $t_0$ as in Subsection~\ref{subsec:minor-coordinates}, with $\det M_S\neq0$, and we consider the reduced system $(\tilde D,\tilde B)$ of Definition~\ref{def:reduced-system}. By Theorems~\ref{thm:restriction-LS} and~\ref{thm:restriction-LS-infinity}, the common zeros on $L_S$ of the original Jacobians are the common zeros of the reduced Jacobians $\tilde\Gamma=(\tilde\Gamma_{j+1},\dots,\tilde\Gamma_k)$. Since $(\tilde D,\tilde B)$ is again a pair of type $(n,k-j)$, the results of Section~\ref{sec:boundary-degree} apply to it. We obtain:
\begin{itemize}
\item a facet-count formula for the boundary degree $\tilde d(\tilde{\mathcal C})$ of a chamber $\tilde{\mathcal C}$ of $L_S$ (Theorem~\ref{thm:facet-count-LS}), and a lower bound for it in terms of bounds on the numbers of zeros on the facets (Proposition~\ref{prop:degree-lower-bound});
\item the lower bound $|\tilde d(\tilde{\mathcal C})|$ for the number of zeros of the original Jacobians on $\tilde{\mathcal C}$ (Proposition~\ref{prop:lower-bound-LS});
\item in dimension one, explicit bounds for the number of zeros of $\Gamma_k$ on each arc of a line of the arrangement (Proposition~\ref{prop:line-zeros}). They start the recursion of Section~\ref{sec:recursive}.
\end{itemize}

\subsection{A facet-count formula on $L_S$}

Let $\tilde{\mathcal C}$ be a connected component of $\mathbb{R}^{k-j}\setminus\bigcup_{i\in\tilde R\setminus\{t_0\}}\{\tilde p_i=0\}$, possibly unbounded, and let $\tilde\epsilon=\prod_i\sign\tilde p_i|_{\tilde{\mathcal C}}$. As in Subsection~\ref{subsec:unbounded}, let $\tilde\Sigma$ be the trace on $S^{k-j-1}$ of the recession cone of $\tilde{\mathcal C}$, and $\tilde{\mathcal C}_R=\tilde{\mathcal C}\cap B(0,R)$.

\begin{theorem}\label{thm:facet-count-LS}
Assume that $(\tilde D,\tilde B)$ and $\tilde{\mathcal C}$ satisfy the hypotheses (H) and (H$_\infty$) of Theorem~\ref{thm:facet-count-unbounded}. For $i\in\tilde R\setminus\{t_0\}$, let $\tilde N_i$ be the number of zeros of $(\tilde\Gamma_{j+2},\dots,\tilde\Gamma_k)$ in the relative interior of the facet $\{\tilde p_i=0\}\cap\partial\tilde{\mathcal C}$ ($\tilde N_i=0$ if there is no such facet), and let $\tilde N_{t_0}$ be the number of common zeros of $\tilde\Gamma^{\mathrm{top}}_{j+2},\dots,\tilde\Gamma^{\mathrm{top}}_k$ on $\tilde\Sigma$. Then the degree $\tilde d(\tilde{\mathcal C})$ of $\tilde\Gamma/|\tilde\Gamma|$ on $\partial\tilde{\mathcal C}_R$ does not depend on $R$ large and, when the zeros of $\tilde\Gamma$ in $\tilde{\mathcal C}$ are isolated, it is the sum of their indices. Moreover
\[
\tilde d(\tilde{\mathcal C})=-\frac{\tilde\epsilon}{2}\sum_{i\in\tilde R}\sign\big(\tilde B_{i,j+1}\big)\,\tilde N_i
=-\frac{\tilde\epsilon\,\sign(\det M_S)}{2}\sum_{i\in\tilde R}\sign\Big(\det B_{S\cup\{i\},\{1,\dots,j+1\}}\Big)\,\tilde N_i .
\]
In particular,
\[
|\tilde d(\tilde{\mathcal C})|=\frac12\,\Big|\sum_{i\in\tilde R}\sign\Big(\det B_{S\cup\{i\},\{1,\dots,j+1\}}\Big)\,\tilde N_i\Big| ,
\]
which depends neither on $\tilde\epsilon$ nor on $\det M_S$.
\end{theorem}

\begin{proof}
By Proposition~\ref{prop:reduced-format}, $(\tilde D,\tilde B)$ is a pair of type $(n,k-j)$, with the row $t_0$ playing the role of the row $n+k+1$. Its Jacobians are $\tilde\Gamma_{j+1},\dots,\tilde\Gamma_k$, and the first column of $\tilde B$ is the column $m=j+1$. The first formula is Theorem~\ref{thm:facet-count-unbounded} applied to this pair, and the second follows from Proposition~\ref{prop:reduced-format}(3). The last one follows by taking absolute values.
\end{proof}

Since $\tilde d(\tilde{\mathcal C})$ is a signed count of the zeros on the facets, bounds on these numbers of zeros give a lower bound for $|\tilde d(\tilde{\mathcal C})|$.

\begin{proposition}\label{prop:degree-lower-bound}
Assume the hypotheses of Theorem~\ref{thm:facet-count-LS}. Suppose that for every $i\in\tilde R$ we are given integers $l_i\le\tilde N_i\le L_i$, with $l_i=L_i=0$ if $\{\tilde p_i=0\}$ does not support a facet of $\tilde{\mathcal C}$ (for $i\neq t_0$), respectively if $\tilde\Sigma=\emptyset$ (for $i=t_0$). Let
\[
I_\pm=\{i\in\tilde R:\ \sign\tilde B_{i,j+1}=\pm1\},
\]
and put
\[
a=\sum_{i\in I_+}l_i-\sum_{i\in I_-}L_i,\qquad b=\sum_{i\in I_+}L_i-\sum_{i\in I_-}l_i .
\]
Then
\[
|\tilde d(\tilde{\mathcal C})|\ \ge\ \max\Big(0,\ \Big\lceil\frac a2\Big\rceil,\ \Big\lceil-\frac b2\Big\rceil\Big).
\]
Moreover, if $a>0$ then $\sign\tilde d(\tilde{\mathcal C})=-\tilde\epsilon$, and if $b<0$ then $\sign\tilde d(\tilde{\mathcal C})=\tilde\epsilon$.
\end{proposition}

\begin{proof}
Let $X=\sum_{i\in I_+}\tilde N_i-\sum_{i\in I_-}\tilde N_i$. By Theorem~\ref{thm:facet-count-LS}, $\tilde d(\tilde{\mathcal C})=-\tilde\epsilon X/2$, so $X$ is an even integer, and $a\le X\le b$. If $a>0$, then $X\ge a$, hence $X\ge2\lceil a/2\rceil$ and $|\tilde d(\tilde{\mathcal C})|\ge\lceil a/2\rceil$, with $\sign\tilde d(\tilde{\mathcal C})=-\tilde\epsilon$. The case $b<0$ is symmetric.
\end{proof}

The same statement holds for the original system, with $d(\mathcal C)$, $\epsilon$, $N_i$ and $\sign B_{i1}$ ($1\le i\le n+k+1$), by Theorem~\ref{thm:facet-count-unbounded}.

\subsection{Zeros of the original Jacobians on $L_S$}

If $n+k+1\notin S$, we take $t_0=n+k+1$ and put $\Gamma^S=(\Gamma_{j+1}\circ\varphi,\dots,\Gamma_k\circ\varphi)$. If $n+k+1\in S$, we put $\Gamma^S=(\Gamma^{\mathrm{top}}_{j+1}\circ\varphi,\dots,\Gamma^{\mathrm{top}}_k\circ\varphi)$. Here $\varphi$ is the parametrization of Subsection~\ref{subsec:minor-coordinates}.

\begin{proposition}\label{prop:lower-bound-LS}
Assume the hypotheses of Theorem~\ref{thm:facet-count-LS}, and let $\mathcal Z$ be the set of zeros of $\Gamma^S$ in $\tilde{\mathcal C}$.
\begin{enumerate}
\item $\mathcal Z$ is compact, and it is also the set of zeros of $\tilde\Gamma$ in $\tilde{\mathcal C}$. A point of $\mathcal Z$ is a nondegenerate zero of $\Gamma^S$ if and only if it is a nondegenerate zero of $\tilde\Gamma$.
\item If every point of $\mathcal Z$ is nondegenerate, then
\[
\#\mathcal Z\ \ge\ |\tilde d(\tilde{\mathcal C})|\qquad\text{and}\qquad \#\mathcal Z\equiv\tilde d(\tilde{\mathcal C})\pmod 2 .
\]
\end{enumerate}
\end{proposition}

\begin{proof}
By Theorems~\ref{thm:restriction-LS} and~\ref{thm:restriction-LS-infinity}, $\Gamma^S=T\tilde\Gamma$, where $T(w)$ is an upper triangular matrix with polynomial entries and constant diagonal entries $\lambda^{2^{k-m}}\neq0$; it is diagonal if $n+k+1\notin S$. So $\Gamma^S$ and $\tilde\Gamma$ have the same zeros, and at a zero $D_w\Gamma^S=T\,D_w\tilde\Gamma$, which gives the statement on nondegeneracy. By (H) and Theorem~\ref{thm:facet-count-LS}, $\tilde\Gamma$ does not vanish on $\partial\tilde{\mathcal C}$ nor on $\tilde{\mathcal C}\setminus B(0,R_0)$, so $\mathcal Z$ is compact. If its points are nondegenerate, they are isolated, hence finitely many, and each has index $\pm1$. Since $\tilde d(\tilde{\mathcal C})$ is the sum of these indices (Theorem~\ref{thm:facet-count-LS}), assertion~2 follows.
\end{proof}

Since $\varphi$ is an affine isomorphism onto its image, $\varphi(\mathcal Z)$ is the set of common zeros of $\Gamma_{j+1},\dots,\Gamma_k$ (respectively of $\Gamma^{\mathrm{top}}_{j+1},\dots,\Gamma^{\mathrm{top}}_k$) on the chamber $\varphi(\tilde{\mathcal C})$ of $L_S$, and nondegeneracy means that the restriction of these functions to $L_S$ has invertible tangential differential there.

\subsection{Zeros of $\Gamma_k$ on a line}\label{subsec:line-zeros}

We now take $j=k-1$, so that the flat is a line. Let $S'\subset\{1,\dots,n+k+1\}$ with $|S'|=k-1$ and $R'=\{1,\dots,n+k+1\}\setminus S'$, so $|R'|=n+2$. For $l\in R'$ put
\[
\sigma_l=\sign\det B_{S'\cup\{l\},\{1,\dots,k\}},
\]
where the rows are those of $S'$ in increasing order followed by the row $l$. For $k=2$ and $S'=\{i\}$, $\sigma_l=\sign(B_{i1}B_{l2}-B_{i2}B_{l1})$, with $B_{n+3,m}=-\sum_{l\le n+2}B_{lm}$. We assume that the rows $D_s$ ($s\in S'$), $D_l$, $D_{l'}$ are linearly independent for all $l\neq l'$ in $R'$, that $\det M'\neq0$ with $M'=B_{S',\{1,\dots,k-1\}}$, and that $\sigma_l\neq0$ for all $l\in R'$. These assumptions hold under (U$_D$) and (U$_B$).

The reduced system along $S'$ lives on a line with coordinate $w$. It has the $n+1$ affine functions $\tilde p_l$, $l\in R'\setminus\{t_0\}$, and the single exponent column
\[
\gamma_l=\tilde B_{lk}=\frac{\det B_{S'\cup\{l\},\{1,\dots,k\}}}{\det M'},\qquad l\in R',\qquad \sum_{l\in R'}\gamma_l=0,
\]
by Proposition~\ref{prop:reduced-format}, so that $\sign\gamma_l=\sign(\det M')\,\sigma_l$. On $\mathbb{P}^1=\mathbb{R}\cup\{\infty\}$, let $r_l$ be the root of $\tilde p_l$ for $l\neq t_0$, and $r_{t_0}=\infty$. These $n+2$ points are pairwise distinct and cut $\mathbb{P}^1$ into $n+2$ open arcs. For an arc $\alpha$ with endpoints $r_a$ and $r_b$, let $N_\alpha$ be the number of zeros of $\tilde\Gamma_k$ in $\alpha$, counted with multiplicity, and put $l_\alpha=1$ if $\sigma_a\sigma_b>0$ and $l_\alpha=0$ if $\sigma_a\sigma_b<0$. Let $W$ be the number of sign changes of the cyclic sequence $(\sigma_l)$, the points $r_l$ being taken in their cyclic order on $\mathbb{P}^1$; it is even.

By Theorems~\ref{thm:restriction-LS} and~\ref{thm:restriction-LS-infinity}(1), $N_\alpha$ is also the number of zeros of $\Gamma_k$ (respectively of $\Gamma^{\mathrm{top}}_k$ if $n+k+1\in S'$) on the corresponding arc of the line $L_{S'}$, and these arcs are the edges of the chambers of the arrangement on $L_{S'}$. For instance, for $k=2$, $N_\alpha$ is the number $N_i$ of Theorem~\ref{thm:facet-count-unbounded} for an edge $\alpha$ on $\{p_i=0\}$, and $N_{n+3}$ for $S'=\{n+3\}$.

\begin{proposition}\label{prop:line-zeros}
Under the above assumptions:
\begin{enumerate}
\item $\tilde\Gamma_k$ is a polynomial of degree exactly $n$ with no zero at the points $r_l$. In particular
\[
\sum_{\alpha}N_\alpha\le n\qquad\text{and}\qquad\sum_\alpha N_\alpha\equiv n\pmod 2,
\]
where $\alpha$ runs over the $n+2$ arcs.
\item \emph{Parity.} For every arc $\alpha$ with endpoints $r_a$ and $r_b$, $N_\alpha$ is odd if $\sigma_a=\sigma_b$ and even if $\sigma_a\neq\sigma_b$; that is, $N_\alpha\equiv l_\alpha\pmod 2$. In particular $N_\alpha\ge1$ if $\sigma_a=\sigma_b$, and so $N_\alpha\ge l_\alpha$.
\item \emph{Upper bound.} For every arc $\alpha$,
\[
N_\alpha\ \le\ L_\alpha:=n-\sum_{\alpha'\neq\alpha}l_{\alpha'}\ =\ W-2+l_\alpha ,
\]
and $L_\alpha\equiv l_\alpha\pmod 2$.
\end{enumerate}
\end{proposition}

On a line of the arrangement, the reduced system has $\tilde k=1$, and $\tilde\Gamma_k$ is, up to a nonzero factor, the derivative of a function of the form $\sum_l\gamma_l\log|\tilde p_l|$. So the statement is a Descartes' rule of signs for circuits on each arc. The proof, which follows \cite{BD}, is given in Appendix~\ref{app:proofs}.

\section{Recursive bounds on all flats and chambers}\label{sec:recursive}

In this section we combine the previous results into a recursive procedure. For every flat of the arrangement and every chamber of that flat, it produces a lower and an upper bound for the number of common zeros of the corresponding Jacobians. The recursion runs over the dimension of the flat. In dimension one, the bounds are those of Proposition~\ref{prop:line-zeros}, namely Descartes' rule of signs for circuits \cite{BD}. The proofs of the auxiliary lemmas of this section are given in Appendix~\ref{app:proofs}.

\subsection{Setting}

We work in the real projective space $\mathbb{RP}^k$ with homogeneous coordinates $(y,y_{k+1})$. For $i\in[n+k+1]:=\{1,\dots,n+k+1\}$ let $H_i=\{P_i=0\}$, with $P_i$ as in Section~\ref{sec:mu}. So $H_{n+k+1}=\{y_{k+1}=0\}$ is the hyperplane at infinity, and $\mathbb{RP}^k\setminus H_{n+k+1}=\mathbb{R}^k$. For $S\subset[n+k+1]$ put
\[
\Lambda_S=\bigcap_{i\in S}H_i\subset\mathbb{RP}^k .
\]
This is the image in $\mathbb{RP}^k$ of the flat $L_S\subset\mathbb{S}^k$ of Section~\ref{sec:mu}.

We assume:
\begin{itemize}
\item[(U$_D$)] any $k+1$ rows of $D$ are linearly independent;
\item[(U$_B$)] for every $U\subset[n+k+1]$ with $1\le|U|\le k$, $\det B_{U,\{1,\dots,|U|\}}\neq0$.
\end{itemize}
In (U$_B$) the rows of $B$ include the row $n+k+1$ fixed by the zero column sums. Both conditions hold for generic $D$ and $B$.

Let $|S|=j\le k$. By (U$_D$), $\Lambda_S$ is a projective subspace of dimension $d=k-j$. The traces $H_i\cap\Lambda_S$, $i\notin S$, are $n+d+1$ hyperplanes of $\Lambda_S$, any $d+1$ of which have empty intersection. A \emph{chamber} of $\Lambda_S$ is a connected component of $\Lambda_S\setminus\bigcup_{i\notin S}H_i$. We write $\mathrm{Ch}(S)$ for the set of chambers of $\Lambda_S$.
\begin{itemize}
\item If $|S|=k$, then $\Lambda_S$ is a point $v_S$, which lies on no $H_i$ with $i\notin S$, and $\mathrm{Ch}(S)=\{\{v_S\}\}$.
\item If $S=\emptyset$, the chambers are the connected components of $\mathbb{R}^k\setminus\bigcup_{i\le n+k}\{p_i=0\}$. More generally, if $n+k+1\notin S$, the chambers of $\Lambda_S$ are the chambers of the affine flat $L_S\cap\mathbb{R}^k$. If $n+k+1\in S$, they are the cells at infinity, each one identified with its antipode.
\item Let $d\ge1$ and $\mathcal C\in\mathrm{Ch}(S)$. Let $I(\mathcal C)$ be the set of $i\notin S$ such that $H_i\cap\overline{\mathcal C}$ has dimension $d-1$. For $i\in I(\mathcal C)$, the relative interior of $H_i\cap\overline{\mathcal C}$ is a chamber
\[
F_i(\mathcal C)\in\mathrm{Ch}(S\cup\{i\}),
\]
the \emph{facet} of $\mathcal C$ on $H_i$. In any affine chart $\{P_{t_0}\neq0\}$ with $t_0\notin S$, the chamber $\mathcal C$ is a convex polyhedron, possibly unbounded. Its facets are the $F_i(\mathcal C)$ with $i\neq t_0$, together with the facet at infinity $F_{t_0}(\mathcal C)$ when $t_0\in I(\mathcal C)$.
\end{itemize}

\subsection{Reduced systems, chart independence and the numbers $Z(S,\mathcal C)$}

For $1\le|S|=j\le k-1$ we call \emph{reduced system along $S$} any system $(\tilde D,\tilde B)$ as in Definition~\ref{def:reduced-system}, for any choice of $T$ and of $t_0\in T$. For $S=\emptyset$, the reduced system is $(D,B)$ itself, written in the affine chart $\{P_{t_0}\neq0\}$ for some $t_0\in[n+k+1]$ (the chart $t_0=n+k+1$ being the original one); in this case we put $\det M_S=1$, so that $\sigma_\emptyset(i)=\sign B_{i1}$. By Proposition~\ref{prop:reduced-format}, it is a pair of type $(\tilde n,\tilde k)$ (Definition~\ref{def:setting}), with $\tilde n=n$ and $\tilde k=d$. Its hyperplanes are the traces $H_i\cap\Lambda_S$ ($i\notin S$), written in coordinates of the affine chart $\{P_{t_0}\neq0\}$ of $\Lambda_S$. Its Jacobians are $\tilde\Gamma_{j+1},\dots,\tilde\Gamma_k$.

\begin{lemma}[Chart independence]\label{lem:chart-independence}
Let $(D,B)$ be a pair of type $(n,k)$ (Definition~\ref{def:setting}). Let $(D',B')$ be another one, obtained from the same linear forms $P_i$ and the same exponents by one of the following operations: another choice of the hyperplane at infinity (and reordering of the rows accordingly), a linear change of the homogeneous coordinates, or multiplication of the forms $P_i$ by nonzero constants. Let $\tau$ be the induced coordinate change between the two affine charts. Then, on the intersection of the two charts, off the arrangement, there are smooth functions $g_m$, nowhere zero, and smooth functions $h_{mm'}$ such that
\[
\Gamma'_m\circ\tau=g_m\,\Gamma_m+\sum_{m'>m}h_{mm'}\,\Gamma_{m'},\qquad m=1,\dots,k .
\]
So, for every $m$, the common zero sets of $(\Gamma_m,\dots,\Gamma_k)$ and of $(\Gamma'_m,\dots,\Gamma'_k)\circ\tau$ in the complement of the arrangement coincide. At a common zero of $\Gamma_1,\dots,\Gamma_k$, one of the two families is nondegenerate if and only if the other one is.
\end{lemma}

The reason is that the functions $\psi_j$ are homogeneous of degree $0$ and only change by additive constants, so that, by the chain rule, each Jacobian is multiplied by a nowhere vanishing function, modulo the later Jacobians. The proof is given in Appendix~\ref{app:proofs}.

\begin{definition}
Assume (U$_D$) and (U$_B$). Let $|S|=j\le k-1$ and $\mathcal C\in\mathrm{Ch}(S)$. We denote by $Z(S,\mathcal C)$ the number of common zeros in $\mathcal C$ of the Jacobians $\tilde\Gamma_{j+1},\dots,\tilde\Gamma_k$ of a reduced system along $S$. By Lemma~\ref{lem:chart-independence} this number, and the nondegeneracy of these zeros, do not depend on the choice of the reduced system. For $|S|=k$ we put $Z(S,\{v_S\})=1$, the number of zeros of the empty family.
\end{definition}

By Theorems~\ref{thm:restriction-LS} and~\ref{thm:restriction-LS-infinity}, together with Proposition~\ref{prop:lower-bound-LS}(1), $Z(S,\mathcal C)$ is also the number of common zeros in $\mathcal C$ of $\Gamma_{j+1},\dots,\Gamma_k$ if $n+k+1\notin S$, respectively of $\Gamma^{\mathrm{top}}_{j+1},\dots,\Gamma^{\mathrm{top}}_k$ if $n+k+1\in S$. For $S=\emptyset$, $Z(\emptyset,\mathcal C)$ is the number of solutions of $\Gamma_1=\dots=\Gamma_k=0$ in the chamber $\mathcal C$ of $\mathbb{R}^k$.

We now add the genericity assumption:
\begin{itemize}
\item[(G)] for every $S$ with $|S|\le k-1$ and every $\mathcal C\in\mathrm{Ch}(S)$, the zeros counted by $Z(S,\mathcal C)$ are nondegenerate.
\end{itemize}

\begin{lemma}\label{lem:inheritance}
Assume (U$_D$) and (U$_B$). Let $|S|=j\le k-1$ and let $(\tilde D,\tilde B)$ be a reduced system along $S$, with $M_S=B_{S,\{1,\dots,j\}}$.
\begin{enumerate}
\item For $U\subset\tilde R$ with $|U|=d+1$, $\det\tilde D_U=\pm\det D_{S\cup U}/\det D_{S\cup T}$. For $U\subset\tilde R$ with $1\le|U|\le d$,
\[
\det\tilde B_{U,\{j+1,\dots,j+|U|\}}=\pm\frac{\det B_{S\cup U,\{1,\dots,j+|U|\}}}{\det M_S}.
\]
So $(\tilde D,\tilde B)$ satisfies (U$_D$) and (U$_B$).
\item For $i\in\tilde R$, a reduced system of $(\tilde D,\tilde B)$ along $\{i\}$ is obtained from a reduced system of $(D,B)$ along $S\cup\{i\}$ by operations as in Lemma~\ref{lem:chart-independence}.
\end{enumerate}
\end{lemma}

Assertion~1 follows from Lemma~\ref{lem:minor-coordinates} and Schur's determinant formula, and assertion~2 from the uniqueness of the reduced exponents, which are the relations of the support vanishing on the rows indexed by $S$. The proof is given in Appendix~\ref{app:proofs}.

\begin{lemma}[Vertices]\label{lem:vertices}
Let $(D,B)$ satisfy (U$_D$) and (U$_B$). Then $\Gamma_k$ does not vanish at the finite vertices $\Lambda_{S'}$, $|S'|=k$, $n+k+1\notin S'$, and $\Gamma^{\mathrm{top}}_k$ does not vanish at the vertices at infinity.
\end{lemma}

At a vertex, only one term of the Cauchy--Binet formula of Proposition~\ref{prop:gammak} survives, and it is a product of nonzero minors. The proof is given in Appendix~\ref{app:proofs}.

The next lemma is what makes the recursion work. It says that, under the genericity assumptions, the facet-count formula of Theorem~\ref{thm:facet-count-LS} can be applied to every chamber $\mathcal C$ of every flat $\Lambda_S$, and that the numbers of zeros on the facets which enter this formula are exactly the numbers $Z(S\cup\{i\},F_i(\mathcal C))$ attached to the facets $F_i(\mathcal C)$, which are chambers of flats of one dimension less. In this way the boundary degree of $\mathcal C$, hence a lower bound for $Z(S,\mathcal C)$, is expressed in terms of the same kind of numbers in lower dimension.

\begin{lemma}[Facets]\label{lem:facets-recursive}
Assume (U$_D$), (U$_B$) and (G). Let $|S|=j\le k-1$, $\mathcal C\in\mathrm{Ch}(S)$, and let $(\tilde D,\tilde B)$ be a reduced system along $S$ with chart $t_0$. Let $\tilde{\mathcal C}$ be the chamber $\mathcal C$ written in this chart. Then:
\begin{enumerate}
\item $\tilde{\mathcal C}$ satisfies the hypotheses (H) and (H$_\infty$) of Theorem~\ref{thm:facet-count-unbounded} for the system $(\tilde D,\tilde B)$;
\item the numbers of Theorem~\ref{thm:facet-count-LS} are $\tilde N_i=Z(S\cup\{i\},F_i(\mathcal C))$ for $i\in I(\mathcal C)$, and $\tilde N_i=0$ for $i\notin I(\mathcal C)$.
\end{enumerate}
\end{lemma}

The facets of a chamber of $\Lambda_S$ are chambers of the flats $\Lambda_{S\cup\{i\}}$, and by the restriction theorems of Section~\ref{sec:restriction}, applied to the reduced system, the reduced Jacobians on these facets are those of the reduced systems along $S\cup\{i\}$. The proof is given in Appendix~\ref{app:proofs}.

\subsection{The recursive bounds}

\begin{lemma}[Bézout bound]\label{lem:bezout-flat}
Assume (U$_D$), (U$_B$) and (G). Let $|S|=j\le k-1$ and $d=k-j$. Then
\[
\sum_{\mathcal C\in\mathrm{Ch}(S)}Z(S,\mathcal C)\ \le\ \beta_d:=n^d\,2^{d(d-1)/2}.
\]
\end{lemma}

\begin{proof}
Fix a reduced system along $S$ with chart $t_0$. Every chamber of $\Lambda_S$ lies in $\{P_{t_0}\neq0\}$, since $H_{t_0}$ belongs to the arrangement on $\Lambda_S$. The zeros counted by the $Z(S,\mathcal C)$ are therefore distinct real zeros of $\tilde\Gamma_{j+1},\dots,\tilde\Gamma_k$ in $\mathbb{R}^d$. They are nondegenerate, hence isolated in $\mathbb{C}^d$. By Proposition~\ref{prop:gammak}(2) applied to the reduced system, $\deg\tilde\Gamma_m\le2^{k-m}n$. By the Bézout inequality for isolated solutions \cite[Ex.~8.4.6]{Fulton}, their number is at most $\prod_{m=j+1}^k2^{k-m}n=n^d2^{d(d-1)/2}$.
\end{proof}

\begin{definition}\label{def:recursive-bounds}
For $|S|=j\le k-1$ and $i\notin S$ put
\[
\sigma_S(i)=\sign\det B_{S\cup\{i\},\{1,\dots,j+1\}},
\]
where the rows are those of $S$ in increasing order followed by the row $i$. It is nonzero by (U$_B$). Define integers $\ell_S(\mathcal C)$ and $u_S(\mathcal C)$, for all $S$ and all $\mathcal C\in\mathrm{Ch}(S)$, by induction on $d=k-|S|$.
\begin{itemize}
\item $d=0$: $\ell_S(\{v_S\})=u_S(\{v_S\})=1$.
\item $d\ge1$: for $\mathcal C\in\mathrm{Ch}(S)$, put
\[
a_S(\mathcal C)=\sum_{\substack{i\in I(\mathcal C)\\ \sigma_S(i)>0}}\ell_{S\cup\{i\}}\big(F_i(\mathcal C)\big)-\sum_{\substack{i\in I(\mathcal C)\\ \sigma_S(i)<0}}u_{S\cup\{i\}}\big(F_i(\mathcal C)\big),
\]
\[
b_S(\mathcal C)=\sum_{\substack{i\in I(\mathcal C)\\ \sigma_S(i)>0}}u_{S\cup\{i\}}\big(F_i(\mathcal C)\big)-\sum_{\substack{i\in I(\mathcal C)\\ \sigma_S(i)<0}}\ell_{S\cup\{i\}}\big(F_i(\mathcal C)\big),
\]
and then
\[
\ell_S(\mathcal C)=\max\Big(0,\ \Big\lceil\frac{a_S(\mathcal C)}2\Big\rceil,\ \Big\lceil-\frac{b_S(\mathcal C)}2\Big\rceil\Big),\qquad
u_S(\mathcal C)=\beta_d-\sum_{\mathcal C'\in\mathrm{Ch}(S)\setminus\{\mathcal C\}}\ell_S(\mathcal C') .
\]
\end{itemize}
At level $d$, the lower bounds $\ell_S$ use only the bounds at level $d-1$, and the upper bounds $u_S$ use the lower bounds at level $d$. So the sequences $(\ell_S(\mathcal C))$ and $(u_S(\mathcal C))$ are well defined.
\end{definition}

\begin{theorem}\label{thm:recursive-bounds}
Assume (U$_D$), (U$_B$) and (G). Then for every $S\subset[n+k+1]$ with $|S|\le k$ and every $\mathcal C\in\mathrm{Ch}(S)$,
\[
\ell_S(\mathcal C)\ \le\ Z(S,\mathcal C)\ \le\ u_S(\mathcal C).
\]
In particular, every chamber $\mathcal C$ of $\mathbb{R}^k\setminus\bigcup_{i\le n+k}\{p_i=0\}$ contains at least $\ell_\emptyset(\mathcal C)$ and at most $u_\emptyset(\mathcal C)$ solutions of $\Gamma_1=\dots=\Gamma_k=0$.
\end{theorem}

\begin{proof}
We argue by induction on $d=k-|S|$. For $d=0$ both bounds equal $Z=1$.

Let $d\ge1$, $\mathcal C\in\mathrm{Ch}(S)$, and choose a reduced system along $S$. By Lemma~\ref{lem:facets-recursive}, Theorem~\ref{thm:facet-count-LS} applies to $\tilde{\mathcal C}=\mathcal C$, with $\tilde N_i=Z(S\cup\{i\},F_i(\mathcal C))$ for $i\in I(\mathcal C)$ and $\tilde N_i=0$ otherwise. By the induction hypothesis, $\ell_{S\cup\{i\}}(F_i(\mathcal C))\le\tilde N_i\le u_{S\cup\{i\}}(F_i(\mathcal C))$.

Proposition~\ref{prop:degree-lower-bound} is stated with the signs $\sign\tilde B_{i,j+1}=\sign(\det M_S)\,\sigma_S(i)$ (Proposition~\ref{prop:reduced-format}). If $\det M_S<0$, the two sets $I_\pm$ are exchanged, which replaces $(a,b)$ by $(-b,-a)$ and leaves the bound unchanged. We get $|\tilde d(\mathcal C)|\ge\ell_S(\mathcal C)$. By Proposition~\ref{prop:lower-bound-LS}(2), whose nondegeneracy hypothesis is (G), $Z(S,\mathcal C)\ge|\tilde d(\mathcal C)|\ge\ell_S(\mathcal C)$.

For the upper bound, Lemma~\ref{lem:bezout-flat} and the lower bounds for the other chambers of $\Lambda_S$ give
\[
Z(S,\mathcal C)\le\beta_d-\sum_{\mathcal C'\neq \mathcal C}Z(S,\mathcal C')\le\beta_d-\sum_{\mathcal C'\neq \mathcal C}\ell_S(\mathcal C')=u_S(\mathcal C). \qedhere
\]
\end{proof}

\begin{remark}
\begin{enumerate}
\item \emph{Dimension one.} If $|S|=k-1$, the chambers of the projective line $\Lambda_S$ are the $n+2$ arcs $\alpha$ of Proposition~\ref{prop:line-zeros}. Their facets are the two endpoints, and $\sigma_S(i)=\sign\det B_{S\cup\{i\},\{1,\dots,k\}}$. Moreover $\beta_1=n$. Definition~\ref{def:recursive-bounds} then gives $\ell_S(\alpha)=l_\alpha$ and $u_S(\alpha)=L_\alpha$. So the first step of the recursion is exactly the Descartes' rule of signs for circuits \cite{BD}, with its parity refinement.
\item \emph{Combinatorial nature.} The bounds only depend on two pieces of data: the face structure of the projective arrangement $\{H_i\}$, which is determined by the oriented matroid of $D$, and the signs of the minors $\det B_{S\cup\{i\},\{1,\dots,|S|+1\}}$. No function $\Gamma_m$ has to be computed.
\end{enumerate}
\end{remark}

\section{Bounds for the positive solutions}\label{sec:possol}

In this section we deduce from the recursive bounds of Section~\ref{sec:recursive} our bounds for the number of positive solutions. We first determine the faces of $\Delta$ which are reached by the ends of the curves $C_j$, then we bound the number of positive solutions by numbers of zeros of Jacobians on these faces (Theorem~\ref{thm:possol-Z}), and finally by the recursive bounds (Corollary~\ref{cor:possol-recursive}).

We now return to the polynomial system~\eqref{eq:system} of Section~\ref{sec:polysys}, with $B$ and $D$ Gale dual matrices of $A$ and $C$, and to its positive chamber $\Delta=\{p_1>0,\dots,p_{n+k}>0\}$. We assume $\Delta\neq\emptyset$; otherwise~\eqref{eq:system} has no positive solution by Theorem~\ref{thm:gale-duality}. We view $\Delta$ as a chamber of $\mathrm{Ch}(\emptyset)$. For $S\subset[n+k+1]$ we write $\Delta_S$ for the chamber of $\mathrm{Ch}(S)$ contained in $\overline\Delta$, when it exists; it is the face of $\overline\Delta$ lying on $\Lambda_S$.

\begin{definition}\label{def:boundary-faces}
For $1\le j\le k$, let $\mathcal S_j(\Delta)$ be the set of subsets $S\subset[n+k+1]$ with $|S|=j$ such that
\begin{enumerate}
\item $\Delta_S$ exists, that is, $\overline\Delta$ has a face of codimension $j$ on $\Lambda_S$;
\item the vectors $b_i^{(j-1)}=(B_{i1},\dots,B_{i,j-1})$, $i\in S$, positively span $\mathbb{R}^{j-1}$.
\end{enumerate}
Condition~2 is motivated by the curves $C_j$. Since $C_j$ is contained in the level set $\mu_j$, its ends lie in $\overline{\mu_j}\setminus\mu_j$. Under (U$_D$) and (U$_B$), Theorem~\ref{thm:mu-boundary} shows that $\overline{\mu_j}$ meets the relative interior of a face $\Delta_S$ with $|S|=j$ if and only if condition~2 holds, and Proposition~\ref{prop:ends} below shows that the ends of the curves $C_j(\Theta)$, obtained from $C_j$ by perturbing the Jacobians (Section~\ref{sec:KR}), lie on these faces only. For $j=1$ condition~2 is empty, so $\mathcal S_1(\Delta)$ is the set of facets of $\Delta$, including the facet at infinity if $\Delta$ is unbounded. For $j=k$, $\Delta_S$ is a vertex and $Z(S,\Delta_S)=1$.
\end{definition}

In Definition~\ref{def:boundary-faces}, condition~1 only depends on the oriented matroid of $D$, equivalently of $C$, and condition~2 only on signs of minors of $B$, equivalently of the lifted exponent matrices (Remark~\ref{rem:flag-minors}). We first make condition~1 explicit. Recall that $\overline\Delta\subset\mathbb{S}^k$ is the image of the cone $K=\{v\in\mathbb{R}^{k+1}:\ \langle D_i,v\rangle\ge0,\ i\in[n+k+1]\}$, since $P_i(v)=\langle D_i,v\rangle$ and $P_{n+k+1}=y_{k+1}$.

\begin{proposition}[Faces of $\Delta$]\label{prop:faces-Delta}
Let $S\subset[n+k+1]$ with $|S|=j\le k$. The following conditions are equivalent.
\begin{enumerate}
\item $\Delta_S$ exists, that is, $\overline\Delta$ has a face of codimension $j$ on $\Lambda_S$.
\item There is $v\in\mathbb{R}^{k+1}$ with $\langle D_s,v\rangle=0$ for $s\in S$ and $\langle D_i,v\rangle>0$ for $i\notin S$. In other words, the sign vector which is $0$ on $S$ and $+$ outside $S$ is a covector of the oriented matroid of the rows of $D$.
\item There is no $u\in\mathbb{R}^n$ such that the vector $u^TC\in\mathbb{R}^{n+k+1}$ has nonnegative entries outside $S$, not all zero. Equivalently, no linear combination of the equations of~\eqref{eq:system} has all its coefficients outside $S$ nonnegative and not all zero.
\end{enumerate}
Under (U$_D$), they are also equivalent to:
\begin{enumerate}
\setcounter{enumi}{3}
\item $S$ is contained in a set $V$ with $|V|=k$ such that the determinants $\det(D_V;D_i)$, $i\notin V$, all have the same sign. Here $(D_V;D_i)$ is the square matrix with the rows $D_v$, $v\in V$, in increasing order, followed by the row $D_i$. Equivalently, $\overline\Delta$ has a vertex $\Delta_V$ lying on $\Lambda_S$.
\end{enumerate}
\end{proposition}

\begin{proof}
The faces of $\overline\Delta$ are the images of the faces $K\cap\{\langle D_s,v\rangle=0,\ s\in S\}$ of $K$, and $\Delta_S$ is the image of the relative interior of such a face on which all the other forms are positive. This gives $1\Leftrightarrow2$. By Motzkin's transposition theorem, the system of~2 has no solution if and only if there is $\eta\in\mathbb{R}^{n+k+1}$ with $\eta^TD=0$, $\eta_i\ge0$ for $i\notin S$ and $\eta_i\neq0$ for some $i\notin S$. Since the columns of $D$ form a basis of $\ker C$, the vectors $\eta$ with $\eta^TD=0$ are the vectors $\eta=C^Tu$, $u\in\mathbb{R}^n$. This gives $2\Leftrightarrow3$. Finally, $D$ has rank $k+1$, so $K$ is a pointed cone and every nonzero face of $K$ contains an extreme ray. Under (U$_D$), the extreme rays are spanned by $\pm v_V$, $|V|=k$, where $v_V$ is the vector of cofactors of $D_V$, so that $\langle D_i,v_V\rangle=\det(D_V;D_i)$. Such a ray lies in $K$ if and only if these numbers have the same sign for $i\notin V$, and it lies in the face of~2 if and only if $S\subset V$. This gives $2\Leftrightarrow4$.
\end{proof}

By Proposition~\ref{prop:gale-minors}, $\det(D_V;D_i)=\kappa_D\,(-1)^{|V\cup\{i\}|_\Sigma+\#\{v\in V:\,v>i\}}\det C^{[n+k+1]\setminus(V\cup\{i\})}$, so condition~4 is a sign condition on maximal minors of the coefficient matrix $C$.

For $S=\emptyset$, condition~3 is Stiemke's criterion: $\Delta\neq\emptyset$ if and only if no linear combination of the equations has nonnegative coefficients, not all zero, which is the obvious obstruction to the existence of positive solutions.

Under (U$_B$), condition~2 of Definition~\ref{def:boundary-faces} is a sign condition on minors of $B$. Write $S=\{i_1<\dots<i_j\}$. Any $j-1$ of the vectors $b_i^{(j-1)}$, $i\in S$, are linearly independent, so they satisfy a unique linear relation up to scaling. By Cramer's rule its coefficients are $(-1)^r\det B_{S\setminus\{i_r\},\{1,\dots,j-1\}}$, $r=1,\dots,j$. Condition~2 holds if and only if these $j$ numbers all have the same sign. For $j=2$ this means $B_{i_11}B_{i_21}<0$.

We use the notation $C_j(\Theta)$, $Z_j(\Theta)$ and the degrees $D_m=2^{k-m}n$ of Section~\ref{sec:KR}, and the homogenizations $\Gamma^h_m$ of degree $D_m$ of Section~\ref{subsec:restriction-infinity}. So $\Gamma^h_m=y_{k+1}^{D_m}\Gamma_m(y/y_{k+1})$ for $y_{k+1}>0$, and $\Gamma^h_m=\Gamma^{\mathrm{top}}_m$ on the hyperplane at infinity. Closures are taken in $\mathbb{S}^k$, which we identify with the unit sphere of $\mathbb{R}^{k+1}$, so that forms are functions on $\mathbb{S}^k$.

\begin{lemma}[Zeros on closed faces]\label{lem:closed-faces}
Assume (U$_D$), (U$_B$) and (G). Let $0\le j\le k-1$.
\begin{enumerate}
\item For every $S'$ with $j+1\le|S'|\le k$ such that $\Delta_{S'}$ exists, the forms $\Gamma^h_{j+1},\dots,\Gamma^h_k$ have no common zero on $\overline{\Delta_{S'}}$.
\item Let $|S|=j$ be such that $\Delta_S$ exists, with $\Delta_\emptyset=\Delta$. The common zeros in $\overline{\Delta_S}$ of the restrictions of $\Gamma^h_{j+1},\dots,\Gamma^h_k$ to $L_S$ lie in $\Delta_S$. There are $Z(S,\Delta_S)$ of them, they are nondegenerate zeros of these restrictions, and, if $j\ge1$, $\Gamma^h_j$ does not vanish at them.
\end{enumerate}
\end{lemma}

\begin{proof}
Let $\xi\in\overline\Delta$, lying in the relative interior of the face $\Delta_{S''}$, with $|S''|=j''$. On the hyperplane at infinity, $\Gamma^h_m$ is the homogeneous part of degree $D_m$ of $\Gamma_m$, which is $\Gamma^{\mathrm{top}}_m$ by Lemma~\ref{lem:top}. We apply Steps~1 and~2 of the proof of Lemma~\ref{lem:facets-recursive} (Appendix~\ref{app:proofs}) with $S=\emptyset$, $\mathcal C=\Delta$ and $t_0=n+k+1$. The reduced system along $\emptyset$ is then $(D,B)$ itself, and its top-degree parts on the hyperplane at infinity are the restrictions of the $\Gamma^h_m$. If $j''=k$, then $\Gamma^h_k(\xi)\neq0$ by Lemma~\ref{lem:vertices}. If $j''<k$ and $\Gamma^h_{j''+1}(\xi)=\dots=\Gamma^h_k(\xi)=0$, then $\Gamma^h_{j''}(\xi)\neq0$ by Step~2. In both cases $\Gamma^h_{j''},\dots,\Gamma^h_k$ have no common zero on $\Delta_{S''}$. If $j''\ge j+1$, this family is contained in $\Gamma^h_{j+1},\dots,\Gamma^h_k$. Since $\overline{\Delta_{S'}}$ is the union of the faces $\Delta_{S''}$ with $S''\supset S'$, this proves~1. By~1, the common zeros of~2 lie in $\Delta_S$. By Step~1 of the same proof, they are the zeros counted by $Z(S,\Delta_S)$. By Theorems~\ref{thm:restriction-LS}(1) and~\ref{thm:restriction-LS-infinity}(1)--(2), at such a zero the Jacobian matrices of the restrictions and of the reduced Jacobians differ by an invertible triangular matrix and by a change of coordinates, so these zeros are nondegenerate by (G). Finally $\Gamma^h_j$ does not vanish at them by Step~2, applied with $S''=S$.
\end{proof}

We use the following consequence of the implicit function theorem and of compactness. Let $K$ be a compact subset of a smooth manifold $M$ of dimension $d$, and let $g\colon M\to\mathbb{R}^{d'}$ be smooth.
\begin{itemize}
\item[($\ast$)] If $g$ has no zero in $K$, then every $g'$ close enough to $g$ in the $C^0$ topology on $K$ has no zero in $K$. If $d'=d$ and the zeros of $g$ in $K$ are finitely many, nondegenerate and in the interior of $K$, then every $g'$ close enough to $g$ in the $C^1$ topology on a neighborhood of $K$ has the same number of zeros in $K$, all nondegenerate, in the interior of $K$, and each close to a zero of $g$.
\end{itemize}
If $\Theta_m$ is a polynomial of degree at most $D_m$ close to $\Gamma_m$, then the form $\Theta^h_m$ is close to $\Gamma^h_m$ in the $C^1$ topology on $\mathbb{S}^k$.

\begin{proposition}\label{prop:ends}
Assume (R), (U$_D$), (U$_B$) and (G). There is $\eta>0$ with the following property. Let $\Theta_1,\dots,\Theta_k$ be polynomials with $\deg\Theta_m\le D_m$, at distance at most $\eta$ from $\Gamma_1,\dots,\Gamma_k$. Then:
\begin{enumerate}
\item $\#Z_0(\Theta)=Z(\emptyset,\Delta)$;
\item if $\Theta$ satisfies condition~(i) of Section~\ref{sec:KR}, then for $1\le j\le k$,
\[
2\,b(C_j(\Theta))\ \le\ \sum_{S\in\mathcal S_j(\Delta)}Z(S,\Delta_S);
\]
\item for $1\le m\le k$, if the points of $Z_m(\Theta)$ are nondegenerate solutions of $\psi_1=\dots=\psi_m=\Theta_{m+1}=\dots=\Theta_k=0$, then $Z_m(\Theta)$ is finite.
\end{enumerate}
Assertion~3 only involves $\Theta_{m+1},\dots,\Theta_k$.
\end{proposition}

The idea is that, for $\Theta$ close enough to $\Gamma$, the zeros of the perturbed Jacobians on the closed faces of $\Delta$ are close to those of the Jacobians, and that the ends of the curves $C_j(\Theta)$ can only accumulate at common zeros of $\Theta_{j+1},\dots,\Theta_k$ on the faces $\Delta_S$, $S\in\mathcal S_j(\Delta)$, by Theorem~\ref{thm:mu-boundary}, and that exactly one end converges to each of these zeros. The proof is given in Appendix~\ref{app:proofs}.

\begin{theorem}\label{thm:possol-Z}
Assume (R), (U$_D$), (U$_B$) and (G). Then the number of nondegenerate positive solutions of~\eqref{eq:system} is at most
\[
Z(\emptyset,\Delta)+\frac12\sum_{j=1}^k\ \sum_{S\in\mathcal S_j(\Delta)}Z(S,\Delta_S).
\]
\end{theorem}

\begin{proof}
Rescale the coefficients as explained after Lemma~\ref{lem:perturbed-jacobians}. This does not decrease the number of nondegenerate positive solutions. It preserves (R), (U$_D$), (U$_B$), (G), the chamber $\Delta$, its faces and the sets $\mathcal S_j(\Delta)$, and it multiplies the Jacobians and the reduced Jacobians by positive constants, so that it does not change the numbers $Z(S,\mathcal C)$. So we may assume that $0$ is a regular value of $(\psi_1,\dots,\psi_r)$ on $\Delta$ for $1\le r\le k$; in particular all the positive solutions are nondegenerate. Let $\eta$ be as in Proposition~\ref{prop:ends}, and let $\Theta$ be given by Lemma~\ref{lem:perturbed-jacobians} for this $\eta$. By Theorem~\ref{thm:KR-general} and Proposition~\ref{prop:ends},
\[
\#\{x\in\mathbb{R}^n_{>0}:\ F(x)=0\}\ \le\ \#Z_0(\Theta)+\sum_{j=1}^kb(C_j(\Theta))\ \le\ Z(\emptyset,\Delta)+\frac12\sum_{j=1}^k\ \sum_{S\in\mathcal S_j(\Delta)}Z(S,\Delta_S).\qedhere
\]
\end{proof}

\begin{corollary}\label{cor:possol-recursive}
Under the assumptions of Theorem~\ref{thm:possol-Z}, the number of nondegenerate positive solutions of~\eqref{eq:system} is at most
\[
u_\emptyset(\Delta)+\frac12\sum_{j=1}^{k}\ \sum_{S\in\mathcal S_j(\Delta)}u_S(\Delta_S).
\]
The term $j=k$ equals $\frac12\#\mathcal S_k(\Delta)$, half the number of vertices $v_S$ of $\overline\Delta$ such that the vectors $b_i^{(k-1)}$, $i\in S$, positively span $\mathbb{R}^{k-1}$.
\end{corollary}

\begin{proof}
Combine Theorem~\ref{thm:possol-Z} with Theorem~\ref{thm:recursive-bounds}, and use $Z(S,\{v_S\})=u_S(\{v_S\})=1$ for $|S|=k$.
\end{proof}

\begin{remark}\label{rem:possol-small-k}
\begin{enumerate}
\item For $k=1$, $\mathcal S_1(\Delta)$ consists of the two endpoints of the interval $\Delta$, and Corollary~\ref{cor:possol-recursive} reads $\#\{\text{positive solutions of }\eqref{eq:system}\}\le L_\Delta+1$. This is the Descartes' rule of signs for circuits of \cite{BD}.
\item For $k=2$, the curves are $C_1(\Theta)=\{\Theta_2=0\}\cap\Delta$, where $\Theta_2$ is a small perturbation of $\Gamma_2$, and $C_2=\{\psi_1=0\}\cap\Delta$. The set $\mathcal S_1(\Delta)$ consists of the edges of $\Delta$, and $\mathcal S_2(\Delta)$ of the vertices $\{i,i'\}$ of $\overline\Delta$, where the edges on $H_i$ and $H_{i'}$ meet, such that $B_{i1}B_{i'1}<0$.
\end{enumerate}
\end{remark}

\begin{remark}[Minors of $B$ with fewer columns]\label{rem:flag-minors}
The minors $\det B_{U,\{1,\dots,j\}}$ with $|U|=j<k$, which enter the signs $\sigma_S(i)$ and assumption (U$_B$), are not determined by $\mathcal A$ alone: they depend on the flag $V_1\subset\dots\subset V_k=\ker A$, where $V_j$ is spanned by the first $j$ columns $b_1,\dots,b_j$ of $B$. They are nevertheless maximal minors of a lifted exponent matrix. Choose vectors $w_1,\dots,w_k\in\mathbb{R}^{n+k+1}$ with $\langle w_m,b_l\rangle=\delta_{ml}$, and let $A_j$ be the $(n+k+1-j)\times(n+k+1)$ matrix obtained by adding to $A$ the rows $w_{j+1},\dots,w_k$. Its columns are the points $(1,a_i,w_{j+1,i},\dots,w_{k,i})$, a lifting of $\mathcal A$ to $\mathbb{R}^{n+k-j}$, and $\ker A_j=V_j$. Proposition~\ref{prop:gale-minors} applied to $A_j$ and $(b_1,\dots,b_j)$ gives
\[
\det B_{U,\{1,\dots,j\}}=\kappa_j\,(-1)^{|U|_\Sigma}\det A_j^{[n+k+1]\setminus U},\qquad |U|=j,
\]
for a nonzero constant $\kappa_j$. For $j=k$ this is Proposition~\ref{prop:gale-minors}. For $j=1$ it says that the entries of the relation $b_1$ are, up to the signs $(-1)^u$, the maximal minors of the lifted configuration obtained by adding $k-1$ coordinates. So the choice of the relations $b_1,\dots,b_{k-1}$ amounts to the choice of a lifting of $\mathcal A$, and the signs of all the minors of $B$ used in this paper are signs of maximal minors of the liftings $A_j$.
\end{remark}

\begin{remark}[Dependence on signs and on the flag]\label{rem:flag-choice}
\emph{Signs only.} In Theorem~\ref{thm:possol-Z}, the index sets $\mathcal S_j(\Delta)$ depend only on signs: by Proposition~\ref{prop:faces-Delta}, the existence of $\Delta_S$ depends only on the oriented matroid of $C$ (equivalently of $D$), and the positive spanning condition depends only on the signs of the minors $\det B_{U,\{1,\dots,|U|\}}$, that is, by Remark~\ref{rem:flag-minors}, on the oriented matroids of the lifted configurations $A_j$. The numbers $Z(\emptyset,\Delta)$ and $Z(S,\Delta_S)$ are numbers of zeros of Jacobians, but they are bounded above by the recursive bounds of Section~\ref{sec:recursive}, which are computed from the same sign data. Hence \emph{the bound of Corollary~\ref{cor:possol-recursive} depends only on the oriented matroid of $C$ and on the oriented matroids of $A$ and of the lifted configurations $A_j$, $1\le j\le k-1$.} The maximal minors of $A$ enter through the data on the lines. In particular it depends on signs only, but on the signs of more than the maximal minors of $A$. The sharper bound of Corollary~\ref{cor:possol-BDF} also uses, on the lines of the arrangement, the values of the maximal minors of $B$ through their partial sums (Proposition~\ref{prop:BDF-line}).

\emph{The flag.} The number of positive solutions of~\eqref{eq:system} does not depend on the ordered basis $b_1,\dots,b_k$ of $V_B=\ker A$, but the constructions of Sections~\ref{sec:KR}--\ref{sec:BDF} do. The Jacobian $\Gamma_k$ only depends on $V_B$, but for $m<k$ the Jacobian $\Gamma_m$ depends on the flag $V_m\subset\dots\subset V_{k-1}$, where $V_j$ is spanned by $b_1,\dots,b_j$. So the curves $C_j$ of Section~\ref{sec:KR}, the numbers $Z(S,\mathcal C)$ and all the sign-based bounds depend on the flag. The signs used in the bounds are those of the minors $\det B_{U,\{1,\dots,j\}}$, namely of the Plücker coordinates of the subspaces $V_j$. The data on the lines, which only involve the maximal minors of $B$, do not depend on the flag.

There are only finitely many sign patterns, so the minimum of any of our bounds over all flags is attained, and it can be found by a finite search. Note that, for $|U|=j$, the minor $\det B_{U,\{1,\dots,j\}}$ vanishes if and only if $V_j$ contains a nonzero vector vanishing on $U$. For $j=k-1$, the vectors of $V_B$ vanishing on a set $U$ with $|U|=k-1$ form a line, spanned by the relation $c_U$ of the circuit $\mathcal A\setminus\{a_u\}_{u\in U}$, which has $n+2$ elements. So the sign data of the hyperplane $V_{k-1}\subset V_B$ only changes when $V_{k-1}$ passes through one of the lines $\mathbb{R}c_U$. The possible sign data of $V_{k-1}$ correspond to the chambers of the arrangement of the $\binom N{k-1}$ hyperplanes $c_U^\perp$ in the dual projective space $\mathbb{P}(V_B^*)\cong\mathbb{P}^{k-1}$, which only depends on $\mathcal A$. Recursively, once $V_{j+1}$ is fixed, the sign data of $V_j\subset V_{j+1}$ correspond to the chambers of the analogous arrangement in $\mathbb{P}(V_{j+1}^*)$, whose hyperplanes are defined by the lines $V_{j+1}\cap\{v_U=0\}$, $|U|=j$. In the language of Remark~\ref{rem:flag-minors}, these lines are the circuits of the lifted configuration $A_{j+1}$, and choosing a flag amounts to choosing a chain of liftings of $\mathcal A$ to $\mathbb{R}^{n+1},\dots,\mathbb{R}^{n+k-1}$, up to their oriented matroids.

There is no choice which is best for all the terms of the bounds at the same time: the lower bounds of the other chambers favour relations whose signs are constant on the boundaries of many chambers, whereas the boundary terms $b(C_j)$ favour relations with few sign changes around the positive chamber. For $k\ge3$, a practical procedure is to choose $V_{k-1}$ first, among the chambers of the arrangement of the $c_U^\perp$, then $V_{k-2}$, and so on; this greedy procedure need not give the global minimum.
\end{remark}

\begin{remark}[The assumptions]\label{rem:genericity}
The assumptions (R) and (U$_D$) hold as soon as all the maximal minors of $C$ are nonzero (Proposition~\ref{prop:gale-minors}), hence for generic coefficients. The assumption (U$_B$) holds for a generic choice of the flag $b_1,\dots,b_{k-1}$ when no $n+1$ points of $\mathcal A$ lie on an affine hyperplane. The transversality assumption of the Khovanskii--Rolle argument of \cite{BS} for the curves $C_j$ is not needed: following \cite{BS2}, we apply this argument to small perturbations of the Jacobians (Lemma~\ref{lem:perturbed-jacobians}). The assumption (G) is a nondegeneracy assumption on the zeros of the Jacobians on the flats of the arrangement. It cannot be removed in the same way, since the recursion uses the exact Jacobians through the facet-count formula and the restriction theorems. Since the bound of Corollary~\ref{cor:possol-recursive} only depends on locally constant sign data, and nondegenerate positive solutions persist under small perturbations of the coefficients, these bounds hold for the nondegenerate positive solutions of any system satisfying (R), (U$_D$) and (U$_B$) which can be approximated by systems satisfying (G) with the same sign data.

\emph{Degenerate supports and coefficients.} The assumptions on $\mathcal A$ and $C$ can be removed at the cost of a minimization. Theorem~\ref{thm:gale-duality} holds for real exponents, and a nondegenerate positive solution of~\eqref{eq:system} persists, as a nondegenerate positive solution, under small perturbations of the points of $\mathcal A$ in $\mathbb{R}^n$ and of the coefficients, by the implicit function theorem. It follows that the number of nondegenerate positive solutions of~\eqref{eq:system} is at most the bound of Corollary~\ref{cor:possol-recursive} (or~\ref{cor:possol-BDF}) of any sufficiently close system satisfying the assumptions, and therefore at most the minimum of these bounds over the sign data of such systems. If, for instance, $n+1$ points of $\mathcal A$ lie on an affine hyperplane, the corresponding maximal minor of $A$ vanishes, and a small perturbation of $\mathcal A$ makes it positive or negative; the admissible sign data are those of the configurations in general position close to $\mathcal A$, and similarly for $C$.
\end{remark}

\section{Sharper recursive bounds from the optimal Descartes' rule for circuits}\label{sec:BDF}

In this section we replace the base case of the recursion of Section~\ref{sec:recursive} by the optimal Descartes' rule of signs of Bihan, Dickenstein and Forsgård \cite{BDF}. The upper bounds for the numbers of zeros on arcs of lines then use the partial sums of the exponents rather than the exponents themselves. The recursion is otherwise unchanged, and we obtain bounds that are never worse than those of Theorem~\ref{thm:recursive-bounds}.

Throughout this section we assume (U$_D$), (U$_B$) and (G).

\subsection{The optimal Descartes' rule on a line}

Let $S\subset[n+k+1]$ with $|S|=k-1$, so that $\Lambda_S$ is a projective line, and use the notation of Subsection~\ref{subsec:line-zeros} with $S'=S$: the $n+2$ points $r_l$, $l\in R'=[n+k+1]\setminus S$, the arcs $\alpha$, the numbers $N_\alpha=Z(S,\alpha)$, $l_\alpha$ and $L_\alpha$ of Proposition~\ref{prop:line-zeros}, and the exponents $\gamma_l$ of the reduced system on $\Lambda_S$. Multiplying all the $\gamma_l$ by a common nonzero constant does not change the numbers defined below, so we may and do take
\[
\gamma_l=\det B_{S\cup\{l\},\{1,\dots,k\}},\qquad l\in R',
\]
with the rows of $S$ in increasing order followed by the row $l$. So $\sum_{l\in R'}\gamma_l=0$, and $\gamma_l\neq0$ by (U$_B$).

Let $\alpha\in\mathrm{Ch}(S)$ be an arc with endpoints $r_a$ and $r_b$. List the points $r_l$ in their cyclic order on $\Lambda_S$, starting from $r_b$ and moving away from $\alpha$:
\[
q_0=b,\ q_1,\ \dots,\ q_{n+1}=a .
\]
Define the partial sums
\[
\mu_\ell(\alpha)=\sum_{s=0}^{\ell}\gamma_{q_s},\qquad \ell=0,\dots,n ,
\]
and put
\[
U_\alpha=\operatorname{sgnvar}\big(\mu_0(\alpha),\dots,\mu_n(\alpha)\big),
\]
where zero terms are discarded. The omitted partial sum $\mu_{n+1}(\alpha)=\sum_{l\in R'}\gamma_l$ vanishes.

\begin{proposition}\label{prop:BDF-line}
\begin{enumerate}
\item $N_\alpha\le U_\alpha$.
\item $U_\alpha\equiv l_\alpha\pmod 2$.
\item $U_\alpha$ does not change if the roles of $r_a$ and $r_b$ are exchanged.
\item $U_\alpha\le L_\alpha$.
\end{enumerate}
\end{proposition}

The proof follows \cite{BDF}: after an Abel summation, the zeros on an arc are the zeros of a combination of the functions $1/(u-u_l)$ whose coefficients are the partial sums $\mu_\ell(\alpha)$, and they are bounded by the number of sign changes of these partial sums. The details are given in Appendix~\ref{app:proofs}.

\subsection{The recursive bounds}

\begin{definition}\label{def:recursive-bounds-BDF}
Define integers $\ell^*_S(\mathcal C)$ and $u^*_S(\mathcal C)$, for all $S$ with $|S|\le k$ and all $\mathcal C\in\mathrm{Ch}(S)$, by induction on $d=k-|S|$.
\begin{itemize}
\item $d=0$: $\ell^*_S(\{v_S\})=u^*_S(\{v_S\})=1$.
\item $d=1$: for an arc $\alpha\in\mathrm{Ch}(S)$, $\ell^*_S(\alpha)=l_\alpha$ and $u^*_S(\alpha)=U_\alpha$.
\item $d\ge2$: $\ell^*_S(\mathcal C)$ and $u^*_S(\mathcal C)$ are given by the formulas of Definition~\ref{def:recursive-bounds}, with $\ell$ and $u$ replaced everywhere by $\ell^*$ and $u^*$:
\[
\ell^*_S(\mathcal C)=\max\Big(0,\ \Big\lceil\frac{a^*_S(\mathcal C)}2\Big\rceil,\ \Big\lceil-\frac{b^*_S(\mathcal C)}2\Big\rceil\Big),\qquad
u^*_S(\mathcal C)=\beta_d-\sum_{\mathcal C'\in\mathrm{Ch}(S)\setminus\{\mathcal C\}}\ell^*_S(\mathcal C'),
\]
where $a^*_S(\mathcal C)$ and $b^*_S(\mathcal C)$ are computed from the numbers $\ell^*_{S\cup\{i\}}(F_i(\mathcal C))$ and $u^*_{S\cup\{i\}}(F_i(\mathcal C))$ and the signs $\sigma_S(i)$, $i\in I(\mathcal C)$.
\end{itemize}
\end{definition}

\begin{theorem}\label{thm:recursive-bounds-BDF}
Assume (U$_D$), (U$_B$) and (G). For every $S$ with $|S|\le k$ and every $\mathcal C\in\mathrm{Ch}(S)$,
\[
\ell_S(\mathcal C)\ \le\ \ell^*_S(\mathcal C)\ \le\ Z(S,\mathcal C)\ \le\ u^*_S(\mathcal C)\ \le\ u_S(\mathcal C).
\]
In particular, every chamber $\mathcal C$ of $\mathbb{R}^k\setminus\bigcup_{i\le n+k}\{p_i=0\}$ contains at least $\ell^*_\emptyset(\mathcal C)$ and at most $u^*_\emptyset(\mathcal C)$ solutions of $\Gamma_1=\dots=\Gamma_k=0$.
\end{theorem}

\begin{proof}
We argue by induction on $d$. For $d=0$ all the numbers equal $1$. For $d=1$ we have $\ell_S(\alpha)=\ell^*_S(\alpha)=l_\alpha\le N_\alpha$ by Proposition~\ref{prop:line-zeros}, and $N_\alpha\le U_\alpha=u^*_S(\alpha)\le L_\alpha=u_S(\alpha)$ by Proposition~\ref{prop:BDF-line}.

Let $d\ge2$. The middle inequalities $\ell^*_S(\mathcal C)\le Z(S,\mathcal C)\le u^*_S(\mathcal C)$ follow exactly as in the proof of Theorem~\ref{thm:recursive-bounds}. That proof only uses that the facet bounds satisfy $\ell^*\le Z\le u^*$, which holds by the induction hypothesis. For the outer inequalities, the induction hypothesis gives $\ell\le\ell^*$ and $u^*\le u$ on the facets. Since $a$ is increasing in the lower bounds and decreasing in the upper bounds, and $b$ is the opposite, we get $a^*_S(\mathcal C)\ge a_S(\mathcal C)$ and $b^*_S(\mathcal C)\le b_S(\mathcal C)$. So $\ell^*_S(\mathcal C)\ge\ell_S(\mathcal C)$ for every $\mathcal C\in\mathrm{Ch}(S)$. It follows that $u^*_S(\mathcal C)=\beta_d-\sum_{\mathcal C'\neq \mathcal C}\ell^*_S(\mathcal C')\le\beta_d-\sum_{\mathcal C'\neq \mathcal C}\ell_S(\mathcal C')=u_S(\mathcal C)$.
\end{proof}

\subsection{Positive solutions}

\begin{corollary}\label{cor:possol-BDF}
Under the assumptions of Theorem~\ref{thm:possol-Z}, the number of nondegenerate positive solutions of~\eqref{eq:system} is at most
\[
u^*_\emptyset(\Delta)+\frac12\sum_{j=1}^{k}\ \sum_{S\in\mathcal S_j(\Delta)}u^*_S(\Delta_S),
\]
which is at most the bound of Corollary~\ref{cor:possol-recursive}.
\end{corollary}

\begin{proof}
Combine Theorem~\ref{thm:possol-Z} with Theorem~\ref{thm:recursive-bounds-BDF}.
\end{proof}

For $k=1$ this reads $\#\{\text{positive solutions of }\eqref{eq:system}\}\le U_\Delta+1$, which is the optimal Descartes' rule of signs for circuits of \cite{BDF}.

\subsection{Examples and the case $n=k=2$}\label{subsec:examples}

In both examples $k=2$, and we use the bound of Corollary~\ref{cor:possol-recursive}, which reads
\[
\#\{\text{positive solutions}\}\ \le\ u_\emptyset(\Delta)+\frac12\sum_{i}u_{\{i\}}(\Delta_{\{i\}})+\frac12\,\#\mathcal S_2(\Delta),
\]
where $i$ runs over the edges of $\overline\Delta$, $u_{\{i\}}(\Delta_{\{i\}})=L_{\Delta_{\{i\}}}$ is the Descartes bound of Proposition~\ref{prop:line-zeros} on the edge, and $\mathcal S_2(\Delta)$ is the set of vertices of $\overline\Delta$ at which $B_{i1}$ changes sign. We also give the value of the sharper bound of Corollary~\ref{cor:possol-BDF}, where $u_\emptyset$ is replaced by $u^*_\emptyset$ and $L$ by the partial-sum bound $U$ of Proposition~\ref{prop:BDF-line}. In both cases we assume, as in Remark~\ref{rem:genericity}, that the system, or a small perturbation of it with the same sign data, satisfies (G).

\begin{figure}[htbp]
\centering
\begin{tikzpicture}[scale=2.2]
  \fill[gray!15] (0,0) -- (1.6,0) -- (1.6,1.6) -- (0.782,1.6) -- (0,1/6) -- cycle;
  \draw[->] (-0.2,0) -- (1.75,0) node[right] {$y_1$};
  \draw[->] (0,-0.2) -- (0,1.75) node[above] {$y_2$};
  \draw[thick] (-0.05,0.075) -- (0.8,1.633) node[above left] {$H_1{:}\,+$};
  \node at (1.0,0.6) {$\Delta$};
  \node[left] at (0,1.0) {$H_2{:}\,-$};
  \node[below] at (0.8,0) {$H_3{:}\,-$};
  \node[font=\small,align=center] at (1.25,1.3) {$H_4$ at infinity${:}\,+$};
  \fill (0,1/6) circle (0.025) node[left,font=\scriptsize] {$\{1,2\}$};
  \draw[->,thin] (1.2,0.12) -- (1.58,0.12) node[above,font=\scriptsize] {$\{3,4\}$};
\end{tikzpicture}
\qquad
\begin{tikzpicture}[scale=1.4]
  \foreach \i in {0,...,4}{\coordinate (P\i) at ({90+72*\i}:1);}
  \fill[gray!15] (P0)--(P1)--(P2)--(P3)--(P4)--cycle;
  \draw[thick] (P0)--(P1)--(P2)--(P3)--(P4);
  \draw[thick,dashed] (P4)--(P0);
  \node at (0,0) {$\Delta$};
  \path (P0)--(P1) node[midway,above left,font=\small] {$H_2{:}\,+$};
  \path (P1)--(P2) node[midway,left,font=\small] {$H_1{:}\,-$};
  \path (P2)--(P3) node[midway,below,font=\small] {$H_3{:}\,-$};
  \path (P3)--(P4) node[midway,right,font=\small] {$H_4{:}\,+$};
  \path (P4)--(P0) node[midway,above right,font=\small] {$H_5{:}\,+$};
  \fill (P1) circle (0.045) node[left,font=\scriptsize] {$\{1,2\}$};
  \fill (P3) circle (0.045) node[below right,font=\scriptsize] {$\{3,4\}$};
\end{tikzpicture}
\caption{The positive chambers of Example~\ref{ex:n1k2} (left) and Example~\ref{ex:EH} (right, schematic, the dashed edge being at infinity), with the signs of $b_1$ on the edges. The marked vertices are those where $B_{i1}$ changes sign.}
\label{fig:examples}
\end{figure}
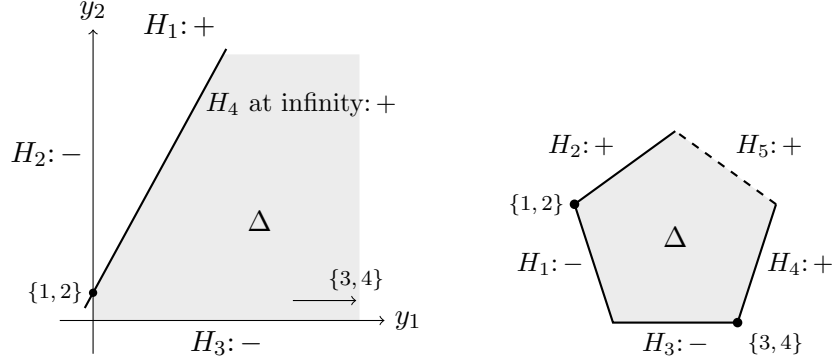

\begin{example}[One polynomial with four monomials]\label{ex:n1k2}
Let $n=1$ and
\[
F(x)=x^3-6x^2+11x-6=(x-1)(x-2)(x-3),
\]
with exponents $a_j=j-1$, $j=1,\dots,4$, and coefficient matrix $C=(-6,11,-6,1)$. So $k=2$, and $F$ has three positive roots, as predicted by Descartes' rule of signs. A Gale dual matrix of $C$, normalized as in Definition~\ref{def:setting}, has rows
\[
D_1=\big(\tfrac{11}6,-1,\tfrac16\big),\quad D_2=(1,0,0),\quad D_3=(0,1,0),\quad D_4=(0,0,1),
\]
so that $p_1=\frac16(1+11y_1-6y_2)$, $p_2=y_1$ and $p_3=y_2$. The positive chamber
\[
\Delta=\{y_1>0,\ y_2>0,\ 6y_2<1+11y_1\}
\]
is unbounded. Its edges lie, in cyclic order, on $H_1,H_2,H_3$ and the line at infinity $H_4$ (Figure~\ref{fig:examples}, left). The four lines cut $\mathbb{RP}^2$ into $7$ chambers. As Gale dual matrix of $A$ we take the columns
\[
b_1=(1,-1,-1,1),\qquad b_2=(0,1,-2,1).
\]
\begin{itemize}
\item \emph{Interior.} The Bézout number is $\beta_2=2n^2=2$. The lower bounds $\ell_\emptyset$ of the six other chambers add up to $1$, so $Z(\emptyset,\Delta)\le u_\emptyset(\Delta)=1$.
\item \emph{Edges.} The Descartes bounds on the edges on $H_1,H_2,H_3,H_4$ are $1,0,0,1$, with sum $2$; the partial-sum bounds are the same.
\item \emph{Vertices.} The signs of $b_1$ are $(+,-,-,+)$. They change at the vertices $\{1,2\}$ and $\{3,4\}$ of $\overline\Delta$, so $\#\mathcal S_2(\Delta)=2$.
\end{itemize}
Corollaries~\ref{cor:possol-recursive} and~\ref{cor:possol-BDF} give $1+\frac12(2+2)=3$, which is sharp. The choice of the first relation matters (Remark~\ref{rem:flag-choice}): for $b_1=(1,-3,3,-1)$, whose signs alternate, $B_{i1}$ changes sign at the four vertices of $\overline\Delta$, $u_\emptyset(\Delta)=2$, and the bound becomes $2+\frac12(2+4)=5$.
\end{example}

\begin{example}[El Hilany's system]\label{ex:EH}
El Hilany \cite{EH} constructed a system of two equations in two variables with seven positive solutions and support
\[
a_1=(0,0),\quad a_2=(6,0),\quad a_3=(3,6),\quad a_4=(-14,7),\quad a_5=(-12,9).
\]
Up to linear combinations of the equations, which change neither the positive solutions nor $\ker C$, the system of \cite{EH} reads
\[
-1+x^6+x^3y^6-t\,x^{-14}y^7=0,\qquad -1+c_0t^7+x^6+(1-c_0t)\,x^3y^6-c\,t\,x^{-12}y^9=0,
\]
with $t=10^{-5}$, $c_0=0.36008$ and $c=(44/31)^{5/6}$. So $n=k=2$ and $n+k+1=5$. A Gale dual matrix of $A$ has columns
\[
b_1=(-15,17,-8,3,3),\qquad b_2=(-9,11,-6,0,4).
\]
For these coefficients, the five lines cut $\mathbb{RP}^2$ into $11$ chambers, and $\Delta$ is the pentagon of this arrangement, which is unbounded in $\mathbb{R}^2$: its edges lie, in cyclic order, on $H_2,H_1,H_3,H_4$ and the line at infinity $H_5$ (Figure~\ref{fig:examples}, right).
\begin{itemize}
\item \emph{Interior.} The Bézout number is $\beta_2=8$. The lower bounds $\ell_\emptyset$ of the ten other chambers add up to $1$, so $Z(\emptyset,\Delta)\le u_\emptyset(\Delta)=7$. With the partial-sum bounds they add up to $2$, and $u^*_\emptyset(\Delta)=6$.
\item \emph{Edges.} On the edges on $H_1,\dots,H_5$, the Descartes bounds are $0,1,2,0,1$, with sum $4$, and the partial-sum bounds are $0,1,0,0,1$, with sum $2$.
\item \emph{Vertices.} The signs of $b_1$ are $(-,+,-,+,+)$. They change at the vertices $\{1,2\}$ and $\{3,4\}$, so $\#\mathcal S_2(\Delta)=2$.
\end{itemize}
Corollary~\ref{cor:possol-recursive} gives $7+\frac12(4+2)=10$, and Corollary~\ref{cor:possol-BDF} gives $6+\frac12(2+2)=8$, while the system has $7$ positive solutions. None of the five sign patterns of the first relation gives better bounds (Remark~\ref{rem:flag-choice}). For comparison, the bound of \cite{BS} for systems of two equations in two variables with five monomials is $15$.
\end{example}

\begin{remark}[Two equations with five monomials]\label{rem:n2k2}
For $n=k=2$, the bound of Corollary~\ref{cor:possol-recursive} only depends on finitely many data: the chamber $\Delta$ of the arrangement of five lines of $\mathbb{RP}^2$, which is unique up to isomorphism and has $5$ triangles, $5$ quadrilaterals and one pentagon; the oriented matroid of the rows $B_1,\dots,B_5$ of $B$; and the sign vector of $b_1$, among five possibilities (Remark~\ref{rem:flag-choice}). An exhaustive enumeration of these data, $11$ chambers and $1320$ labeled oriented matroids, with exact arithmetic and the minimum taken over $b_1$, gives the following. The bound is at most $11$ for $14360$ of the $14520$ types. Its maximum is $12$ when $\Delta$ is a triangle, $13$ when $\Delta$ is a quadrilateral and $11$ when $\Delta$ is the pentagon, except on the family $\mathcal F$ of systems such that (F1) $a_1,\dots,a_5$ are the vertices of a convex pentagon, (F2) $\Delta$ is the pentagon, and (F3) the cyclic order of the $a_i$ along the boundary of their convex hull is the cyclic order of the edges of $\Delta$. On $\mathcal F$ the bound is $15$, the bound of \cite{BS}. The same maxima are obtained with the coherent bounds of Section~\ref{sec:coherent}. The program is provided as an ancillary file. The system of Example~\ref{ex:EH} satisfies (F1) and (F2), but not (F3). We do not know whether a system in $\mathcal F$ can have more than seven positive solutions.
\end{remark}

\section{Coherent chambers and closed-form bounds}\label{sec:coherent}

In this section we restrict the lower bounds of Section~\ref{sec:recursive} to \emph{coherent} chambers, in which no cancellation occurs in the facet-count formula. The resulting bounds are weaker than those of Corollary~\ref{cor:possol-recursive}, but they do not involve the upper bounds $u$ in the recursion, and they can be written in closed form. For $k=1$ they are the Descartes' rule of signs for circuits of \cite{BD}, which only uses coherent arcs (see the Introduction).

Throughout this section we assume (U$_D$), (U$_B$) and (G), and we use the notation of Section~\ref{sec:recursive}.

\subsection{Coherent chambers}

\begin{definition}\label{def:coherent}
Let $|S|=j\le k-1$ and $\mathcal C\in\mathrm{Ch}(S)$. The chamber $\mathcal C$ is \emph{coherent} if the signs $\sigma_S(i)$ of Definition~\ref{def:recursive-bounds} are the same for all $i\in I(\mathcal C)$, namely for all the facets of $\mathcal C$. By convention, the vertices $\{v_S\}$, $|S|=k$, are coherent.
\end{definition}

For $|S|=k-1$, the chambers of the line $\Lambda_S$ are arcs, and an arc is coherent if and only if the signs $\sigma_S$ at its two endpoints agree, or equivalently if $l_\alpha=1$ in Proposition~\ref{prop:line-zeros}. For $S=\emptyset$, $\sigma_\emptyset(i)=\sign B_{i1}$, so a chamber $\mathcal C$ of $\mathbb{R}^k$ is coherent if and only if the first relation $b_1$ has the same sign on all the indices of the facets of $\mathcal C$, including the facet at infinity if $\mathcal C$ is unbounded. Coherence only depends on the face structure of the arrangement, which is determined by the oriented matroid of $C$, and on the signs of the minors $\det B_{U,\{1,\dots,|U|\}}$.

\begin{lemma}\label{lem:coherent-lower}
Let $|S|=j\le k-1$ and let $\mathcal C\in\mathrm{Ch}(S)$ be coherent. Then
\[
\ell_S(\mathcal C)\ \ge\ \Big\lceil\frac12\sum_{i\in I(\mathcal C)}\ell_{S\cup\{i\}}\big(F_i(\mathcal C)\big)\Big\rceil .
\]
\end{lemma}

\begin{proof}
If $\sigma_S(i)>0$ for all $i\in I(\mathcal C)$, the second sum in the definition of $a_S(\mathcal C)$ is empty, so $a_S(\mathcal C)=\sum_{i\in I(\mathcal C)}\ell_{S\cup\{i\}}(F_i(\mathcal C))$, and $\ell_S(\mathcal C)\ge\lceil a_S(\mathcal C)/2\rceil$. If $\sigma_S(i)<0$ for all $i\in I(\mathcal C)$, the first sum in the definition of $b_S(\mathcal C)$ is empty, so $-b_S(\mathcal C)=\sum_{i\in I(\mathcal C)}\ell_{S\cup\{i\}}(F_i(\mathcal C))$, and $\ell_S(\mathcal C)\ge\lceil -b_S(\mathcal C)/2\rceil$.
\end{proof}

\subsection{The coherent recursion}

\begin{definition}\label{def:coherent-bounds}
Define integers $\ell^{\mathrm{coh}}_S(\mathcal C)$ and $u^{\mathrm{coh}}_S(\mathcal C)$, for all $S$ with $|S|\le k$ and all $\mathcal C\in\mathrm{Ch}(S)$, by induction on $d=k-|S|$:
\begin{itemize}
\item $d=0$: $\ell^{\mathrm{coh}}_S(\{v_S\})=u^{\mathrm{coh}}_S(\{v_S\})=1$;
\item $d\ge1$:
\[
\ell^{\mathrm{coh}}_S(\mathcal C)=
\begin{cases}
\Big\lceil\dfrac12\displaystyle\sum_{i\in I(\mathcal C)}\ell^{\mathrm{coh}}_{S\cup\{i\}}\big(F_i(\mathcal C)\big)\Big\rceil & \text{if $\mathcal C$ is coherent,}\\[3mm]
0 & \text{otherwise,}
\end{cases}
\qquad
u^{\mathrm{coh}}_S(\mathcal C)=\beta_d-\sum_{\mathcal C'\in\mathrm{Ch}(S)\setminus\{\mathcal C\}}\ell^{\mathrm{coh}}_S(\mathcal C').
\]
\end{itemize}
\end{definition}

Unlike Definition~\ref{def:recursive-bounds}, the lower bounds $\ell^{\mathrm{coh}}$ are defined by a recursion on lower bounds only, and the upper bounds $u^{\mathrm{coh}}$ are only used at the end.

\begin{proposition}\label{prop:coherent-bounds}
For every $S$ with $|S|\le k$ and every $\mathcal C\in\mathrm{Ch}(S)$,
\[
\ell^{\mathrm{coh}}_S(\mathcal C)\ \le\ \ell_S(\mathcal C)\ \le\ Z(S,\mathcal C)\ \le\ u_S(\mathcal C)\ \le\ u^{\mathrm{coh}}_S(\mathcal C).
\]
For $|S|=k-1$, $\ell^{\mathrm{coh}}_S=\ell_S$ and $u^{\mathrm{coh}}_S=u_S$.
\end{proposition}

\begin{proof}
The middle inequalities are Theorem~\ref{thm:recursive-bounds}. We prove $\ell^{\mathrm{coh}}_S(\mathcal C)\le\ell_S(\mathcal C)$ by induction on $d$. For $d=0$ both equal $1$. Let $d\ge1$. If $\mathcal C$ is not coherent, $\ell^{\mathrm{coh}}_S(\mathcal C)=0\le\ell_S(\mathcal C)$. If $\mathcal C$ is coherent, Lemma~\ref{lem:coherent-lower}, the induction hypothesis and the monotonicity of $x\mapsto\lceil x/2\rceil$ give the inequality. The last inequality follows, since $u_S$ and $u^{\mathrm{coh}}_S$ are obtained by subtracting the lower bounds $\ell_S$ and $\ell^{\mathrm{coh}}_S$ from the same number $\beta_d$.

For $|S|=k-1$, the facets of an arc are its two endpoints, with $\ell=\ell^{\mathrm{coh}}=1$. A coherent arc has $\ell^{\mathrm{coh}}=\lceil 2/2\rceil=1=l_\alpha$, and a non-coherent arc has $\ell^{\mathrm{coh}}=0=l_\alpha$. By the remark after Theorem~\ref{thm:recursive-bounds}, $\ell_S(\alpha)=l_\alpha$. So $\ell^{\mathrm{coh}}_S=\ell_S$ on lines, and hence $u^{\mathrm{coh}}_S=u_S$.
\end{proof}

\begin{theorem}[Coherent Descartes' rule]\label{thm:coherent}
Assume (R), (U$_D$), (U$_B$) and (G). Then the number of nondegenerate positive solutions of~\eqref{eq:system} is at most
\[
u^{\mathrm{coh}}_\emptyset(\Delta)+\frac12\sum_{j=1}^{k}\ \sum_{S\in\mathcal S_j(\Delta)}u^{\mathrm{coh}}_S(\Delta_S).
\]
For $k=1$ this is the Descartes' rule of signs for circuits of \cite{BD}:
\[
\#\{\text{positive solutions}\}\ \le\ 1+n-\#\{\text{coherent arcs }\Delta'\neq\Delta\}.
\]
\end{theorem}

\begin{proof}
Combine Corollary~\ref{cor:possol-recursive} and Proposition~\ref{prop:coherent-bounds}. For $k=1$, $S=\emptyset$ has $|S|=k-1$, so $u^{\mathrm{coh}}_\emptyset(\Delta)=u_\emptyset(\Delta)=L_\Delta=n-\#\{\text{coherent arcs }\Delta'\ne\Delta\}$. The two endpoints of $\Delta$ form $\mathcal S_1(\Delta)$, and each one contributes $\frac12$.
\end{proof}

\subsection{Closed form}

\begin{definition}\label{def:coherent-flags}
Let $|S|=j\le k$ and $\mathcal C\in\mathrm{Ch}(S)$, and put $d=k-j$. A \emph{coherent flag} of $\mathcal C$ is a sequence of chambers
\[
\mathcal C=\mathcal C_d,\ \mathcal C_{d-1},\ \dots,\ \mathcal C_0,
\]
where $\mathcal C_{m-1}=F_{i_m}(\mathcal C_m)$ is a facet of $\mathcal C_m$ for some $i_m\in I(\mathcal C_m)$, and each $\mathcal C_m$ is coherent. In particular $\mathcal C_0$ is a vertex of $\overline{\mathcal C}$, and $\mathcal C_m\in\mathrm{Ch}(S\cup\{i_d,\dots,i_{m+1}\})$. We denote by $\kappa_S(\mathcal C)$ the number of coherent flags of $\mathcal C$. Equivalently, $\kappa_S(\{v_S\})=1$, and for $d\ge1$
\[
\kappa_S(\mathcal C)=\sum_{i\in I(\mathcal C)}\kappa_{S\cup\{i\}}\big(F_i(\mathcal C)\big)\ \text{ if $\mathcal C$ is coherent},\qquad \kappa_S(\mathcal C)=0\ \text{ otherwise.}
\]
\end{definition}

\begin{proposition}\label{prop:closed-form}
For every $S$ with $|S|=j\le k$ and every $\mathcal C\in\mathrm{Ch}(S)$,
\[
\ell_S(\mathcal C)\ \ge\ \ell^{\mathrm{coh}}_S(\mathcal C)\ \ge\ \Big\lceil\frac{\kappa_S(\mathcal C)}{2^{\,k-j}}\Big\rceil .
\]
Consequently, under the assumptions of Theorem~\ref{thm:coherent}, the number of nondegenerate positive solutions of~\eqref{eq:system} is at most
\[
U_\emptyset(\Delta)+\frac12\sum_{j=1}^{k}\ \sum_{S\in\mathcal S_j(\Delta)}U_S(\Delta_S),\qquad\text{where}\quad
U_S(\mathcal C)=\beta_{k-|S|}-\sum_{\mathcal C'\in\mathrm{Ch}(S)\setminus\{\mathcal C\}}\Big\lceil\frac{\kappa_S(\mathcal C')}{2^{\,k-|S|}}\Big\rceil .
\]
\end{proposition}

\begin{proof}
The first inequality is Proposition~\ref{prop:coherent-bounds}. The second one holds by induction on $d=k-j$. For $d=0$ both sides equal $1$. For $d\ge1$, it is clear if $\mathcal C$ is not coherent. If $\mathcal C$ is coherent, the induction hypothesis gives
\[
\ell^{\mathrm{coh}}_S(\mathcal C)\ \ge\ \Big\lceil\frac12\sum_{i\in I(\mathcal C)}\frac{\kappa_{S\cup\{i\}}(F_i(\mathcal C))}{2^{d-1}}\Big\rceil=\Big\lceil\frac{\kappa_S(\mathcal C)}{2^{d}}\Big\rceil .
\]
The bound follows from Theorem~\ref{thm:coherent}, since $u^{\mathrm{coh}}_S(\mathcal C)\le U_S(\mathcal C)$.
\end{proof}

\begin{remark}\label{rem:closed-form}
\begin{enumerate}
\item The bound of Proposition~\ref{prop:closed-form} is the Bézout number of each flat minus a count of coherent flags. It generalizes ``the Bézout number minus the coherent arcs'' of the case $k=1$, for which $\kappa_\emptyset(\Delta')=2$ for a coherent arc and $\lceil2/2\rceil=1$.
\item If all the faces of $\overline{\mathcal C}$ of positive dimension are coherent, then every vertex $v$ of $\overline{\mathcal C}$ is reached by exactly $d!$ coherent flags. By (U$_D$), $v$ lies on exactly $d$ hyperplanes of the arrangement on $\Lambda_S$, every subset of them cuts out a face of $\overline{\mathcal C}$ containing $v$, and the coherent flags ending at $v$ correspond to the $d!$ orders in which these hyperplanes are added. In that case $\kappa_S(\mathcal C)=d!\,f_0(\overline{\mathcal C})$, where $f_0$ is the number of vertices, and $\ell_S(\mathcal C)\ge\lceil d!\,f_0(\overline{\mathcal C})/2^d\rceil$.
\item Like Corollary~\ref{cor:possol-recursive}, the bound depends only on the oriented matroid of $C$ and on the signs of the minors $\det B_{U,\{1,\dots,|U|\}}$. It can be minimized over the flag $b_1,\dots,b_{k-1}$ (Remark~\ref{rem:flag-choice}).
\end{enumerate}
\end{remark}

\appendix
\section{Proofs of technical results}\label{app:proofs}

This appendix contains the proofs of the more technical results of the paper, in the order in which they appear.

\subsection{Proof of Lemma~\ref{lem:perturbed-jacobians}}
\begin{proof}[Proof of Lemma~\ref{lem:perturbed-jacobians}]
We choose $\Theta_k,\Theta_{k-1},\dots,\Theta_1$ in this order, as $\Theta_m=J_m(\Theta)+c_m$, where $c_m$ is a small real constant and $J_m(\Theta)$ only involves the already chosen $\Theta_{m+1},\dots,\Theta_k$. The proof of Proposition~\ref{prop:gammak} only uses the definition of the Jacobians and the degrees of the last rows, so it shows that $J_m(\Theta)$ is a polynomial of degree at most $D_m$. Its coefficients depend continuously on those of $\Theta_{m+1},\dots,\Theta_k$, and $J_m(\Gamma)=\Gamma_m$. We may assume that $\eta$ is at most the number $\eta$ of Proposition~\ref{prop:ends}. So there exists $\eta'>0$ such that $\Theta_m$ is at distance at most $\eta$ from $\Gamma_m$ for every $m$ as soon as $|c_m|\le\eta'$ for every $m$.

For $0\le r\le m$ put $M_{m,r}=\{y\in\Delta:\ \psi_1=\dots=\psi_r=0,\ \Theta_{m+1}=\dots=\Theta_k=0\}$. We show by downward induction on $m$ that the $c_m$ can be chosen so that, for $0\le r\le m$, the differentials of $\psi_1,\dots,\psi_r,\Theta_{m+1},\dots,\Theta_k$ are linearly independent along $M_{m,r}$, which is then a smooth submanifold of dimension $m-r$. For $m=k$ this is the assumption on $\psi$. Assume it for $m$. By Sard's theorem applied to the restrictions of $J_m(\Theta)$ to the manifolds $M_{m,r}$, $0\le r\le m-1$, for almost every $c_m\in\mathbb{R}$ the value $-c_m$ is a regular value of all these restrictions, that is, $0$ is a regular value of the restriction of $\Theta_m$ to each $M_{m,r}$, and then $M_{m-1,r}=M_{m,r}\cap\{\Theta_m=0\}$ has the required property. Moreover $M_{m,m}=Z_m(\Theta)$ consists of nondegenerate solutions, and it is finite by Proposition~\ref{prop:ends}(3). At each of its points $J_m(\Theta)\neq0$, since the solution is nondegenerate and $P\neq0$. So $\Theta_m$ has the sign of $J_m(\Theta)$ on $Z_m(\Theta)$ as soon as $|c_m|$ is small enough. We choose $c_m$ satisfying these conditions, with $|c_m|\le\eta'$.

Then $C_j(\Theta)=M_{j,j-1}$ and $Z_j(\Theta)=M_{j,j}$ for $1\le j\le k$, which gives (i) and (ii), and (iii) holds by the choice of the $c_j$.
\end{proof}

\subsection{Proof of Theorem~\ref{thm:facet-count-unbounded}}
\begin{proof}[Proof of Theorem~\ref{thm:facet-count-unbounded}]
For $\theta\in S^{k-1}$ and $R>0$ put $G^R_m(\theta)=R^{-D_m}\Gamma_m(R\theta)$. By Lemma~\ref{lem:top}, $G^R_m\to\Gamma^{\mathrm{top}}_m$ in $C^1(S^{k-1})$ as $R\to\infty$. Let $S_R$ be the sphere of radius $R$ and write $\mathcal C\cap S_R=R\,\Sigma_R$. By the Minkowski--Weyl theorem, $\mathcal C=Q+\mathcal C_\infty$ with $Q$ a bounded polytope. It follows that $\Sigma_R$ lies in the $O(1/R)$-neighborhood of $\Sigma$, and contains any given compact subset of $\Sigma^\circ$ in its interior for $R$ large.

By (H$_\infty$)(i), $\Gamma^{\mathrm{top}}$ does not vanish on a neighborhood of $\Sigma$, so $G^R$ does not vanish on $\Sigma_R$, which means that $\Gamma$ does not vanish on $\mathcal C\cap S_R$, for $R\ge R_0$. This gives the first assertions: $d(\mathcal C)$ is the degree of $\Gamma/|\Gamma|$ on the boundary of the convex body $\mathcal C_R$ for any $R\ge R_0$, and it is the sum of the indices of the zeros of $\Gamma$ in $\mathcal C$.

By (H), the zeros of $(\Gamma_2,\dots,\Gamma_k)$ on the facets of $\mathcal C$ form a discrete semialgebraic set, hence a finite one, contained in $B(0,R)$ for $R$ large. By compactness, (H$_\infty$)(ii), $C^1$ convergence and the implicit function theorem, for $R$ large the zeros of $(G^R_2,\dots,G^R_k)$ on $\Sigma_R$ lie in its interior, are nondegenerate, and are in bijection with the $N_{n+k+1}$ common zeros $\theta$ of $\Gamma^{\mathrm{top}}_2,\dots,\Gamma^{\mathrm{top}}_k$ on $\Sigma$; moreover, at the zero close to $\theta$, the signs of $G^R_1$ and of the tangential Jacobian determinant of $(G^R_2,\dots,G^R_k)$ are those of $\Gamma^{\mathrm{top}}_1(\theta)$ and of $\det\big(D_{v_j}\Gamma^{\mathrm{top}}_i(\theta)\big)$.

So the zeros of $(\Gamma_2,\dots,\Gamma_k)$ on $\partial\mathcal C_R$ lie in its smooth part and are transversal. The proof of Theorem~\ref{thm:cauchy-general} only uses this, so it applies to $\mathcal C_R$ and expresses $d(\mathcal C)$ as one half of the sum of
\[
\sign\Gamma_1(y)\,\sign\det\big(D_{v_j}\Gamma_i(y)\big)
\]
over these zeros, with $(v_1,\dots,v_{k-1})$ a positive basis of the tangent space of $\partial\mathcal C_R$ at $y$. By Lemma~\ref{lem:sign-identity}, which is local, the zeros on the facets contribute $-\frac{\epsilon}{2}\sum_{i=1}^{n+k}\sign(B_{i1})\,N_i$. On $\mathcal C\cap S_R$ the outward normal at $R\theta$ is $\theta$, and the rescaling does not change any sign, since $\Gamma_m(R\theta)=R^{D_m}G^R_m(\theta)$. By Lemma~\ref{lem:sign-infinity}, the zeros on $\mathcal C\cap S_R$ therefore contribute $-\frac{\epsilon}{2}\sign(B_{n+k+1,1})\,N_{n+k+1}$. Adding the two contributions gives the formula.
\end{proof}

\subsection{Proof of Theorem~\ref{thm:mu-boundary}}
\begin{proof}[Proof of Theorem~\ref{thm:mu-boundary}]
\emph{Step 1: an affine chart.} Choose a linear form $L$ on $\mathbb{R}^{k+1}$ with $L(\xi)>0$; for instance $L=y_{k+1}$ if $\xi$ is finite and $L(y,y_{k+1})=\langle\theta,y\rangle$ if $\xi=[\theta:0]$. The open set $\{L>0\}\subset\mathbb{S}^k$ is identified with the affine hyperplane $A_L=\{L=1\}\subset\mathbb{R}^{k+1}$. On $A_L$ set $q_i=P_i|_{A_L}$, an affine function. By homogeneity of degree $0$,
\[
f_m=\prod_{i=1}^{n+k+1}|q_i|^{B_{im}}\qquad\text{on }A_L\setminus\widehat H .
\]
The functions $q_i$ vanishing at $\xi$ are exactly those with $i\in S$.

Assume the rows $D_i$, $i\in S$, are linearly independent. Since $P_i(\xi)=0$ for $i\in S$ and $L(\xi)\neq0$, the form $L$ is not in the span of the $P_i$, $i\in S$. It follows that the restrictions of these $P_i$ to $\ker L$, which are the linear parts of the $q_i$, are still linearly independent. The functions $z_i=q_i$, $i\in S$, can then be completed by coordinates $w$ on $L_S\cap A_L$ to local coordinates of $A_L$ near $\xi$.

\emph{Step 2: the equations near $\xi$.} Near $\xi$,
\[
\psi_m=\sum_{i\in S}B_{im}\log|q_i|+c_m,\qquad c_m=\sum_{i\notin S}B_{im}\log|q_i|,
\]
where $c_m$ is smooth and bounded near $\xi$ because $q_i(\xi)\neq0$ for $i\notin S$. Let $K_S$ be the $(j-1)\times|S|$ matrix with columns $b_i^{(j-1)}$, $i\in S$, and put $t_i=-\log|q_i|$ for $i\in S$. Then the equations $f_1=\dots=f_{j-1}=1$ read
\[
K_St=c,\qquad c=(c_1,\dots,c_{j-1}).
\]

\emph{Step 3: proof of 1.} Let $y^\nu\in\mu_j$ converge to $\xi$. Then $t_i^\nu\to+\infty$ for every $i\in S$, while $K_St^\nu=c(y^\nu)$ stays bounded. Suppose there is no $\tau>0$ in $\ker K_S$. By Lemma~\ref{lem:stiemke} there is $\mu$ with $\lambda=K_S^T\mu\ge0$ and $\lambda\neq0$. Then
\[
\langle\lambda,t^\nu\rangle=\langle\mu,K_St^\nu\rangle=\langle\mu,c(y^\nu)\rangle
\]
is bounded, whereas $\langle\lambda,t^\nu\rangle\ge\big(\min_{i\in S}t_i^\nu\big)\sum_{i\in S}\lambda_i\to+\infty$. This is a contradiction.

\emph{Step 4: proof of 2.} Use the coordinates $(z,w)$ of Step~1 and fix signs $\rho\in\{\pm1\}^S$; each choice selects a connected component of $V\setminus\widehat H$ near $\xi$. The components contained in $\{y_{k+1}>0\}$ are those with $\rho_{n+k+1}=+1$ when $n+k+1\in S$, and all of them otherwise. Write $z_i=\rho_ie^{-t_i}$ and $w=w(\xi)$. The equations become
\[
K_St=\tilde c(t):=c\big(z(t),w(\xi)\big),
\]
where $\tilde c(t)\to c(\xi)$ as $\min_it_i\to+\infty$, together with its derivatives in $t$, which are $O(e^{-\min_i t_i})$.

Choose $\tau\in\ker K_S$ with all $\tau_i>0$, and a right inverse $K_S^+$ of $K_S$, which exists since $\rank K_S=j-1$. Look for solutions $t=s\tau+K_S^+u$ with $u\in\mathbb{R}^{j-1}$. Since $K_S\tau=0$, the equation becomes
\[
u=\tilde c\big(s\tau+K_S^+u\big).
\]
For $s$ large, the right-hand side maps a closed ball around $c(\xi)$ into itself and is a contraction there. So it has a unique fixed point $u(s)$, with $u(s)\to c(\xi)$. Then $t_i(s)\to+\infty$ for every $i\in S$, so the corresponding points lie in $\mu_j$, in the chosen component, and converge to $\xi$.
\end{proof}

\subsection{Proof of Theorem~\ref{thm:restriction-LS}}
\begin{proof}[Proof of Theorem~\ref{thm:restriction-LS}]
\emph{Step 1: the surviving terms.} Apply Proposition~\ref{prop:gammak}(1) to $\Gamma_m$. If $S\not\subset I$, the product $\prod_{i\notin I}p_i$ contains some $p_s$, $s\in S$, and the term vanishes on $L_S$. For $m<j$ this holds for every $I$, which proves assertion~3. For $m\ge j$ the remaining terms are those with $I=S\cup J$, $J\subset\tilde R\setminus\{t_0\}$, $|J|=m-j$; we list the rows of $S$ first, then those of $J$, each in increasing order.

\emph{Step 2: minors of $B$.} With $G$ as in~\eqref{eq:G}, the first $m$ columns of $G$ vanish outside the rows $1,\dots,m$, and its leading $m\times m$ block is unipotent. So $\det B_{S\cup J,\{1,\dots,m\}}=\det(BG)_{S\cup J,\{1,\dots,m\}}$. By the proof of Proposition~\ref{prop:reduced-format}, the rows of $BG$ indexed by $S$ are those of $\big[\,M_S\,\big|\,0\,\big]$, and its rows indexed by $J$ in the columns $j+1,\dots,m$ are those of $\tilde B$. This matrix is therefore block lower triangular, and
\[
\det B_{S\cup J,\{1,\dots,m\}}=\det M_S\cdot\det\tilde B_{J,\{j+1,\dots,m\}}
\]
(for $|J|=1$ this is Proposition~\ref{prop:reduced-format}(3)).

\emph{Step 3: minors of $D$.} In the affine coordinates $z_s=p_s$ ($s\in S$), $w_t=p_t$ ($t\in T'$), we have $L_S=\{z=0\}$ and $\varphi(w)=(0,w)$. Their Jacobian matrix with respect to $y$ has rows $D'_S$, $D'_{T'}$ and determinant $\det(D_S;D_{T'};D_{n+k+1})$, and the rows $D'_s$ become the unit vectors $e_{z_s}$. Expanding along them, and using $p_i\circ\varphi=\tilde p_i$ for $i\in\tilde R\setminus\{t_0\}$ (by~\eqref{eq:minor-coordinates}, since $P_{t_0}=1$), we get
\[
\det\big(D'_S;D'_J;\nabla\Gamma_{m+1};\dots;\nabla\Gamma_k\big)\circ\varphi=\det(D_S;D_{T'};D_{n+k+1})\cdot\det\big(\nabla\tilde p_J;\nabla(\Gamma_{m+1}\circ\varphi);\dots;\nabla(\Gamma_k\circ\varphi)\big),
\]
with gradients on the right taken with respect to $w$.

\emph{Step 4: induction.} By Steps~1--3, for $m\ge j$,
\begin{equation}\label{eq:restriction-CB}
\Gamma_m\circ\varphi=\lambda\sum_{J}\det\tilde B_{J,\{j+1,\dots,m\}}\prod_{i\in\tilde R\setminus(\{t_0\}\cup J)}\tilde p_i\cdot\det\big(\nabla\tilde p_J;\nabla(\Gamma_{m+1}\circ\varphi);\dots;\nabla(\Gamma_k\circ\varphi)\big).
\end{equation}
By Proposition~\ref{prop:gammak}(1) applied to the pair $(\tilde D,\tilde B)$ of type $(n,k-j)$ (Proposition~\ref{prop:reduced-format}), the right-hand side is $\lambda\tilde\Gamma_m$ once each $\Gamma_{m'}\circ\varphi$ is replaced by $\tilde\Gamma_{m'}$. We get $\Gamma_k\circ\varphi=\lambda\tilde\Gamma_k$, and if $\Gamma_{m'}\circ\varphi=\lambda^{2^{k-m'}}\tilde\Gamma_{m'}$ for $m<m'\le k$, then $\Gamma_m\circ\varphi=\lambda\prod_{m'>m}\lambda^{2^{k-m'}}\,\tilde\Gamma_m=\lambda^{2^{k-m}}\tilde\Gamma_m$. This proves assertion~1 by downward induction. For $m=j$, only $J=\emptyset$ occurs, the product is $\tilde P$, and~\eqref{eq:restriction-CB} with assertion~1 gives assertion~2.
\end{proof}

\subsection{Proof of Theorem~\ref{thm:restriction-LS-infinity}}
\begin{proof}[Proof of Theorem~\ref{thm:restriction-LS-infinity}]
\emph{Step 1: the surviving terms.} By Lemma~\ref{lem:top}, the $\Gamma^{\mathrm{top}}_m$ are the Jacobians of the central system $(\ell_i)_{i\le n+k}$ with the rows $1,\dots,n+k$ of $B$. The proof of Proposition~\ref{prop:gammak}(1) only uses the definition of the Jacobians, so
\[
\Gamma^{\mathrm{top}}_m=\sum_{I}\det B_{I,\{1,\dots,m\}}\prod_{i\in\{1,\dots,n+k\}\setminus I}\ell_i\cdot\det\big(D'_I;\nabla\Gamma^{\mathrm{top}}_{m+1};\dots;\nabla\Gamma^{\mathrm{top}}_k\big),
\]
with $I\subset\{1,\dots,n+k\}$, $|I|=m$. On $V$, a term with $S'\not\subset I$ vanishes. For $m\le j-2$ this holds for every $I$, so $\Gamma^{\mathrm{top}}_m\circ\varphi=0$. For $m\ge j-1$ the remaining terms are those with $I=S'\cup J$, $J\subset R$, $|J|=m-j+1$; we list the rows of $S'$ first, then those of $J$, each in increasing order.

\emph{Step 2: coordinates.} The linear forms $z_s=\ell_s$ ($s\in S'$), $w_t=\ell_t$ ($t\in T'$) and $r=\ell_{t_0}$ are linear coordinates on $\mathbb{R}^k$, with $V=\{z=0\}$, $A_{t_0}=\{z=0,\ r=1\}$ and $\varphi(w)=(0,w,1)$. Their Jacobian matrix with respect to $y$ has rows $D'_{S'}$, $D'_{T'}$, $D'_{t_0}$, and its determinant is $\delta=\det(D_{S'};D_{T'};D_{t_0};D_{n+k+1})$. The rows $D'_s$, $s\in S'$, become unit vectors $e_{z_s}$, so on $V$
\[
\det\big(D'_{S'};D'_J;\nabla\Gamma^{\mathrm{top}}_{m+1};\dots;\nabla\Gamma^{\mathrm{top}}_k\big)=\delta\cdot\det\big(\nabla_{(w,r)}\ell_J;\nabla_{(w,r)}\Gamma^{\mathrm{top}}_{m+1};\dots;\nabla_{(w,r)}\Gamma^{\mathrm{top}}_k\big),
\]
a determinant of size $K=k-j+1$.

\emph{Step 3: Euler's identity.} Each row is the gradient of a homogeneous function $f$, of degree $1$ for $f=\ell_i$ and $D_{m'}$ for $f=\Gamma^{\mathrm{top}}_{m'}$. On $V$, Euler's identity reads $\sum_tw_t\partial_{w_t}f+r\,\partial_rf=(\deg f)f$. At $r=1$, add to the $r$-column $\sum_tw_t$ times the $w_t$-columns; the new $r$-column has entries $\ell_i$, $i\in J$, and $D_{m'}\Gamma^{\mathrm{top}}_{m'}$. Expand along it and compose with $\varphi$. Since $\ell_i\circ\varphi=\tilde p_i$ for $i\in R\setminus\{t_0\}$ and $\ell_{t_0}\circ\varphi=1$, we put $\tilde p_{t_0}=1$, so that $\nabla\tilde p_{t_0}=0$ and $\prod_{i\in\{1,\dots,n+k\}\setminus(S'\cup J)}\ell_i\circ\varphi=\prod_{i\in R\setminus J}\tilde p_i$. If $i$ is the $a$-th element of $J$, its row has position $a$, and moving the row $i$ of $B_{S'\cup J,\{1,\dots,m\}}$ to the last position gives the sign $(-1)^{|J|-a}$. Since $(-1)^{a+K}(-1)^{|J|-a}=(-1)^{k+m}$, we get, with $J'=J\setminus\{i\}$,
\begin{multline*}
\Gamma^{\mathrm{top}}_m\circ\varphi=(-1)^{k+m}\delta\sum_{J'}\Big(\sum_{i}\det B_{(S',J',i),\{1,\dots,m\}}\Big)\prod_{i\in R\setminus J'}\tilde p_i\\
\cdot\det\big(\nabla\tilde p_{J'};\nabla(\Gamma^{\mathrm{top}}_{m+1}\circ\varphi);\dots;\nabla(\Gamma^{\mathrm{top}}_k\circ\varphi)\big)+E_m ,
\end{multline*}
where $J'$ runs over the subsets of $R$ with $|J'|=m-j$, the gradients are taken with respect to $w$, and $E_m$ collects the terms coming from the rows $\nabla\Gamma^{\mathrm{top}}_{m'}$. Each term of $E_m$ is a polynomial multiple of some $\Gamma^{\mathrm{top}}_{m'}\circ\varphi$, $m'>m$. For $m=j-1$ we have $J=\emptyset$, so $\Gamma^{\mathrm{top}}_{j-1}\circ\varphi=E_{j-1}$, which proves assertion~4.

\emph{Step 4: the zero column sums.} Let $m\ge j$. The inner sum may be taken over all $i\in\{1,\dots,n+k\}$, since the other terms have a repeated row. As the columns of $B$ sum to zero, it equals $-\det B_{(S',J',n+k+1),\{1,\dots,m\}}$. Moving the row $n+k+1$ in front of $J'$ gives the sign $(-1)^{m-j}$, and the rows $S',n+k+1$ are those of $S$ in increasing order. By the computation of Step~2 of the proof of Theorem~\ref{thm:restriction-LS}, which only uses that the rows of $BG$ indexed by $S$ are those of $\big[\,M_S\,\big|\,0\,\big]$ and that its rows indexed by $J'\subset R$ in the columns $j+1,\dots,m$ are those of $\tilde B$, we get
\[
\sum_{i}\det B_{(S',J',i),\{1,\dots,m\}}=-(-1)^{m-j}\det M_S\cdot\det\tilde B_{J',\{j+1,\dots,m\}} .
\]
The total sign is $(-1)^{k+m}\cdot(-1)^{m-j+1}=(-1)^{k-j+1}$, so the coefficient is $\lambda$. The terms with $t_0\in J'$ vanish since $\nabla\tilde p_{t_0}=0$. Hence, for $m\ge j$,
\begin{equation}\label{eq:restriction-CB-infinity}
\Gamma^{\mathrm{top}}_m\circ\varphi=\lambda\sum_{J'}\det\tilde B_{J',\{j+1,\dots,m\}}\prod_{i\in R\setminus(\{t_0\}\cup J')}\tilde p_i\cdot\det\big(\nabla\tilde p_{J'};\nabla(\Gamma^{\mathrm{top}}_{m+1}\circ\varphi);\dots;\nabla(\Gamma^{\mathrm{top}}_k\circ\varphi)\big)+E_m ,
\end{equation}
with $J'\subset R\setminus\{t_0\}=\tilde R\setminus\{t_0\}$. Up to $E_m$, this is~\eqref{eq:restriction-CB}.

\emph{Step 5: induction.} For $m=k$ we have $E_k=0$, and Proposition~\ref{prop:gammak}(1) applied to $(\tilde D,\tilde B)$ gives assertion~1. Let $j\le m<k$ and assume that for $m<m'\le k$, $\Gamma^{\mathrm{top}}_{m'}\circ\varphi\equiv\lambda^{2^{k-m'}}\tilde\Gamma_{m'}$ modulo $I_{m'+1}$, and that $\Gamma^{\mathrm{top}}_{m'}\circ\varphi,\dots,\Gamma^{\mathrm{top}}_k\circ\varphi$ generate $I_{m'}$. Then $E_m\in I_{m+1}$. In the determinants of~\eqref{eq:restriction-CB-infinity}, each row $\nabla(\Gamma^{\mathrm{top}}_{m'}\circ\varphi)$ is $\lambda^{2^{k-m'}}\nabla\tilde\Gamma_{m'}$ plus a combination of the rows $\nabla\tilde\Gamma_{m''}$, $m''>m'$, plus multiples of the $\tilde\Gamma_{m''}$. Row operations, from the last row upwards, remove the combinations, and the multiples contribute elements of $I_{m+1}$. By Proposition~\ref{prop:gammak}(1) applied to $(\tilde D,\tilde B)$,
\[
\Gamma^{\mathrm{top}}_m\circ\varphi\equiv\lambda\prod_{m'=m+1}^k\lambda^{2^{k-m'}}\,\tilde\Gamma_m=\lambda^{2^{k-m}}\tilde\Gamma_m\pmod{I_{m+1}},
\]
for $m\ge j+1$, which is assertion~2, and the ideal generated by $\Gamma^{\mathrm{top}}_m\circ\varphi,\dots,\Gamma^{\mathrm{top}}_k\circ\varphi$ is $I_m$. For $m=j$, only $J'=\emptyset$ occurs, the product is $\tilde P$, and the same argument gives assertion~3. The equality of zero sets follows from the equality of ideals.
\end{proof}

\subsection{Proof of Proposition~\ref{prop:line-zeros}}
\begin{proof}[Proof of Proposition~\ref{prop:line-zeros}]
\emph{Proof of~1.} Write $\tilde p_l(w)=a_l(w-r_l)$ for $l\in R'\setminus\{t_0\}$, with $a_l\neq0$. By Proposition~\ref{prop:gammak}(1) for the reduced system, which has $\tilde k=1$,
\[
\tilde\Gamma_k=\sum_{l\in R'\setminus\{t_0\}}\gamma_l\,a_l\prod_{m\neq l}\tilde p_m .
\]
At $w=r_l$ all the terms but the $l$-th vanish, so $\tilde\Gamma_k(r_l)=\gamma_la_l\prod_{m\neq l}a_m(r_l-r_m)\neq0$, since $\gamma_l\neq0$ and the points $r_m$ are distinct. Each term has degree $n$ with leading coefficient $\gamma_l\prod_ma_m$. Since $\sum_{l\neq t_0}\gamma_l=-\gamma_{t_0}$, the coefficient of $w^n$ in $\tilde\Gamma_k$ is $-\gamma_{t_0}\prod_ma_m\neq0$. So $\tilde\Gamma_k$ has degree exactly $n$ and no zero at the points $r_l$. Its real zeros therefore lie in the arcs, and their number, counted with multiplicity, is at most $n$ and congruent to $n$ modulo $2$, since the nonreal zeros come in conjugate pairs.

For the other assertions, consider the $1$-form on $\mathbb{P}^1\setminus\{r_l\}_{l\in R'}$
\[
\omega=\sum_{l\in R'}\gamma_l\,d\log|\ell_l| ,
\]
where $\ell_l$ is a linear form on $\mathbb{R}^2$ vanishing at $r_l$. Since $\sum_l\gamma_l=0$, $\omega$ does not depend on the choice of the $\ell_l$. In the affine coordinate $w$ we have $\ell_{t_0}=1$ and $\ell_l=\tilde p_l$ for $l\neq t_0$, so
\[
\omega=\tilde\psi_k'(w)\,dw=\frac{\tilde\Gamma_k(w)}{\tilde P(w)}\,dw .
\]
In particular, on each arc the zeros of $\omega$, with their multiplicities, are those of $\tilde\Gamma_k$.

\emph{Proof of~2 and~3.} Choose a projective coordinate $u$ on $\mathbb{P}^1$ sending $r_b$ to $\infty$, and write $u_l$ for the image of $r_l$. Zeros of $\omega$ and their multiplicities are preserved. In this coordinate $\ell_b$ becomes constant, and
\[
\omega=F(u)\,du,\qquad F(u)=\sum_{l\neq b}\frac{\gamma_l}{u-u_l}.
\]
The arc $\alpha$ becomes one of the two intervals $(u_a,+\infty)$ or $(-\infty,u_a)$. Assume it is $(u_a,+\infty)$; the other case is symmetric. As $u\to u_a^+$ the term $\gamma_a/(u-u_a)$ dominates, so $\sign F=\sign\gamma_a$, and as $u\to+\infty$, $F(u)\sim\big(\sum_{l\neq b}\gamma_l\big)/u=-\gamma_b/u$. So $N_\alpha$ is odd if and only if $\gamma_a\gamma_b>0$, that is, $\sigma_a\sigma_b>0$, since $\sign\gamma_l=\sign(\det M')\,\sigma_l$. This proves~2.

Let $q_0=b,q_1,\dots,q_{n+1}=a$ be the points $r_l$ in cyclic order starting from $r_b$, away from $\alpha$, so that $u_{q_1}<\dots<u_{q_{n+1}}=u_a$. On $\alpha$ the linear forms $u-u_l$, $l\neq b$, are positive, and $F$ is a linear combination of their reciprocals. By the Descartes' rule of signs for such functions \cite[Prop.~4.3]{BD},
\[
N_\alpha\le V_1:=\operatorname{sgnvar}\big(\sigma_{q_1},\dots,\sigma_{q_{n+1}}\big).
\]
Exchanging the roles of $r_a$ and $r_b$ gives
\[
N_\alpha\le V_2:=\operatorname{sgnvar}(\sigma_{q_0},\dots,\sigma_{q_n}).
\] The sequence $(\sigma_{q_0},\dots,\sigma_{q_{n+1}})$ has $W-1+l_\alpha$ sign changes, since closing it up cyclically adds one sign change exactly when $\sigma_a\sigma_b<0$. Removing its first or its last term decreases this number by one exactly when the arc $(r_{q_0},r_{q_1})$, respectively $(r_{q_n},r_{q_{n+1}})$, adjacent to $\alpha$ has endpoints of opposite signs. It follows that $\min(V_1,V_2)$ equals $L_\alpha=W-2+l_\alpha$ or $L_\alpha+1$. Since $W$ is even, $L_\alpha\equiv l_\alpha\equiv N_\alpha\pmod 2$ by~2, and therefore $N_\alpha\le L_\alpha$. Finally, an arc $\alpha'$ has $l_{\alpha'}=1$ exactly when the signs at its endpoints agree, so $\sum_{\alpha'}l_{\alpha'}=(n+2)-W$ and $L_\alpha=n-\sum_{\alpha'\neq\alpha}l_{\alpha'}$.
\end{proof}

\subsection{Proof of Lemma~\ref{lem:chart-independence}}
\begin{proof}[Proof of Lemma~\ref{lem:chart-independence}]
Since $\sum_iB_{ij}=0$, the function $\psi_j=\sum_iB_{ij}\log|P_i|$ is homogeneous of degree $0$, and in every chart it is the function $\psi_j$ of Section~\ref{sec:KR}. Multiplying the $P_i$ by constants changes $\psi_j$ by an additive constant. We get $\psi'_j\circ\tau=\psi_j+c_j$. By the chain rule,
\[
\Gamma'_k\circ\tau=\frac{P'\circ\tau}{P}\,(\det D\tau)^{-1}\,\Gamma_k ,
\]
and $g_k=(P'\circ\tau/P)(\det D\tau)^{-1}$ vanishes nowhere on the common chart minus the arrangement.

Assume the claim for all $m'>m$. In the determinant defining $\Gamma'_m\circ\tau$, the chain rule replaces derivatives with respect to $y'$ by derivatives with respect to $y$, at the cost of the factor $(\det D\tau)^{-1}$. Moreover
\[
\nabla\big(\Gamma'_{m'}\circ\tau\big)=g_{m'}\nabla\Gamma_{m'}+\sum_{m''>m'}h_{m'm''}\nabla\Gamma_{m''}+\Gamma_{m'}\nabla g_{m'}+\sum_{m''>m'}\Gamma_{m''}\nabla h_{m'm''} .
\]
Row operations, performed from the last row upwards, remove the terms $h\,\nabla\Gamma_{m''}$. Expanding multilinearly, the terms containing $\Gamma_{m'}\nabla g$ or $\Gamma_{m''}\nabla h$ are multiples of $\Gamma_{m'}$ or $\Gamma_{m''}$ with $m',m''>m$. The remaining term is $(P'\circ\tau/P)(\det D\tau)^{-1}\prod_{m'>m}g_{m'}\cdot\Gamma_m$. This proves the formula.

The statement about zero sets follows by downward induction on $m$. At a common zero of all the $\Gamma_m$, the Jacobian matrices of the two families differ by an invertible triangular matrix and by $D\tau$, which proves the last assertion.
\end{proof}

\subsection{Proof of Lemma~\ref{lem:inheritance}}
\begin{proof}[Proof of Lemma~\ref{lem:inheritance}]
1. By Lemma~\ref{lem:minor-coordinates}, the rows of $D_{S\cup U}D_{S\cup T}^{-1}$ indexed by $S$ are unit vectors supported in the columns $S$, and the block indexed by $U$ and the columns $T$ is $\tilde D_U$. Taking determinants gives the first formula. The second one is Schur's determinant formula applied to the matrix with rows $S\cup U$, $S$ first, and columns $1,\dots,j+|U|$.

2. For the exponents, recall from the proof of Proposition~\ref{prop:reduced-format} that, for $m>j$, the column $\tilde b_m$ of $\tilde B$ is the restriction to the rows $\tilde R$ of the unique vector of the form $b_m+\sum_{l\le j}\alpha_lb_l$ which vanishes on the rows $S$; uniqueness holds because $\det M_S\neq0$. Reducing $(\tilde D,\tilde B)$ along $\{i\}$ therefore gives, for $m>j+1$, the restriction of a vector of the form $b_m+\sum_{l\le j+1}\alpha'_lb_l$ which vanishes on the rows $S$ and on the row $i$. Reducing $(D,B)$ along $S\cup\{i\}$ gives a vector of the same form with the same vanishing conditions. Since $\det B_{S\cup\{i\},\{1,\dots,j+1\}}\neq0$ by (U$_B$), such a vector is unique, so the two reduced exponent matrices coincide. In particular they do not depend on the order of the rows of $S\cup\{i\}$. For the affine functions: by Lemma~\ref{lem:minor-coordinates}, both reduced systems consist of the restrictions of the forms $P_l$, $l\notin S\cup\{i\}$, to $\Lambda_{S\cup\{i\}}$, written in two systems of projective coordinates and up to constant factors.
\end{proof}

\subsection{Proof of Lemma~\ref{lem:vertices}}
\begin{proof}[Proof of Lemma~\ref{lem:vertices}]
At a finite vertex $v=\Lambda_{S'}$ only the term $I=S'$ of Proposition~\ref{prop:gammak} survives, so
\[
\Gamma_k(v)=\det B'_{S'}\,\det D'_{S'}\prod_{l\notin S'}p_l(v)\neq0 .
\]
Now let $v=[\theta:0]$ with $S'=S''\cup\{n+k+1\}$ and $|S''|=k-1$, and put $R=[n+k]\setminus S''$. Choose for $\theta$ the vector of cofactors of the $(k-1)\times k$ matrix $D'_{S''}$, so that $\det(D'_{S''};x)=\langle x,\theta\rangle$ for all $x$. By Lemma~\ref{lem:top}, $\Gamma^{\mathrm{top}}_k=\sum_{|I|=k}\det B'_I\det D'_I\prod_{l\notin I}\ell_l$. At $\theta$, the terms with $S''\not\subset I$ vanish. For $I=S''\cup\{i\}$ with $i\in R$, and rows in the order $S''$ then $i$ in both determinants,
\[
\det D'_I\prod_{l\notin I}\ell_l(\theta)=\ell_i(\theta)\prod_{l\in R\setminus\{i\}}\ell_l(\theta)=\prod_{l\in R}\ell_l(\theta).
\]
Hence, using the zero column sums,
\[
\Gamma^{\mathrm{top}}_k(\theta)=\prod_{l\in R}\ell_l(\theta)\cdot\det\Big(B_{S''};\ \sum_{i\in R}B_i\Big)_{\{1,\dots,k\}}
=-\prod_{l\in R}\ell_l(\theta)\cdot\det B_{S''\cup\{n+k+1\},\{1,\dots,k\}} .
\]
This is nonzero by (U$_D$) (which gives $\ell_l(\theta)\neq0$ for $l\in R$) and by (U$_B$).
\end{proof}

\subsection{Proof of Lemma~\ref{lem:facets-recursive}}
\begin{proof}[Proof of Lemma~\ref{lem:facets-recursive}]
\emph{Step 1: restriction to faces.} Let $S'\supsetneq S$ with $|S'|=j'\le k-1$, and let $F'\in\mathrm{Ch}(S')$ with $F'\subset\overline{\mathcal C}$. Apply Theorem~\ref{thm:restriction-LS} (if $t_0\notin S'$) or Theorem~\ref{thm:restriction-LS-infinity} (if $t_0\in S'$) to the system $(\tilde D,\tilde B)$ and to the set $S'\setminus S$. The hypotheses hold by Lemma~\ref{lem:inheritance}(1); in particular the relevant pivot is nonzero by (U$_B$). Combine this with Lemma~\ref{lem:inheritance}(2) and Lemma~\ref{lem:chart-independence}. On $F'$ we obtain, with the top-degree parts when $F'$ lies at infinity of the chart:
\begin{itemize}
\item for every $m\ge j'+1$, the common zeros of $\tilde\Gamma_m,\dots,\tilde\Gamma_k$ are those of the corresponding Jacobians of a reduced system along $S'$; in particular the common zeros of $\tilde\Gamma_{j'+1},\dots,\tilde\Gamma_k$ in $F'$ are the $Z(S',F')$ nondegenerate zeros of (G);
\item $\tilde\Gamma_{j'}$ is congruent, modulo the ideal generated by $\tilde\Gamma_{j'+1},\dots,\tilde\Gamma_k$, to a nonzero constant times $\tilde P'\det\partial(\Gamma'_{j'+1},\dots,\Gamma'_k)/\partial w$, where the $\Gamma'$ are the Jacobians of the reduced system along $S'$;
\item for $j+1\le m\le j'-1$, $\tilde\Gamma_m$ vanishes on $F'$ or lies in the ideal generated by $\tilde\Gamma_{m+1},\dots,\tilde\Gamma_k$.
\end{itemize}

\emph{Step 2.} Let $x\in F'$ with $\tilde\Gamma_{j'+1}(x)=\dots=\tilde\Gamma_k(x)=0$. If $j'<k$, then $\tilde P'(x)\neq0$ because $x$ lies in a chamber, and the Jacobian determinant is nonzero by (G), so $\tilde\Gamma_{j'}(x)\neq0$. If $j'=k$, then $F'$ is a vertex and $\tilde\Gamma_k(x)\neq0$ by Lemma~\ref{lem:vertices} applied to $(\tilde D,\tilde B)$; this case does not use Step~1.

\emph{Step 3.} Let $x\in F'$ be a common zero of $\tilde\Gamma_{j+2},\dots,\tilde\Gamma_k$. If $j'\ge j+2$, this family contains $\tilde\Gamma_{j'},\dots,\tilde\Gamma_k$, which contradicts Step~2. So such zeros only lie in facets ($j'=j+1$). There they are nondegenerate, which means transversal, and $\tilde\Gamma_{j+1}(x)\neq0$ by Step~2, so $\tilde\Gamma$ does not vanish on $\partial\tilde{\mathcal C}$. This is (H) for the finite facets. For the facet on $H_{t_0}$ the same argument, with top-degree parts, gives (H$_\infty$). Assertion~2 is Step~1 with $S'=S\cup\{i\}$ and $m=j+2$.
\end{proof}

\subsection{Proof of Proposition~\ref{prop:ends}}
\begin{proof}[Proof of Proposition~\ref{prop:ends}]
By Lemma~\ref{lem:closed-faces} and ($\ast$), applied to the finitely many closed faces of $\overline\Delta$, we may choose $\eta$ such that the following holds for $0\le j\le k-1$:
\begin{itemize}
\item[(P1)] for every $S'$ with $j+1\le|S'|\le k$ such that $\Delta_{S'}$ exists, $\Theta^h_{j+1},\dots,\Theta^h_k$ have no common zero on $\overline{\Delta_{S'}}$;
\item[(P2)] for every $S$ with $|S|=j$ such that $\Delta_S$ exists, the restrictions of $\Theta^h_{j+1},\dots,\Theta^h_k$ to $L_S$ have exactly $Z(S,\Delta_S)$ common zeros in $\overline{\Delta_S}$, all in $\Delta_S$ and nondegenerate, and, if $j\ge1$, $\Theta^h_j$ does not vanish at them.
\end{itemize}
For $j=0$ and $S=\emptyset$, (P2) says that $\Theta_1,\dots,\Theta_k$ have exactly $Z(\emptyset,\Delta)$ common zeros in $\Delta$, which is~1.

For~2, let $E_j=\overline{C_j(\Theta)}\setminus C_j(\Theta)$. Since $C_j(\Theta)$ is closed in $\Delta$, $E_j\subset\partial\overline\Delta$, where $\overline\Delta$ may meet the hyperplane at infinity. Let $\xi\in E_j$. Let $S'=S(\xi)$ and $j'=|S'|$, so that $\xi$ lies in the relative interior of the face $\Delta_{S'}$, and $1\le j'\le k$ by (U$_D$).

(a) \emph{The vectors $b_i^{(j-1)}$, $i\in S'$, positively span $\mathbb{R}^{j-1}$, and $j'\ge j$.} If $j=1$ there is nothing to prove. If $j\ge2$, then $C_j(\Theta)\subset\mu_j$, so by Theorem~\ref{thm:mu-boundary}(1) the vectors $b_i^{(j-1)}$, $i\in S'$, admit a positive linear relation. By (U$_B$), any $j-1$ of them are linearly independent, so $j'\ge j$ and they positively span $\mathbb{R}^{j-1}$, as in the proof of Corollary~\ref{cor:mu-generic}.

(b) \emph{$j'=j$, $S'\in\mathcal S_j(\Delta)$, and $\xi$ is one of the $Z(S',\Delta_{S'})$ points of~(P2).} If $j=k$, then $j'=k$ by~(a), $\xi$ is the vertex $\Delta_{S'}$, and $Z(S',\Delta_{S'})=1$. Let $j<k$. By continuity, $\Theta^h_{j+1}(\xi)=\dots=\Theta^h_k(\xi)=0$. By~(P1), $j'\le j$, so $j'=j$, and $\xi$ is one of the common zeros of~(P2) for $S=S'$. By~(a), $S'\in\mathcal S_j(\Delta)$.

(c) \emph{Exactly one end of $C_j(\Theta)$ converges to $\xi$.} Work in an affine chart $\{P_{t_0}>0\}$ containing $\xi$, and use the local coordinates $(z,w)$ of Step~1 of the proof of Theorem~\ref{thm:mu-boundary}, with $z_i=q_i$ for $i\in S'$; the chamber $\Delta$ is $\{z_i>0,\ i\in S'\}$ near $\xi$. In this chart, $C_j(\Theta)$ is defined in $\mu_j$ by the vanishing of the functions $\Theta'_m=\Theta^h_m/P_{t_0}^{D_m}$, $m=j+1,\dots,k$. By~(P2), their restrictions to $\{z=0\}$ have a nondegenerate common zero at $\xi$ (for $j=k$ there are no such functions). By the implicit function theorem, near $\xi$ their common zero set $V$ is the graph $w=w(z)$ of a smooth map defined near $z=0\in\mathbb{R}^j$. If $j=1$, then near $\xi$ the set $C_1(\Theta)=V\cap\Delta$ is the single arc $\{w=w(z),\ z>0\}$. If $j\ge2$, then near $\xi$ the curve $C_j(\Theta)$ is defined in $V\cap\Delta$ by $K_{S'}t=c\big(z,w(z)\big)$, with $t_i=-\log z_i$, as in Step~2 of that proof. The $(j-1)\times j$ matrix $K_{S'}$ has rank $j-1$, and its kernel is spanned by a vector $\tau$ with positive entries. Write $t=s\tau+K_{S'}^+u$, with $K_{S'}^+$ a right inverse of $K_{S'}$ whose image is complementary to $\mathbb{R}\tau$. Then $u=K_{S'}t=c$ stays close to $c(\xi)$, and $s\to+\infty$ as $\xi$ is approached. Since $w(\cdot)$ is $C^1$, the derivatives of $c\big(z,w(z)\big)$ with respect to $t$ are still $O(e^{-\min_it_i})$, so the contraction argument of Step~4 of that proof applies: for every large $s$ there is exactly one solution $u$ in a fixed ball around $c(\xi)$. Every point of $C_j(\Theta)$ close enough to $\xi$ has $u$ in this ball and $s$ large. So near $\xi$, $C_j(\Theta)$ is a single arc ending at $\xi$.

(d) \emph{Conclusion.} By (a) and (b), $E_j$ is contained in the finite set of the $\sum_{S\in\mathcal S_j(\Delta)}Z(S,\Delta_S)$ points of~(P2). A noncompact component of $C_j(\Theta)$ is a smooth curve homeomorphic to $\mathbb{R}$, closed in $\Delta$. The limit set of each of its two ends is a nonempty connected subset of this finite set, hence a single point. By (c), each point of $E_j$ is the limit of exactly one end of $C_j(\Theta)$. We get $2b(C_j(\Theta))\le\#E_j\le\sum_{S\in\mathcal S_j(\Delta)}Z(S,\Delta_S)$.

For~3, the set $Z_m(\Theta)$ is discrete, since its points are nondegenerate solutions, and closed in $\Delta$. Since $\overline\Delta$ is compact, it suffices to show that $Z_m(\Theta)$ has no accumulation point $\xi\in\partial\overline\Delta$. Such a point lies in the closure of $C_{m+1}(\Theta)$ if $m<k$, and of $\mu_k$ if $m=k$. Steps~(a) and~(b) above, which do not use condition~(i), apply to these sets with $j=m+1$, respectively $j=k$; let $S'=S(\xi)$.

If $m<k$, then by~(b) for $j=m+1$, $\xi$ is one of the common zeros of~(P2) for $S=S'$ and $j=m+1$ (or, if $m+1=k$, a vertex). By~(P2), respectively by~(P1) for $j=k-1$, $\Theta^h_{m+1}(\xi)\neq0$. But $\Theta_{m+1}$ vanishes on $Z_m(\Theta)$, so $\Theta^h_{m+1}(\xi)=0$ by continuity, a contradiction.

If $m=k$, then by~(a) $|S'|=k$, $\xi$ is a vertex, and the vectors $b_i^{(k-1)}$, $i\in S'$, admit a positive relation $\tau$. In the coordinates of Step~2 of the proof of Theorem~\ref{thm:mu-boundary}, $\psi_k=-\sum_{i\in S'}B_{ik}t_i+c_k$ with $c_k$ bounded near $\xi$, and on $\mu_k$ we have $t=s\tau+K_{S'}^+u$ with $u=c$ bounded and $s\to+\infty$ as $\xi$ is approached. It follows that $\psi_k=-s\sum_{i\in S'}\tau_iB_{ik}+O(1)$. Since $\sum_{i\in S'}\tau_ib_i^{(k-1)}=0$ and $\det B_{S',\{1,\dots,k\}}\neq0$ by (U$_B$), we have $\sum_{i\in S'}\tau_iB_{ik}\neq0$, so $|\psi_k|\to+\infty$ at $\xi$, and $\psi_k$ has no zero near $\xi$, a contradiction.
\end{proof}

\subsection{Proof of Proposition~\ref{prop:BDF-line}}
\begin{proof}[Proof of Proposition~\ref{prop:BDF-line}]
1. Choose a projective coordinate $u$ on $\Lambda_S$ such that the point $u=\infty$ lies in an arc different from $\alpha$. Then all the points $r_l$ are finite, say $r_l=\{u=u_l\}$, and $\alpha$ is a bounded interval. Let $h_l=1/(u-u_l)$ on $\alpha$. Since $\sum_l\gamma_l=0$, the form $\omega=\sum_l\gamma_l\,d\log|\ell_l|$ of the proof of Proposition~\ref{prop:line-zeros} reads
\[
\omega=G'(u)\,du,\qquad G'=\sum_{l\in R'}\gamma_l\,h_l .
\]
Its zeros on $\alpha$, counted with multiplicity, are the $N_\alpha$ zeros of $\tilde\Gamma_k$ on $\alpha$. By Abel summation, using $\mu_{n+1}=0$,
\[
G'=\sum_{\ell=0}^{n}\mu_\ell(\alpha)\,\big(h_{q_\ell}-h_{q_{\ell+1}}\big).
\]
For fixed $u\in\alpha$, the function $r\mapsto1/(u-r)$ is increasing as $r$ runs from $r_b$ away from $\alpha$, through $\infty$, back to $r_a$. So the ordering $q_0,\dots,q_{n+1}$ is an ordering in the sense of \cite[Def.~2.1]{BDF}, and each difference $h_{q_\ell}-h_{q_{\ell+1}}$ has constant sign on $\alpha$. Bihan, Dickenstein and Forsgård show that the functions $h_{q_\ell}-h_{q_{\ell+1}}$, $\ell=0,\dots,n$, satisfy Descartes' rule of signs on $\alpha$, via the Wronskian criterion \cite[Prop.~2.6]{BDF}. Their argument does not use the integrality of the coefficients. So the number of zeros of $G'$ on $\alpha$, counted with multiplicity, is at most $\operatorname{sgnvar}(\mu_0,\dots,\mu_n)$; see the proof of \cite[Thm.~2.4]{BDF}.

2. We have $\mu_0=\gamma_b$ and, since $\mu_{n+1}=0$, $\mu_n=-\gamma_a$. Both are nonzero. The parity of the number of sign changes of a sequence equals $1$ if and only if its first and last nonzero terms have opposite signs. Here this means that $\gamma_a\gamma_b>0$, and so $l_\alpha=1$.

3. Starting from $r_a$ instead gives the partial sums $\mu'_\ell=\sum_{s=0}^{\ell}\gamma_{q_{n+1-s}}=-\mu_{n-\ell}$, because the total sum vanishes. The sequence $(\mu'_\ell)$ is $(\mu_\ell)$ reversed and negated, so it has the same number of sign changes.

4. We first show that $\operatorname{sgnvar}(\mu_0,\dots,\mu_n)\le\operatorname{sgnvar}(\gamma_{q_0},\dots,\gamma_{q_n})$. Let $0=s_0<s_1<\dots<s_V$ be indices at which the signs of the $\mu_{s_r}$ alternate, with $V=\operatorname{sgnvar}(\mu)$. For $r\ge1$, the difference $\mu_{s_r}-\mu_{s_{r-1}}=\sum_{s_{r-1}<t\le s_r}\gamma_{q_t}$ has the sign of $\mu_{s_r}$. So there is $t_r\in(s_{r-1},s_r]$ such that $\gamma_{q_{t_r}}$ has the sign of $\mu_{s_r}$. Together with $t_0=0$ and $\gamma_{q_0}=\mu_0$, we obtain $V+1$ terms of $(\gamma_{q_t})$ with alternating signs. We get $U_\alpha\le V_2=\operatorname{sgnvar}(\gamma_{q_0},\dots,\gamma_{q_n})$. By~3, the same argument starting from $r_a$ gives $U_\alpha\le V_1$. Finally, by the proof of Proposition~\ref{prop:line-zeros}, where $V_1$ and $V_2$ are written with the signs $\sigma_l=\sign(\det M')\sign\gamma_l$, $L_\alpha$ is the largest integer $\le\min(V_1,V_2)$ with the parity of $l_\alpha$. By~2, $U_\alpha$ has this parity, so $U_\alpha\le L_\alpha$.
\end{proof}

\section*{Declaration of generative AI and AI-assisted technologies in the manuscript preparation process}
During the preparation of this work, the author used the AI assistant Claude Opus~5.5 (Anthropic) in order to obtain explicit formulas for the boundary degree, to write detailed versions of several technical proofs from the author's sketches, to write the program for the enumeration of Remark~\ref{rem:n2k2}, and to edit and reorganize the text. The author has been working on this problem for many years. Around 2020, it became clear that the main obstruction was to refine the Bézout bound in the positive chamber into a bound depending on signs, and the author then realized that the topological degree of the map given by the iterated Jacobians of the Gale system, restricted to the other chambers of the arrangement, could provide such a refinement. The main ideas of the paper, together with sketches of most of the proofs, were thus obtained by the author before the use of AI tools, in particular: the oriented matroid of $C$ determines the arrangement $\{p_i=0\}$ and the face structure of the positive chamber; the numbers $b(C_j)$ reduce to numbers of common zeros of Jacobians of the same type, built from new affine functions whose coefficients are maximal minors of $D$; the sign conditions on $B$ which determine the faces of $\Delta$ reached by the closure of $C_j$ (the sets $\mathcal S_j(\Delta)$); the topological degree on each chamber improves the Bézout bound in the positive chamber; and the results of \cite{BD,BDF} apply to the flats of dimension one, which provides the starting point of a recursion over all flats. The main mathematical contribution of the AI tool was to find an explicit formula for the boundary degree, the facet-count formula of Section~\ref{sec:boundary-degree}. After using this tool, the author reviewed and edited the content as needed, checked all formulas, proofs and computations, and takes full responsibility for the content of the published article.

\providecommand{\bysame}{\leavevmode\hbox to3em{\hrulefill}\thinspace}
\providecommand{\MR}{\relax\ifhmode\unskip\space\fi MR }
\providecommand{\MRhref}[2]{%
  \href{http://www.ams.org/mathscinet-getitem?mr=#1}{#2}
}
\providecommand{\href}[2]{#2}

\end{document}